\documentclass[11pt,a4paper]{article}
\usepackage[a4paper]{geometry}
\usepackage{amssymb,latexsym,amsmath,amsfonts,amsthm}
\usepackage{graphicx}
\usepackage{epstopdf}
\usepackage{float,amsmath,amssymb,mathrsfs,bm,multirow,graphics}
\usepackage{tikz}
\usepackage{pgflibraryshapes}
\usetikzlibrary{arrows,decorations.markings}
\usepackage{subfig}
\usepackage[percent]{overpic}
\usepackage{fullpage}
\usepackage{comment,verbatim}
\usepackage[hidelinks]{hyperref}
\usepackage{color}
\usepackage{mathrsfs}
\usepackage[title,titletoc]{appendix}
\usepackage{ulem}
\usepackage{enumitem}
\usepackage{mathabx}
\usetikzlibrary{arrows.meta, decorations.markings}

\DeclareMathOperator{\tr}{Tr}

\newcommand{\Bes}{\mathrm{Bes}}

\newcommand{\Ai}{\mathrm{Ai}}

\newcommand{\ud}{\,\mathrm{d}}

\newcommand{\wtil}{\widetilde}

\newcommand{\what}{\widehat}

\newcommand{\ii}{\mathrm{i}}

\newcommand{\Boh}{\mathcal{O}}

\newtheorem{theorem}{Theorem}[section]

\newtheorem{proposition}[theorem]{Proposition}

\newtheorem{corollary}[theorem]{Corollary}

\newtheorem{rhp}[theorem]{RH problem}
\newtheorem{assumption}[theorem]{Assumption}

\theoremstyle{definition}

\theoremstyle{remark}

\numberwithin{equation}{section}

\begin{document}
	\title{Edgeworth expansions of extreme eigenvalue distributions for random unitary ensembles}
	
	\author{Junwen Liu\footnotemark[1],\quad Qianlu Mao\footnotemark[2], \quad Lun Zhang\footnotemark[3]}
	
	\renewcommand{\thefootnote}{\fnsymbol{footnote}}
	\footnotetext[1]{Department of Mathematics, The University of Hong Kong, Pokfulam Road, Hong Kong. E-mail: jwliu26@hku.hk}
	\footnotetext[2]{School of Mathematical Sciences, Fudan University, Shanghai 200433, P. R. China. E-mail: qianlumao23\symbol{'100}m.fudan.edu.cn}
	\footnotetext[3]{School of Mathematical Sciences, Center for Applied Mathematics and Shanghai Key Laboratory for Contemporary Applied Mathematics, Fudan University, Shanghai 200433, P. R. China. E-mail: lunzhang\symbol{'100}fudan.edu.cn}
	
	\date{}
	
	\maketitle

	\begin{abstract}
    In this paper, we establish Edgeworth expansions of extreme eigenvalue distributions for two types random unitary ensembles at the spectral edges and reveal certain universal structures for the correction terms. More precisely, we show the correlation kernels for the Gaussian-type unitary ensembles and Laguerre-type unitary ensembles admit  full expansions at the soft edge and the hard edge, respectively. The coefficients of the correction terms are given by finite sums of Airy function or the Bessel functions of the first kind and their derivatives with polynomial coefficients. By lifting the kernel expansion to the associated Fredholm determinant, we further obtain full expansions for the largest eigenvalue distribution of the Gaussian-type unitary ensembles and the smallest eigenvalue distribution of the Laguerre-type unitary ensembles. In addition, the first correction term therein involves the derivatives of the leading term.

		% In classical probability, Edgeworth expansions give higher-order corrections to Gaussian limits. We prove edge analogues for two classes of unitary orthogonal polynomial ensembles with varying exponential weights: Gaussian-type unitary ensembles on $\mathbb{R}$ at a soft edge, and Laguerre-type ensembles on $(0,\infty)$ at a hard edge. In the soft-edge regime, the correction kernels are finite sums of Airy functions and their derivatives with polynomial coefficients in the scaled variables; consequently the largest-eigenvalue distribution has a Fredholm determinant expansion. In the hard-edge regime, the correction kernels have the analogous finite form in Bessel functions of the first kind and their derivatives, and the smallest-eigenvalue distribution admits a corresponding Fredholm determinant expansion.
	\end{abstract}
	
	\tableofcontents
	
\section{Introduction}
In the classical probability theory, Edgeworth expansions refine the 
central limit theorem by expressing higher-order corrections to the 
Gaussian limit in terms of Hermite polynomials and cumulants \cite{Wallace1958}. 
While the normal distribution is replaced by various different universal 
laws in random matrix theory \cite{AGZ10,Mehta}, it remains a fundamental 
issue to investigate the corresponding correction terms. For many applied purposes,
it is essential to understand how closely the asymptotic limit approximates the finite sample distributions \cite{Johnstone2009,Johnstone2008}. 
This paper deals with two classes of random unitary ensembles, which particularly include the classical Gaussian unitary ensemble (GUE) and the Laguerre unitary ensemble (LUE) as 
special cases. One class consists of $n \times n$ Hermitian matrices 
$M$ distributed according to the probability measure
\begin{equation}\label{eq:Gaussian}
	\frac{1}{Z_n^V}\exp(-n\operatorname{Tr} V(M))\operatorname{d}\!M,
\end{equation}
where $Z_n^V$ is the normalization constant, $V$ is a real polynomial of
even degree with positive leading coefficient, and $\mathrm{d}M$ denotes the Lebesgue 
measure on the algebraically independent entries of $M$. 
We call the ensembles in \eqref{eq:Gaussian} Gaussian-type unitary ensembles, which reduces to the GUE if $V(x)=x^2$. The other class of ensembles is the Laguerre-type unitary ensembles, which take the form
\begin{equation}\label{eq:Laguerre}
	\frac{1}{Z_n^Q} (\det M)^\alpha \exp({-n\operatorname{Tr}Q(M)})\operatorname{d}\!M, \qquad \alpha>-1,
\end{equation}
defined on the positive definite Hermitian matrices $M$ of size $n\times n$, where  $Z_n^Q$ is the corresponding normalization constant and $Q$ is a real polynomial with positive leading coefficient. If $Q(x)=x$, \eqref{eq:Laguerre} reduces to the LUE.

It is well-known, see e.g. \cite{Deift99,Mehta}, that the eigenvalue statistics of the ensembles \eqref{eq:Gaussian} and \eqref{eq:Laguerre} are determined by the correlation kernels built from polynomials orthogonal with respect to the corresponding weights. More precisely, the eigenvalue correlation kernel has the form
\begin{equation*}
		K_n(x, y) = \sqrt{w(x)w(y)} \sum_{k=0}^{n-1} p_k(x) p_k(y),
	\end{equation*}
where 
\begin{equation}\label{def:w}
    w(x)
    =\left\{
     \begin{array}{ll}
          w_V(x)=e^{-nV(x)}  , & \hbox{for the Gaussian-type unitary ensembles,} \\
         w_Q(x)=x^{\alpha}e^{-nQ(x)}\chi_{[0,+\infty)}(x), & \hbox{for the Laguerre-type unitary ensembles,}
        \end{array}
        \right.
\end{equation}
and 
$p_j(x) = \gamma_j x^j + \cdots$, $\gamma_j > 0$, $j=0,1,2,\ldots$, are polynomials of degree $j$ satisfying
\begin{equation*}
		\int_{\mathbb{R}} p_j(x) p_k(x) w(x) \operatorname{d}\!x = \delta_{jk}.
	\end{equation*}

While the limiting mean eigenvalue density 
\[
\lim_{n\to \infty}\frac{1}{n}K_{n}(x,x)
\]
is characterized by the minimizer of a weighted energy functional given in \eqref{def:energy} and \eqref{def:energy2} below (cf. \cite{Deift99,ST97}),
local limits of the correlation kernel at the spectrum edges are universal. Near the soft edge where
the limiting mean density typically vanishes like a square root, 
the limiting kernel is the Airy kernel
\begin{equation}\label{def:KAiry}
		K_{\Ai}(x,y):=\frac{\Ai(x)\Ai'(y)-\Ai'(x)\Ai(y)}{x-y}
	\end{equation}
with $\Ai$ being the Airy function; cf. \cite{BI,Deift99,DeiftGioev05,DKMVZ99,PJ1993,TWAiry}. For the 
Laguerre-type ensembles \eqref{eq:Laguerre}, the restriction on the positive definite matrices
leads to a hard edge at $0$ where the limiting mean density typically diverges like an inverse square root. Near the hard edge,
the limiting kernel is given by the Bessel kernel
\begin{equation}\label{def:LBes}
K_{\Bes}(x,y):=\frac{J_\alpha(\sqrt x)\sqrt y\,J_\alpha'(\sqrt y)-\sqrt x\,J_\alpha'(\sqrt x)J_\alpha(\sqrt y)}{2(x-y)}
\end{equation}
with $J_\alpha$ being the Bessel function of the first kind of order $\alpha$; 
cf. \cite{PJ1993,TracyWidom94Bessel,Vanlessen07}. Due to the determinantal structure of the eigenvalue distribution for the unitary invariant ensembles \eqref{eq:Gaussian} and \eqref{eq:Laguerre}, it then follows that the extreme eigenvalue distribution is given by the Airy-kernel or the Bessel-kernel determinants as $n\to \infty$. 
%, which is also known as the Tracy-Widom distribution.   

%For the classical Gaussian and Laguerre ensembles, this problem has been studied from several related perspectives. 
Finite-$n$ corrections to the correlation functions and densities were analyzed by Garoni, Forrester and Frankel 
for the GUE and LUE \cite{GaroniForresterFrankel2005}, and more recently by Byun and Lee in a different 
non-Hermitian setting \cite{ByunLee2024}. Corrections to distribution functions were initiated by Choup, 
who studied Edgeworth expansion of the largest eigenvalue distribution function for the GUE, 
LUE and  Gaussian orthogonal ensemble (GOE) \cite{Choup09,Choup08,Choup06}. Bornemann recently 
developed a general expansion theory for Gaussian and Laguerre ensembles at the soft edge, covering limit laws, 
level densities, and generating function \cite{Bornemann24b,Bornemann25b,Bornemann25c}. Regarding the expansion at 
the hard edge, Edelman, Guionnet and P\'ech\'e conjectured a particular $1/n$ correction term for the 
smallest eigenvalue distribution of LUE at the hard edge \cite{EdelmanGuionnetPeche2016}. Bornemann extracted this first finite-size correction from kernel asymptotics \cite{Bornemann16}, while Perret and Schehr independently proved the conjecture using a different approach \cite{PerretSchehr2016}; see also \cite{ForresterTrinh2019} for the extension to the Laguerre $\beta$ ensemble. Beyond this first correction term, no general full hard edge expansion appears to have been formulated in the literature. We also refer to  \cite{Bornemann24,YZ26} for the Edgeworth expansion of 
the hard-to-soft edge transition.

	% We shall use these kernel asymptotics to expand edge distribution functions, namely largest- and smallest-eigenvalue distributions. In particular, Bornemann conjectured in \cite{Bornemann24b} that the GUE soft-edge distribution admits a full asymptotic expansion to arbitrary order. The soft-edge theorem below proves the corresponding arbitrary-order kernel expansion for general varying-weight unitary ensembles and, after specialization to $V(x)=x^2$, proves the existence part of this conjecture. Hard-edge results are less developed.  Beyond this first correction term, no general full hard-edge expansion appears to have been formulated in the literature.  by Forrester and Trinh at the hard edge for the Laguerre $\beta$ ensemble \cite{ForresterTrinh2019},

	% In the determinantal setting considered here, edge gap probabilities, including the largest-eigenvalue distribution at the soft edge and the smallest-eigenvalue distribution at the hard edge, are Fredholm determinants of the corresponding restricted rescaled kernels. Once the correction kernels are identified with uniform remainder bounds, the associated edge distributions can be expanded by expanding these Fredholm determinants. Convergence-rate results in multivariate statistics give a further reason to keep finite-size terms explicit.
    Based on a operator norm convergence framework developed in \cite{TracyWidom05}, Johnstone and Ma established a convergence rate of $\mathcal{O}(n^{-2/3})$ for the largest eigenvalue distribution function to the Tracy-Widom limit in the GUE of any dimension and the GOE of even dimension \cite{JohnstoneMa2012}, whereas El Karoui \cite{ElKaroui2006} and Ma \cite{Ma2012} obtained a precise rate of convergence for the largest eigenvalue of complex white Wishart matrices. %\footnote{We use the infix notation $n\wedge p := \min(n,p)$.} for the LUE and LOE when $p/n \to \gamma \in (0,\infty)$. 
     More results on the rates of convergence for different random matrix models with different regimes can be found in \cite{Bourgade,Cai26,CaiTaljan25,SX23,SX22}. 
     
     % Recent related mesoscopic convergence-rate results for the GUE, LUE, JUE, and complex Wishart ensembles were obtained by Cai and Taljan~\cite{CaiTaljan25} and Cai~\cite{Cai26}. More details on these convergence-rate results can be found in \cite[\S3.2.2]{Taljan}.

     In view of the above results scattered in the literature, a natural question now is whether the universal limits and isolated first corrections can be replaced by full higher-order expansions in a general setting, both for the correlation kernels and for the distribution functions. Our results below give the existence of such expansions and reveal universal structures for the correction terms in the context of the Gaussian-type and Laguerre-type unitary ensembles.

	% We therefore ask whether the universal limits and isolated first corrections can be replaced by full higher-order edge expansions in a general varying-weight setting, both for the Airy and Bessel correlation kernels and for the associated Fredholm determinants. The results below give such expansions and keep the subleading terms explicit.

	% At the soft edge, after Airy scaling, we prove arbitrary-order kernel expansions in powers of $n^{-2/3}$, with correction terms given by finite linear combinations of $\Ai$ and $\Ai'$ with polynomial coefficients. After the polynomial factors are expanded, the correction operators have finite rank; this yields Fredholm determinant expansions for the largest-eigenvalue distribution. The specialization $V(x)=x^2$ proves the existence of the corresponding arbitrary-order soft-edge expansions for the GUE conjectured in \cite{Bornemann24b}. 

	% At the hard edge, after Bessel scaling, the correction terms are finite linear combinations of $J_\alpha(\sqrt{x})$ and $\sqrt{x}\,J_\alpha'(\sqrt{x})$ with coefficients polynomial in the scaled variables.

	\section{Statement of main results}
	\subsection{Edgeworth expansions for the Gaussian-type unitary  ensembles at the soft edge}
% We first consider the Unitary Invariant Ensembles (UIE) defined by the probability measure on the space of $n \times n$ Hermitian matrices $M$:
% 	 \begin{equation}
% 		     \frac{1}{Z_n} e^{-n \operatorname{Tr} V(M)} \operatorname{d}\!M,
% 		 \end{equation}
% 	   where $Z_n$ is the normalizing constant and $V(x)$ is assumed to be one-cut regular as follows. 
% 	 The induced joint probability density function for the eigenvalues $\lambda_1, \dots, \lambda_n$ is given by
% 	 \begin{equation}
% 		     \frac{1}{Z_n} \exp \left[ -n^2 \left( \frac{1}{n} \sum_{j=1}^n V(\lambda_j) - \frac{1}{n^2} \sum_{i \neq j} \log |\lambda_i - \lambda_j| \right) \right].
% 		 \end{equation}
As aforementioned, the empirical spectral measure of $M$ in the Gaussian-type unitary ensembles \eqref{eq:Gaussian} converges weakly to a probability measure $\mu_V $. It is the unique minimizer of the weighted energy functional
	 \begin{equation}\label{def:energy}
		    I_V[\mu] = \int_{\mathbb{R}} V(x) \operatorname{d}\!\mu(x) - \iint_{\mathbb{R}^2} \log |x - y| \operatorname{d}\!\mu(x) \operatorname{d}\!\mu(y),
		\end{equation}
among all the  probability measures over $\mathbb{R}$. We impose the following one-cut regularity assumptions on the potential $V$ \cite{DKM98,ST97}.
			
	\begin{assumption} \label{def:V}
		The potential $V$ in  \eqref{eq:Gaussian} satisfies the following conditions.
		\begin{itemize} 
			\item  The function $V$ is a real polynomial of even degree with positive leading coefficient.
			
			\item  The support of the equilibrium measure $\mu_V$ consists of a single bounded interval $[a, b]$.
			
			\item  The equilibrium measure is absolutely continuous with respect to Lebesgue measure, i.e., $\operatorname{d}\!\mu_V(x) = \psi_V(x) \operatorname{d}\!x$. The density $\psi_V(x)$ is strictly positive on $(a, b)$ and takes the form
			\begin{equation}\label{def:psifun}
				\psi_V(x)  = \sqrt{(b-x)(x-a)} \, h(x) \chi_{[a,b]}(x),
			\end{equation}
			where $h(x)$ is real analytic and strictly positive on $[a, b]$.
			
			\item  The Euler-Lagrange variational inequality holds strictly outside the support, i.e.,
			\begin{equation*}
				2 \int_{a}^{b} \log|x-s| \psi_V(s) \operatorname{d}\!s - V(x) < \ell_V, \quad \text{for } x \in \mathbb{R} \setminus [a, b],
			\end{equation*}
			where $\ell_V$ is a constant.
		\end{itemize}
	\end{assumption}

%\begin{remark}
		In the above assumption, the restriction to polynomial potentials is imposed to simplify the proof of the kernel estimates uniformly on scaled half-lines bounded from below. Within the class of real analytic confining potentials satisfying the one-cut regularity assumptions, the exterior estimate used below remains valid under, for instance, 
		\begin{equation*}
			\liminf_{x\to+\infty}\frac{V(x)}{x}>0.
		\end{equation*}
	%\end{remark}

Let $K_n^V$ be the correlation kernel for the eigenvalues of the Gaussian-type unitary ensembles. Our first result concerns the Edgeworth expansion of $K_n^V$ at the soft edge $b$. 
    
	\begin{theorem}\label{softthm}
		% Suppose that $V(x)$ is one-cut regular in the sense of Assumption~\ref{def:V}, and define
		% \begin{equation}\label{def:kappaV-soft}
		% 	\kappa_V:=\left(\pi\sqrt{b-a}\,h(b)\right)^{2/3}.
		% \end{equation}
    Under Assumption \ref{def:V} for the potential $V$, there exists a constant $\gamma=\gamma(V)>0$ such that, for each fixed $t_0\in \mathbb{R}$ and $m\in\mathbb N$, 
    %as $n\to \infty$ the rescaled kernel admits an expansion of the form
		\begin{align}\label{eq:softexpan}
			\widehat K_n(x,y)&:=
			\frac{1}{\kappa_V n^{2/3}}
			K_n^V \left(b+\frac{x}{\kappa_V n^{2/3}},
			b+\frac{y}{\kappa_V n^{2/3}}\right) \nonumber\\
			&= K_{\Ai}(x,y) + \sum_{j = 1}^{m} \widehat L_j(x,y)n^{-\frac{2j}{3}} + n^{-\frac{2(m+1)}{3}}\Boh(e^{-\gamma(x+y)}), \qquad n\to \infty, 
		\end{align}
		uniformly for $x,y\in [t_0,\infty)$, where 
        \begin{equation}
			\label{def:kappaV-soft}
			\kappa_V:=\left(\pi\sqrt{b-a}\,h(b)\right)^{2/3} 
		\end{equation}
		with $h$ given in \eqref{def:psifun} and $K_{\Ai}$ is the Airy kernel defined in \eqref{def:KAiry}. Moreover, each function $\widehat L_j(x,y)$ in \eqref{eq:softexpan} takes the form
		\begin{align}\label{eq:soft-correction-kernel-form}
			\widehat L_j(x,y) &=  \what a_{j,00}(x,y)\Ai(x)\Ai(y) + \what a_{j,01}(x,y)\Ai(x)\Ai'(y) \nonumber\\
			& \quad  +\what a_{j,10}(x,y)\Ai'(x)\Ai(y) +\what a_{j,11}(x,y)\Ai'(x)\Ai'(y),
		\end{align}
		where the coefficients $\what a_{j,\kappa \lambda }(x,y) \in \mathbb{R}[x,y]$  for $\kappa, \lambda \in \{0, 1\}$ and one has 
		\begin{align}\label{eq:L1-explicit}
			\widehat L_1(x,y)
			&= -\frac{d_b}{\kappa_V^2}(x^2+xy+y^2)\Ai(x)\Ai(y) \nonumber\\
			&\quad +\frac{\kappa_V+(b-a)d_b}{4(b-a)\kappa_V^2}
			\left(\Ai(x)\Ai'(y)+\Ai'(x)\Ai(y)\right)
			+\frac{d_b}{\kappa_V^2}(x+y)\Ai'(x)\Ai'(y)
		\end{align}
    with  
		\begin{equation}\label{def:db-soft}
			d_b:=\frac{\kappa_V}{5}\left(\frac{1}{b-a}+2\frac{h'(b)}{h(b)}\right).
		\end{equation}
	\end{theorem}
	
	%The correction operators in Theorem~\ref{softthm} have finite rank after the polynomial coefficients are expanded. 
	Let $\lambda_{\max}$ be the largest eigenvalue  of the Gaussian-type unitary ensembles. 
    Combining Theorem \ref{softthm} with the Fredholm determinant argument used in the proof
	of \cite[Theorem~2.1]{Bornemann24b} gives the following Edgeworth expansion for the distribution function of $\lambda_{\max}$.
	%Proposition~\ref{prop:finite-rank-det-template} records the resulting coefficient formulas.
	
	\begin{corollary}\label{soft edge lifting}
	Under Assumption \ref{def:V} for the potential $V$, we have, for each fixed $t_0\in\mathbb R$ and $m\in\mathbb N$,  
        %let $\lambda_{\max}$ denote the largest eigenvalue of the corresponding UIE with varying exponential weights, and define its cumulative distribution function by
		% \begin{equation}
		% 	\widehat F(n;x) := \mathbb{P}\left(\frac{\lambda_{\max}-b}{\kappa_V^{-1}n^{-2/3}} \leqslant x\right).
		% \end{equation}
		% Then, with the same $\gamma$ as in Theorem~\ref{softthm}, for each fixed $t_0\in\mathbb R$ and $m\in\mathbb N$,
		\begin{equation}\label{eq:soft-distribution-expansion}
			\widehat F(n; t) := \mathbb{P}\left(\frac{\lambda_{\max}-b}{\kappa_V^{-1}n^{-2/3}} \leqslant t\right) = F_{\Ai}(t) + \sum_{j=1}^{m} F_{\Ai,j}(t) n^{-\frac{2j}{3}} + n^{-\frac{2(m+1)}{3}}\Boh(e^{-2\gamma t}),
		\end{equation}
		uniformly for $t\in[t_0,\infty)$, where $\kappa_V$ is given in \eqref{def:kappaV-soft},
        \[
        F_{\Ai}(t):=\det(\mathbf I- \mathbf {K}_{\Ai})
        \]
        with $\mathbf {K}_{\Ai}$ being the integral operator acting on $L^2(t,\infty)$ associated with the Airy kernel \eqref{def:KAiry} is the cumulative function of Tracy-Widom distribution, $\gamma$ is the same constant as in \eqref{eq:softexpan} and $F_{\Ai,j}$ are smooth functions. In particular, one has 
        %The first coefficient is
		\begin{equation}\label{eq:soft-F1-explicit}
			F_{\Ai,1}(t)
			=
			\frac{d_b}{\kappa_V^2}t^2F_{\Ai}'(t)
			-\frac{\kappa_V+(b-a)d_b}{4(b-a)\kappa_V^2}F_{\Ai}''(t),
		\end{equation}
where $d_b$ is defined in \eqref{def:db-soft}. 
    \end{corollary}

	%\begin{remark}\label{rem:soft-gue}
For the GUE, i.e., $V(x)=x^2$ in \eqref{eq:Gaussian}, one has $a=-\sqrt2$, $b=\sqrt2$, $\kappa_V=\sqrt2$, $d_b=1/10$, 
which, by \eqref{eq:L1-explicit} and \eqref{eq:soft-F1-explicit}, leads to 
\begin{align*}
		\widehat L_1^{\mathrm{GUE}}(x,y)
		&=-\frac1{20}(x^2+xy+y^2)\Ai(x)\Ai(y)\\
		&\quad+\frac3{40}\left[\Ai(x)\Ai'(y)+\Ai'(x)\Ai(y)\right]
		+\frac1{20}(x+y)\Ai'(x)\Ai'(y)
	\end{align*}
    and
    \[
		F_{\Ai,1}^{\mathrm{GUE}}(t)=\frac{t^2}{20}F_{\Ai}'(t)-\frac{3}{40}F_{\Ai}''(t).
	\]
This agrees with the results in \cite{Bornemann24b} by a straightforward calculation\footnote{In \cite{Bornemann24b} the expansion is taken with respect to $n^{-2/3}/4$.}. In addition, it was conjectured therein that the GUE admits a full asymptotic expansion to arbitrary order at the soft edge, our Theorem~\ref{softthm} and Corollary~\ref{soft edge lifting} actually resolve this conjecture in a more general setting.

	It is also worth pointing out that a closely related problem is the soft edge expansion of the level density, recently studied for the classical ensembles by Bornemann~\cite{Bornemann25b} and independently by Forrester, Rahman, and Shen~\cite{ForresterRahmanShen26}. Here the density expansion follows by restricting the correlation kernel to the diagonal. Thus Theorem~\ref{softthm} gives soft edge density expansions for the Gaussian-type unitary ensembles and recovers the results for the GUE by taking $V(x)=x^2$.

	%We turn to the hard edge. 
	
	\subsection{Edgeworth expansions for the Laguerre-type unitary  ensembles at the 
	hard edge}
	For the Laguerre-type unitary ensembles \eqref{eq:Laguerre}, 
	the empirical spectral measure of $M$ converges weakly to a 
	probability measure $\mu_{Q}$, which minimizes the energy functional 
    \begin{equation}\label{def:energy2}
		I_Q[\mu]=\int_0^\infty Q(x)\,\ud\mu(x)-\iint_{[0,\infty)^2}\log|x-y|\,\ud\mu(x)\,\ud\mu(y),
	\end{equation}
    among all the probability measures over the positive real axis. Similar to Assumption \ref{def:V}, We impose the following one-cut regularity assumptions on the potential $Q$.

	% on $(0,\infty)$ with weight
	% \[
	% 	w(x)=x^\alpha e^{-nQ(x)}, \qquad \alpha>-1,
	% \]
	% it is well known that, as $n \to \infty$, the empirical spectral measure converges weakly to the unique equilibrium measure $\mu_Q \in \mathcal{M}_1([0,\infty))$, where $\mathcal{M}_1([0,\infty))$ denotes the space of Borel probability measures on $[0,\infty)$. This measure is characterized as the unique minimizer of the weighted energy functional:
	
	% We impose the following one-cut regularity assumptions on the potential $Q$ at the origin.
	\begin{assumption}\label{def:Q}
	The potential $Q$ in \eqref{eq:Laguerre} satisfies the following conditions. 
		\begin{itemize}
			\item The function $Q$ is a real polynomial with positive leading coefficient.
			\item The support of the equilibrium measure $\mu_Q$ consists of a single bounded interval $[0,\beta]$.
			\item The equilibrium measure is absolutely continuous with respect to the Lebesgue measure, i.e., $\ud\mu_Q(x)=\psi_Q(x)\,\ud x$. Moreover, 
			\begin{equation}\label{def:Q-density}
				\psi_Q(x)=\frac{1}{2\pi}\sqrt{\frac{\beta-x}{x}}\,h_Q(x)\chi_{[0,\beta]}(x),
			\end{equation}
			where $h_Q$ is real analytic and strictly positive on $[0,\beta]$.
			% Equivalently,
			% \begin{equation}\label{def:hQ}
			% 	h_Q(x)=2\pi\sqrt{\frac{x}{\beta-x}}\,\psi_Q(x), \qquad x\in (0,\beta),
			% \end{equation}
			% and in particular
			% \begin{equation}
			% 	\psi_Q(x)=\frac{h_Q(0)\sqrt\beta}{2\pi\sqrt x}\bigl(1+\mathcal{O}(x)\bigr), \quad x\searrow 0,
			% 	\qquad
			% 	\psi_Q(x)=\frac{h_Q(\beta)}{2\pi\sqrt\beta}\sqrt{\beta-x}\bigl(1+\mathcal{O}(x-\beta)\bigr), \quad x\nearrow \beta.
			% \end{equation}
			\item The Euler--Lagrange variational inequality holds strictly outside the support, i.e., 
			\begin{equation}\label{eq:Q-EL}
				2\int_0^\beta \log|x-y|\,\ud\mu_Q(y)-Q(x)<\ell_Q, \qquad x\in (\beta,+\infty),
			\end{equation}
			where $\ell_Q$ is a constant.
		\end{itemize}
	\end{assumption}

Let $K_n^Q$ be the correlation kernel for the eigenvalues of the Laguerre-type unitary 
ensembles. Although it is still possible to establish the expansion of $K_n^Q$ at the soft 
edge as in Theorem \ref{softthm}, the emphasis here is put on the expansion at the hard edge 0, 
which exhibits significantly different behaviors. 
    
	\begin{theorem}\label{hardthm}
		Under Assumption \ref{def:Q} for the potential $Q$, we have for fixed $\alpha>-1$ and
		each $m\in \mathbb{N}$,
		  \begin{align}\label{expansion-hard}
			\widetilde K_n(u,v)&:=
			\frac{1}{\kappa_Q n^2}
			K_n^Q\left(\frac{u}{\kappa_Q n^2},\frac{v}{\kappa_Q n^2}\right) \nonumber\\
			&=K_{\Bes}(u,v)+\sum_{j=1}^{m}\widetilde L_j(u,v)n^{-j}+\mathcal{O}(n^{-m-1}), 
			\qquad n\to \infty,
		\end{align}
		uniformly for $u,v\in (0,s]$ with $s>0$ fixed, where
		\begin{equation}\label{def:kappaQ-hard}
			\kappa_Q:=\beta h_Q(0)^2
		\end{equation}
		with $h_Q$ given in \eqref{def:Q-density}, $K_{\Bes}$ is the Bessel kernel defined 
		in \eqref{def:LBes}, the function $\widetilde L_j(u,v)$ 
		takes the form
		\begin{align}\label{eq:hard-correction-kernel-form}
			\widetilde L_j(u,v)&=\widetilde a_{j,00}(u,v)J_\alpha(\sqrt u)J_\alpha(\sqrt v)+\widetilde a_{j,01}(u,v)J_\alpha(\sqrt u)\sqrt v\,J_\alpha'(\sqrt v)\nonumber\\
			&\quad +\widetilde a_{j,10}(u,v)\sqrt u\,J_\alpha'(\sqrt u)J_\alpha(\sqrt v)+\widetilde a_{j,11}(u,v)\sqrt u\,J_\alpha'(\sqrt u)\sqrt v\,J_\alpha'(\sqrt v),
		\end{align}
		with $\widetilde a_{j,\kappa\lambda}(u,v)\in \mathbb{R}[u,v]$ for $\kappa,\lambda\in\{0,1\}$ and 
        \begin{equation}\label{eq:hard-L1-explicit}
			\widetilde L_1(u,v)=\frac{\alpha}{2\beta h_Q(0)}J_\alpha(\sqrt u)J_\alpha(\sqrt v).
		\end{equation} 
        Moreover, one has the following weighted uniform estimate for $u,v\in [0,s]$:
		\begin{equation}\label{expansion-hard-weighted}
			(u v)^{-\alpha/2}\left(\widetilde K_n(u,v)-K_{\Bes}(u,v)-\sum_{j=1}^{m}\widetilde L_j(u,v)n^{-j}\right)=\mathcal{O}(n^{-m-1}).
		\end{equation}
		%The first correction term is 
		
	\end{theorem}

    Let $\lambda_{\min}$ be the smallest eigenvalue of the Laguerre-type unitary ensembles. By lifting the kernel expansion to the associated 
    Fredholm determinant, one has the following Edgeworth expansion for 
    the distribution of $\lambda_{\min}$. 
    
	% At the hard edge, the proof of \cite[Theorem~2.1]{Bornemann24} converts the kernel expansion into a Fredholm determinant expansion once the correction operator is written in finite-rank form. Proposition~\ref{prop:finite-rank-det-template} records the resulting coefficient formulas.
	
	\begin{corollary}\label{hard edge lifting}
		% Under the assumptions of Theorem~\ref{hardthm}, let $\lambda_{\min}$ denote the smallest eigenvalue in the corresponding Laguerre-type ensemble, and define its rescaled hard-edge gap probability by
    Under Assumption \ref{def:Q} for the potential $Q$, we have 
		% \begin{equation}
		% 	\widetilde F(n;t):=\mathbb{P}\left(\lambda_{\min}>\frac{t}{\kappa_Q n^2}\right).
		% \end{equation}
		% Then, 
    uniformly for $t\in[0,s]$ and each $m\in\mathbb{N}$,
		\begin{equation}\label{eq:hard-distribution-expansion}
			\widetilde F(n; t) :=\mathbb{P}\left(\lambda_{\min}>\frac{t}{\kappa_Q n^2}\right) = F_{\Bes}(t) + \sum_{j=1}^{m} F_{\Bes,j}(t) n^{-j} + \mathcal{O}(n^{-m-1}),
		\end{equation}
		where 
         \[
        F_{\Bes}(t):=\det(\mathbf I- \mathbf {K}_{\Bes})
        \]
        with $\mathbf {K}_{\Bes}$ being the integral operator acting on $L^2(0,t)$ associated with the Bessel kernel \eqref{def:LBes}
        and $F_{\Bes,j}$ are smooth functions on $(0,\infty)$ with
		\begin{equation}\label{eq:hard-d1-explicit}
			F_{\Bes,1}(t)
			=
			\frac{2\alpha t}{\beta h_Q(0)}F_{\Bes}'(t).
		\end{equation}
	\end{corollary}

For the LUE, i.e., $Q(x)=x$ in \eqref{eq:Laguerre}, one has $\beta=4$ and $h_Q(0)=1$, which, by \eqref{eq:hard-L1-explicit} and \eqref{eq:hard-d1-explicit}, leads to
	\[
		\widetilde L_1^{\mathrm{LUE}}(u,v)=\frac{\alpha}{8}J_\alpha(\sqrt u)J_\alpha(\sqrt v)
	\qquad
    \mathrm{and}
    \qquad
		F_{\Bes,1}^{\mathrm{LUE}}(t)=\frac{\alpha t}{2}F_{\Bes}'(t).
	\]
This agrees with the result in \cite{Bornemann16} and the result in \cite{ForresterRahmanShen26} by re-expanding the scaling variable used therein. 

Our results also imply an optimally modified hard edge scaling for the Laguerre-type unitary ensembles, which was first observed in \cite{Bornemann16} for the LUE. More precisely, set
		\[
			c_Q:=\frac{2\alpha}{\beta h_Q(0)},\qquad
			S_n^{\mathrm{sh}}:=\kappa_Q n^2+2\alpha h_Q(0)n
			=\kappa_Q n^2\left(1+\frac{c_Q}{n}\right),
		\]
		and define
		\[
			\widetilde K_n^{\mathrm{sh}}(u,v):=
			\frac{1}{S_n^{\mathrm{sh}}}
			K_n^Q\left(\frac{u}{S_n^{\mathrm{sh}}},\frac{v}{S_n^{\mathrm{sh}}}\right).
		\]
We then obtain from Theorem \ref{hardthm} that 
for each fixed \(s>0\),  
		\begin{equation}\label{eq:optKn}
		    (uv)^{-\alpha/2}\left(\widetilde K_n^{\mathrm{sh}}(u,v)-K_{\Bes}(u,v)\right)
			=\mathcal{O}(n^{-2}),
		\end{equation}
uniformly for \(u,v\in[0,s]\). Indeed, note that
\[
			\widetilde K_n^{\mathrm{sh}}(u,v)
			=
			\left(1+\frac{c_Q}{n}\right)^{-1}
			\widetilde K_n\left(\frac{u}{1+c_Q/n},\frac{v}{1+c_Q/n}\right).
		\]
		Using \eqref{expansion-hard} and the Taylor expansion, it is readily seen that
		\begin{equation}\label{eq:TaytildeKn}
		    	\widetilde K_n^{\mathrm{sh}}(u,v)
			=K_{\Bes}(u,v)
			+\frac1n\left\{\widetilde L_1(u,v)
			-c_Q(u\partial_u+v\partial_v+1)K_{\Bes}(u,v)\right\}
			+\mathcal{O}(n^{-2}).
		\end{equation}
        From the differential equation for the Bessel functions, one has  
		\[
			(u\partial_u+v\partial_v+1)K_{\Bes}(u,v)
			=\frac14 J_\alpha(\sqrt u)J_\alpha(\sqrt v).
		\]
        This, together with \eqref{eq:hard-L1-explicit} and \eqref{eq:TaytildeKn}, implies \eqref{eq:optKn}. Similarly, let 
		\[
			\widetilde F^{\mathrm{sh}}(n;t):=
			\mathbb P\left(\lambda_{\min}>\frac{t}{S_n^{\mathrm{sh}}}\right),
		\]
		then, for every fixed \(s>0\), uniformly for \(t\in[0,s]\),
		\[
			\widetilde F^{\mathrm{sh}}(n;t)
			=F_{\Bes}(t)+\mathcal{O}(n^{-2}).
		\]
		This follows from the fact that
		\(\widetilde F^{\mathrm{sh}}(n;t)=\widetilde F(n;t/(1+c_Q/n))\) and Corollary \ref{hard edge lifting}.

\paragraph{About the proof and organization of the paper}
By the Christoffel-Darboux formula, one has 
\begin{equation*}
		K_n(x, y) = \sqrt{w(x)w(y)} 
        \frac{\gamma_{n-1}}{\gamma_{n}} \frac{p_{n}(x)p_{n-1}(y) - p_{n-1}(x)p_{n}(y)}{x - y}.
	\end{equation*}
It is then natural to establish the large $n$ expansion of the correlation kernel by using the asymptotic expansions of the associated orthogonal polynomials. The difficulty of using this idea lies in checking the divisibility of a certain sequence of polynomials to arrive at a full expansion. One probably needs some hidden symmetry of the coefficients in the expansion, which is not even known for the GUE and the LUE. The starting point of our analysis, however, is a Riemann-Hilbert (RH) characterization of $K_n$; see \eqref{def:K} below. While the leading Airy and Bessel limits follow from the Deift--Zhou steepest descent method \cite{DZ93}, through the works of \cite{DKMVZ99,KMVV04,Vanlessen07}, we are able to refine the RH approach therein to show the existence of a full Edgeworth expansion for the correlation kernel at the spectrum edges and the universal structures for the correction terms. More precisely, we observe a  Wronskian identity for general second-order ODEs (see Proposition~\ref{lemma} below), which allows us to remove the singularity in the Christoffel--Darboux quotient. At both edges, the correction kernels are finite sums of products of Airy or Bessel functions with polynomial coefficients. As a consequence, the corresponding correction operators have a finite rank, and the Fredholm determinant expansions reduce to finite-dimensional determinant calculations, yielding explicit formulas for the correction terms in the extreme-eigenvalue distributions. We believe the methodology developed here can be adapted to explore a variety of similar problems arising from random matrix theory and beyond.

The rest of this paper is organized as follows. 
% For orthogonal polynomial ensembles with varying exponential weights, the results below refine the classical RH derivation of edge universality. The leading Airy and Bessel limits follow from the Deift--Zhou steepest descent method, through the works of Deift and collaborators, Kuijlaars and collaborators, and Vanlessen \cite{DKMVZ99,KMVV04,Vanlessen07}; see also \cite{Deift99}. The correction terms considered here require a termwise expansion of the local parametrices and the algebraic cancellations that remove the apparent singularity in the Christoffel--Darboux quotient.
Section \ref{sec:prel} collects the determinant expansions and Wronskian identities used later. %Subsection~\ref{sec:proofs-det-expansions} proves the passage from kernel expansions to Fredholm determinant expansions and gives the first-correction coefficient formulas. Subsection~\ref{sec:proof-lemma} proves the Wronskian identity for second-order ODEs and records its Airy specialization. 
Section \ref{sec:asyanal} gives sketch of asymptotic analysis of RH problems for the orthogonal polynomials associated with the Guassian-type weight and Laguerre-type weight, respectively. The proofs of our main results are presented in Section \ref{sec:proof}, as the outcomes of our asymptotic analysis.

\paragraph{Notation} We use the following notation throughout this paper.
	\begin{itemize}
		\item We denote by $D(z_0, r)$ the open disc centred at $z_0$ with radius $r > 0$, i.e.,
		\begin{equation}\label{def:dz0r}
			D(z_0, r) := \{ z\in \mathbb{C} \mid |z-z_0|<r \},
		\end{equation}
		and by $\partial D(z_0, r)$ its boundary. When $\partial D(z_0, r)$ is used as part of an RH contour, it is oriented clockwise.
		\item For functions $\phi,\psi$ on an interval $I$, we write $\phi\otimes\psi$ for the rank-one kernel
		\begin{equation}\label{eq:tensor-kernel-notation}
			(\phi\otimes\psi)(x,y):=\phi(x)\psi(y),
			\qquad x,y\in I.
		\end{equation}
		
		\item As usual, the three Pauli matrices $\{\sigma_j\}_{j=1}^3$ are defined by
		\begin{equation}\label{def:Pauli}
			\sigma_1=\begin{pmatrix}
				0 & 1 \\
				1 & 0
			\end{pmatrix},
			\qquad
			\sigma_2=\begin{pmatrix}
				0 & -\ii \\
				\ii & 0
			\end{pmatrix},
			\qquad
			\sigma_3=
			\begin{pmatrix}
				1 & 0 \\
				0 & -1
			\end{pmatrix}.
		\end{equation}
	\end{itemize}
	
	\section{Preliminaries}\label{sec:prel}
    \subsection{Finite-rank corrections and determinant expansions}\label{sec:proofs-det-expansions}
    Following the spirits of \cite{Bornemann24b,Bornemann24}, we have the following proposition which records the passage from kernel expansions to Fredholm determinant expansions with explicit formulas for the coefficients.
    
 %    which allows us to pass from kerne expansion to  
	% The next proposition records the finite-dimensional determinant formula used below. The finite-rank determinant identity from \cite[Appendix~A]{Bornemann24b} reduces the truncated correction operator to a finite-dimensional determinant, from which the coefficient formulas follow; see also \cite[Section~2.2]{Bornemann24}.

	\begin{proposition}\label{prop:finite-rank-det-template}
		Let $I_t\subset\mathbb R$ be a Borel set and $\mathbf K_*$ be a trace-class operator on $L^2(I_t)$ so that $\mathbf I-\mathbf K_*$ is invertible. Fix a positive integer $m$, let $\mathbf K_1,\dots,\mathbf K_m$ be bounded operators on $L^2(I_t)$ and  $\mathbf K_{(\eta)}$ be an $\eta$-dependent trace-class operator on $L^2(I_t)$.Suppose that the Fredholm determinant reduction
		\begin{equation}\label{eq:generic-det-reduction}
			\det(\mathbf I-\mathbf K_{(\eta)})
			=F_*(t)D_{\eta}(t)+\mathcal{O}(\eta^{m+1}),
		\end{equation}
		holds, where
		
        \begin{equation}\label{eq:generic-Eeta}
			D_{\eta}(t):=\det(\mathbf I-\mathbf E_{\eta}),
			\qquad
			F_*(t):=\det(\mathbf I-\mathbf K_*),
		\end{equation}
        with
        \begin{equation}
			\mathbf E_{\eta}:=(\mathbf I-\mathbf K_*)^{-1}(\eta \mathbf K_1+\cdots+\eta^m \mathbf K_m).
		\end{equation}
		Assume further that the truncated correction operator admits the finite-rank decomposition
		\begin{equation}\label{eq:generic-finite-rank-decomposition}
			\eta \mathbf K_1+\cdots+\eta^m \mathbf K_m
			=\sum_{\mu=1}^{r_m}\eta^{q_\mu}c_\mu\,\phi_\mu\otimes\psi_\mu,
			\qquad 1\le q_\mu\le m,
		\end{equation}
		where $c_\mu\in\mathbb C$ and $\phi_\mu,\psi_\mu\in L^2(I_t)$. Here \(r_m\) counts the rank-one terms, while \(q_\mu\) records their orders.
		Then $D_{\eta}(t)$ admits the finite-dimensional representation
		\begin{equation}\label{eq:generic-small-matrix}
			D_{\eta}(t)=\det_{\mu,\nu=1}^{r_m}\left(\delta_{\mu\nu}-\eta^{q_\mu}c_\mu \,u_{\mu\nu}(t)\right),
			\qquad
			u_{\mu\nu}(t):=\left\langle(\mathbf I-\mathbf K_*)^{-1}\phi_\mu,\psi_\nu\right\rangle_{L^2(I_t)}.
		\end{equation}
		Writing
		\begin{equation*}
			\mathbf E_{\eta}=\eta \mathbf E_1+\cdots+\eta^m \mathbf E_m,
			\qquad \mathbf E_j=(\mathbf I-\mathbf K_*)^{-1}\mathbf K_j,
		\end{equation*}
		and expanding
		\begin{equation*}
			D_{\eta}(t)=1+\sum_{j\ge 1} d_j(t)\eta^j,
		\end{equation*}
		we obtain, for $1\le \ell\le m$, the weighted minor formula
		\begin{equation}\label{eq:generic-dm}
			d_\ell(t)=\sum_{k=1}^{\ell}(-1)^k
			\sum_{\substack{1\le p_1<\cdots<p_k\le r_m\\ q_{p_1}+\cdots+q_{p_k}=\ell}}
			\left(\prod_{j=1}^{k} c_{p_j}\right)
			\det_{a,b=1}^{k} u_{p_a p_b}(t).
		\end{equation}
		For $\ell=1$, this gives
		\begin{equation}\label{eq:generic-d1}
			d_1(t)=-\sum_{q_\mu=1}c_\mu u_{\mu\mu}(t)=-\tr\mathbf E_1.
		\end{equation}
	\end{proposition}

	\begin{proof}
		The determinant reduction \eqref{eq:generic-det-reduction} is part of the hypothesis. By \eqref{eq:generic-finite-rank-decomposition},
		\begin{equation*}
			\mathbf E_\eta
			=
			(\mathbf I-\mathbf K_*)^{-1}(\eta \mathbf K_1+\cdots+\eta^m \mathbf K_m)
			=
			\sum_{\mu=1}^{r_m}\eta^{q_\mu}c_\mu\,\bigl((\mathbf I-\mathbf K_*)^{-1}\phi_\mu\bigr)\otimes\psi_\mu.
		\end{equation*}
		Define
		\begin{equation*}
			A:L^2(I_t)\to \mathbb C^{r_m},
			\qquad
			Af:=\bigl(\langle f,\psi_\nu\rangle_{L^2(I_t)}\bigr)_{\nu=1}^{r_m},
		\end{equation*}
		and
		\begin{equation*}
			B_\eta:\mathbb C^{r_m}\to L^2(I_t),
			\qquad
			B_\eta e_\mu:=\eta^{q_\mu}c_\mu\,(\mathbf I-\mathbf K_*)^{-1}\phi_\mu.
		\end{equation*}
		It is then readily seen that $\mathbf E_\eta=B_\eta A$. Since $\mathbf E_\eta$ has a finite rank, the finite-rank determinant identity from \cite[Appendix~A]{Bornemann24b} gives
		\begin{equation*}
			D_\eta(t)=\det(\mathbf I-\mathbf E_\eta)=\det(I_{r_m}-AB_\eta).
		\end{equation*}
		The $(\nu,\mu)$-th entry of $AB_\eta$ is
		\begin{equation*}
			\eta^{q_\mu}c_\mu\,
			\left\langle(\mathbf I-\mathbf K_*)^{-1}\phi_\mu,\psi_\nu\right\rangle_{L^2(I_t)}
			=
			\eta^{q_\mu}c_\mu\,u_{\mu\nu}(t).
		\end{equation*}
		Hence $AB_\eta$ is the transpose of the matrix appearing in \eqref{eq:generic-small-matrix}, and transposition does not change the determinant. This proves \eqref{eq:generic-small-matrix}.
		
		Now let
		\begin{equation*}
			M_\eta(t):=\bigl(\eta^{q_\mu}c_\mu\,u_{\mu\nu}(t)\bigr)_{\mu,\nu=1}^{r_m},
		\end{equation*}
		so that $D_\eta(t)=\det(I-M_\eta(t))$. Expanding in principal minors gives
		\begin{equation*}
			D_\eta(t)
			=
			1+\sum_{k=1}^{r_m}(-1)^k
			\sum_{1\le p_1<\cdots<p_k\le r_m}
			\det\bigl[M_\eta(t)\bigr]_{p_1,\dots,p_k}.
		\end{equation*}
		For a fixed set $\{p_1,\dots,p_k\}$, one may factor $\eta^{q_{p_a}}c_{p_a}$ from the $a$th row, so
		\begin{equation*}
			\det\bigl[M_\eta(t)\bigr]_{p_1,\dots,p_k}
			=
			\eta^{q_{p_1}+\cdots+q_{p_k}}
			\left(\prod_{j=1}^k c_{p_j}\right)
			\det_{a,b=1}^k u_{p_ap_b}(t).
		\end{equation*}
		Collecting the terms of total weight $\ell$ yields \eqref{eq:generic-dm}, which leads to the first equality in  \eqref{eq:generic-d1} by choosing $\ell=1$. The trace identity follows from $\mathbf E_1=(\mathbf I-\mathbf K_*)^{-1}\mathbf K_1$ and the finite-rank representation above.

        This completes the proof of Proposition \ref{prop:finite-rank-det-template}.
	\end{proof}
	
	%\begin{remark}
		Proposition~\ref{prop:finite-rank-det-template} is stated on the Lebesgue spaces $L^2(I_t)$ only for notational convenience. The same proof applies verbatim on $L^2(I_t,\mu)$ for any finite measure $\mu$, with the same finite-dimensional determinant formula and coefficient formulas. We will take $\mathbf K_*=\mathbf K_{\Ai}$, $I_t=(t,\infty)$ and $\eta=n^{-2/3}$ in the proof of 
		Corollary~\ref{soft edge lifting}.
		In view of Theorem~\ref{softthm}, the vectors $\phi_\mu,\psi_\mu$ may be chosen among finite linear combinations of $x^r\Ai(x)$ and $x^s\Ai'(x)$. For the proof of Corollary~\ref{hard edge lifting}, one takes $\mathbf K_*=\mathbf K_{\Bes}$, $I_t=(0,t)$, and $\eta=n^{-1}$. In view of Theorem~\ref{hardthm}, the vectors $\phi_\mu,\psi_\mu$ may be chosen among finite linear combinations of $x^rJ_{\alpha}(\sqrt{x})$ and $x^s\sqrt{x}\,J_{\alpha}'(\sqrt{x})$.

	\subsection{A Wronskian identity for the second order ODEs}\label{sec:proof-lemma}

    Our next proposition reveals a conservation law related to the derivatives of solutions of second order ODEs. It plays a key role in the proof of divisibility arising from the Christoffel--Darboux quotient.

	\begin{proposition}\label{lemma}
		Let $y$ be a solution of the second order ODE
		\begin{equation}\label{def:y}
			y''(x) = p(x)y'(x) + q(x)y(x),
		\end{equation}
		such that $y$ and $y'$ are linearly independent over $\mathbb{C}(x)$. 
		For each $n \ge 0$, there exist functions $P_n$ and $Q_n$ such that
		\begin{equation*}
			y^{(n)}(x) = P_n(x) y'(x) + Q_n(x) y(x).
		\end{equation*}
		We define the generating functions by
		\begin{equation*}
			P(z; w) := \sum_{n=0}^{\infty} P_n(z) \frac{w^n}{n!} \quad \text{and} \quad Q(z; w) := \sum_{n=0}^{\infty} Q_n(z) \frac{w^n}{n!},
		\end{equation*}
		If $l(x)$ is an antiderivative satisfying $l'(x) = -p(x)$, then we have the conservation law
		\begin{equation}\label{conservation law}
			\frac{\partial}{\partial w} \left[ e^{l(z+w)} \left( P \frac{\partial Q}{\partial w} - Q \frac{\partial P}{\partial w} \right) \right] = 0.
		\end{equation}
	\end{proposition}

	\begin{proof}
	Let $Y(w) = y(z+w)$. It follows from Taylor's expansion that
		\begin{equation*}
			Y(w) = y'(z) P(z;w) + y(z) Q(z;w).
		\end{equation*}
		Substituting this representation into the shifted equation
		\begin{equation*}
			\partial_w^2 Y = p(z+w)\partial_w Y + q(z+w)Y
		\end{equation*}
		yields
		\begin{align*}
			y'(z) \partial_w^2 P(z;w) + y(z) \partial_w^2 Q(z;w) 
			&= p(z+w) \bigl[ y'(z) \partial_w P(z;w) + y(z) \partial_w Q(z;w) \bigr] \nonumber \\
			&\quad + q(z+w) \bigl[ y'(z) P(z;w) + y(z) Q(z;w) \bigr],
		\end{align*}
        or equivallently, 
		%Collecting the coefficients of $y'(z)$ and $y(z)$ gives
		\begin{align*}
			&y'(z) \left[ \partial_w^2 P(z;w) - p(z+w) \partial_w P(z;w) - q(z+w) P(z;w) \right] \nonumber \\
			&+ y(z) \left[ \partial_w^2 Q(z;w) - p(z+w) \partial_w Q(z;w) - q(z+w) Q(z;w) \right] = 0.
		\end{align*}
		The linear independence of $y$ and $y'$ implies that both $P$ and $Q$ satisfy the differential equation
		\begin{equation*}
			\partial_w^2 \Phi = p(z+w) \partial_w \Phi + q(z+w) \Phi.
		\end{equation*}
		
		Let $\mathcal{W} = P \partial_w Q - Q \partial_w P$. Differentiating $\mathcal{W}$ with respect to $w$ and using the above differential equation gives us 
		\begin{equation*}
			\partial_w \mathcal{W} = P \partial_w^2 Q - Q \partial_w^2 P = p(z+w)(P \partial_w Q - Q \partial_w P) = p(z+w)\mathcal{W}.
		\end{equation*}
	This, together with the fact that $l' = -p$, implies that 
		\begin{equation*}
			\partial_w \left( e^{l(z+w)} \mathcal{W} \right) = e^{l(z+w)} \bigl(l'(z+w) \mathcal{W} + \partial_w \mathcal{W}\bigr) = e^{l(z+w)} \bigl(-p(z+w)\mathcal{W} + p(z+w)\mathcal{W}\bigr) = 0,
		\end{equation*}
        as required.

    This completes the proof of Proposition~\ref{lemma}.
	\end{proof}
	
	%\begin{remark}\label{rem:riccati}
		The solution $y$ and its derivative $y'$ are linearly dependent over the field of rational functions $\mathbb{C}(x)$ if and only if the associated Riccati equation $w' + w^2 - p(x)w - q(x) = 0$ admits a rational solution $w \in \mathbb{C}(x)$. Regarding the stronger property of algebraic independence, Bornemann~\cite{Bornemann25a} recently proved that, for the Airy equation as a special case of the second-order ODEs considered here, a nontrivial solution $u$, its derivative $u'$, and an antiderivative are algebraically independent over $\mathbb{C}(x)$. 
	%\end{remark}
	
	%We record the Airy specialization used later.
	
	\paragraph{Example}
	The Airy function $\mathrm{Ai}$ satisfies
	\begin{equation}
		\mathrm{Ai}''(z) = z\mathrm{Ai}(z).
		\label{airy_ode}
	\end{equation}
	For each integer $m \geqslant 0$, there exist polynomials $\what P_m(z)$ and $\what Q_m(z)$ such that
	\begin{equation}\label{Airy_function}
		\mathrm{Ai}^{(m)}(z) = \what P_m(z)\mathrm{Ai}(z) + \what Q_m(z)\mathrm{Ai}'(z).
	\end{equation}
	Applying Proposition~\ref{lemma} with $p(z)=0$ (so that $l(z)=0$), the conservation law \eqref{conservation law} reduces to
	\begin{equation}\label{eq:Airyconservation}
		\frac{\partial}{\partial w}\left( P\frac{\partial Q}{\partial w} - Q\frac{\partial P}{\partial w} \right) = 0.
	\end{equation}
	Expanding the above equality in powers of $w$, we obtain
	\begin{equation*}
		P\frac{\partial Q}{\partial w} - Q\frac{\partial P}{\partial w}
		= \sum_{N = 0}^{\infty} \left[ \sum_{k = 0}^{N} \frac{1}{k!(N-k)!}
		\left( \what P_k(z) \what Q_{N-k+1}(z) - \what P_{k+1}(z) \what Q_{N-k}(z) \right) \right] w^N.
	\end{equation*}
	Since the left-hand side is independent of $w$, the coefficients of $w^N$ vanish for all $N \geqslant 1$, which gives us
	\begin{equation}\label{airy law}
		\sum_{k = 0}^{N} \frac{1}{k!(N-k)!}
		\left( \what P_k(z) \what Q_{N-k+1}(z) - \what P_{k+1}(z) \what Q_{N-k}(z) \right) = 0.
	\end{equation}
This equality was first proved in \cite{YZ26} by a different method. 
	
	% \begin{remark}
	% 	In the soft-edge analysis, Proposition~\ref{lemma} is used through the Airy identity \eqref{airy law}, which gives the diagonal divisibility needed for the finite-rank representation of the correction kernels. At the hard edge, we use the same conservation law in matrix form, because the variables are $\sqrt u$ and $\sqrt v$ rather than $u$ and $v$.
	% \end{remark}

	\section{Asymptotic analysis of the RH problem for orthogonal polynomials}\label{sec:asyanal}
	The RH problem for orthogonal polynomials was introduced by Fokas, Its, and Kitaev \cite{FIK92}. Given a non-negative weight function on an interval $J \subseteq \mathbb{R}$ and a positive integer $n$, it asks for a $2\times 2$ matrix-valued function $Y(z)$ satisfying
	
	\begin{rhp}\label{rhp:Y}
		\hfill
		\begin{itemize}
			\item [\rm (a)] $Y(z)$ is analytic in $\mathbb{C} \setminus J$.
			\item [\rm (b)]  $Y_+(x) = Y_-(x) \begin{pmatrix} 1 & w(x) \\ 0 & 1 \end{pmatrix}$ for $x \in J$,
			\item [\rm (c)] $Y(z) = (I + \mathcal{O}(z^{-1})) \begin{pmatrix} z^n & 0 \\ 0 & z^{-n} \end{pmatrix}$ as $z \to \infty$.
			\item [\rm (d)] $Y(z)$ satisfies the endpoint conditions dictated by the weight $w$.
		\end{itemize}
	\end{rhp}
	
	The solution of RH problem~\ref{rhp:Y} is given by 
	\begin{align*}
		Y(z) = \begin{pmatrix}
			\gamma_n^{-1} p_n(z) &
			\frac{1}{2\pi \mathrm{i} \gamma_n} \displaystyle\int_{J} \frac{p_n(s) w(s)}{s - z}\, \mathrm{d}s \\
			-2\pi \mathrm{i}\, \gamma_{n-1} p_{n-1}(z)&
			-\gamma_{n-1} \displaystyle\int_{J} \frac{p_{n-1}(s) w(s)}{s - z}\, \mathrm{d}s
		\end{pmatrix}.
	\end{align*}
	and the correlation kernel $K_n(x,y)$ is then related to $Y$ through 
	\begin{equation}\label{def:K}
		K_n(x,y) = \frac{1}{2\pi i(x-y)} \sqrt{w(x)}\sqrt{w(y)}\begin{pmatrix} 0 & 1 \end{pmatrix} Y_+(y)^{-1} Y_+(x) \begin{pmatrix} 1 \\ 0 \end{pmatrix}. 
	\end{equation}

We will perform Deift-Zhou nonlinear steepest descent analysis \cite{Deift99,DZ93} on the above RH problem for the weight function in \eqref{def:w}, which consists of a series of explicit transformations. The transformations needed here are now canonical, so we go over them briefly and only record the key steps and formulas for the convenience of the reader but refer to \cite{Deift99,DKMVZ99,Vanlessen07} for details.   
    
	% We use the Deift--Zhou steepest descent method in its standard form: a $g$-function normalization, a lens opening, local and global parametrices, and a final small-norm problem.

	\subsection{Asymptotic analysis of RH problem \ref{rhp:Y} with Gaussian-type weight}
	For the Gaussian-type unitary ensembles \eqref{eq:Gaussian}, the associated weight function is given by $w(x) = e^{-nV(x)}$ on $\mathbb{R}$. Under Assumption~\ref{def:V} for the potential $V$, We give a sketch of the analysis below, following \cite{Deift99,DKMVZ99}.

	\subsubsection{First transformation: $Y \to \what T$}
	Let $\psi_V(x) \operatorname{d}\!x$ be the equilibrium measure associated with $V$, and define the associated $g$-function by
	\begin{equation}\label{defgx}
		\what g(z):=\int_a^b \log(z-s)\psi_V(s)\,\ud s,
		\qquad z\in\mathbb{C}\setminus(-\infty,b],
	\end{equation}
	where the logarithm is taken on its principal branch. Standard properties
	of the equilibrium measure give the following relations.
	\begin{proposition}\label{prop:gV}
		The function $\what g$ defined in \eqref{defgx} is analytic in
		$\mathbb{C}\setminus(-\infty,b]$ and satisfies
		\begin{align}
			\what g_+(x)+\what g_-(x)-V(x)-\ell_V&=0,
			&&x\in[a,b],\label{gV-1}\\
			2\what g(x)-V(x)-\ell_V&<0,
			&&x\in(b,+\infty),\label{gV-2}\\
			\what g_+(x)-\what g_-(x)&=2\pi \ii,
			&&x\in(-\infty,a),\label{gV-3}\\
			\what g_+(x)-\what g_-(x)
			&=2\pi \ii\int_x^b\psi_V(s)\,\ud s,
			&&x\in[a,b].\label{gV-4}
		\end{align}
	\end{proposition}
	We also set
	\begin{align}
		&\varphi(z)=\pi \int_b^z \left((s-b)(s-a)\right)^{\frac{1}{2}} h(s) \ud s, \quad z \in \mathbb{C} \setminus(-\infty, b], \label{def:phi} \\
		&\tilde{\varphi}(z)=\pi \int_a^z \left((s-b)(s-a)\right)^{\frac{1}{2}} h(s) \ud s, \quad z \in \mathbb{C} \setminus [a,+\infty). \label{def:tilde-phi}
	\end{align}
	The square roots in \eqref{def:phi} and \eqref{def:tilde-phi} are chosen positive on $(b,+\infty)$ and negative on $(-\infty,a)$, respectively.
	Proposition~\ref{prop:gV} gives
	\begin{align}
		\varphi_+(x)-\varphi_-(x)&=-2\pi \ii,
		&&x\in(-\infty,a),\label{phiV-1}\\
		2\varphi_+(x)=-2\varphi_-(x)
		&=-\bigl(\what g_+(x)-\what g_-(x)\bigr),
		&&x\in(a,b),\label{phiV-2}\\
		2\varphi(x)&=V(x)+\ell_V-2\what g(x),
		&&x\in[b,+\infty),\label{phiV-3}\\
		2\tilde{\varphi}(x)
		&=V(x)+\ell_V-\what g_+(x)-\what g_-(x),
		&&x\in(-\infty,a].\label{tildephiV}
	\end{align}

	Define the first transformation by
	\begin{equation}\label{def:hatT}
		\what T(z)=e^{-\frac{1}{2}n\ell_V\sigma_3}\,Y(z)\,e^{-n \what g(z)\sigma_3}\,e^{\frac{1}{2}n\ell_V\sigma_3},
	\end{equation}
	The transformed matrix $\what T$ then satisfies the following RH problem.
	
	\begin{rhp}\label{rhphatT}
		\hfill
		\begin{itemize}
			
			\item[\rm (a)]
			$\what T(z)$ is defined and analytic in $\mathbb{C}\setminus\mathbb{R}$.
			
			\item[\rm (b)]
			$\what T$ satisfies ${\what T}_+(z) = {\what T}_-(z) J_{\what T}(z),$ for $z \in \mathbb{R}$, where
			\begin{equation}\label{def:JT}
				J_{\what T}(z)=
				\begin{cases}
					\begin{pmatrix}
						e^{2n\varphi_+(x)} & 1 \\
						0 & e^{2n\varphi_-(z)}
					\end{pmatrix},
					& z \in [a,b], \\
					
					\begin{pmatrix}
						1 & e^{-2n\varphi(z)} \\
						0 & 1
					\end{pmatrix},
					& z \in (b,+\infty), \\
					
					\begin{pmatrix}
						1 & e^{-2n\tilde{\varphi}(z)} \\
						0 & 1
					\end{pmatrix},
					& z \in (-\infty,a).
				\end{cases}
			\end{equation}
			
			\item[\rm (c)]
			As $z\to\infty$, $\what T(z)=I+\mathcal{O}\!\left(\frac{1}{z}\right).$
		\end{itemize}
	\end{rhp}

	\subsubsection{Second transformation: $\what T\mapsto {\what S}$}\label{rhp.P1}
	This transformation involves lens-opening, which is based on the factorization
	\begin{align*}
		\begin{pmatrix}
			e^{2n\varphi_{+}(x)} & 1 \\
			0 & e^{2n\varphi_{-}(x)}
		\end{pmatrix}
		=
		\begin{pmatrix}
			1 & 0 \\
			e^{2n\varphi_{-}(x)} & 1
		\end{pmatrix}
		\begin{pmatrix}
			0 & 1 \\
			-1 & 0
		\end{pmatrix}
		\begin{pmatrix}
			1 & 0 \\
			e^{2n\varphi_{+}(x)} & 1
		\end{pmatrix}
	\end{align*}
	of the jump matrix $J_{\what T}$ on $(a, b)$. %This factorization gives a contour deformation for which the jump matrices tend to the identity as $n\to\infty$.
	
	Set
	\begin{align}\label{def:hatS}
		{\what S}= \what T
		\begin{cases}
			\begin{pmatrix}
				1 & 0 \\
				-e^{2n\varphi} & 1
			\end{pmatrix}, & \hbox{$z\in \Omega_{\widehat{R},+}$,}\\
			\begin{pmatrix}
				1 & 0 \\
				e^{2n\varphi} & 1
			\end{pmatrix}, & \hbox{$z\in \Omega_{\widehat{R},-}$,} \\
			I, & \hbox{elsewhere,}
		\end{cases}
	\end{align}
    where the regions $\Omega_{\widehat{R},\pm}$ are shown in Figure \ref{fig:lenses}. 
	We the following  RH problem for $\what S$.
	
	\begin{rhp}\label{rhS}
		\hfill
		\begin{itemize}
			
			\item[\rm (a)] 
			$\what S(z)$ is defined and analytic in $\mathbb{C}\setminus \Gamma_{\what S}$, where
			\begin{equation}\label{jumpS}
				\Gamma_{\what S} = \mathbb{R} \cup \partial \Omega_{\what R, \pm};
			\end{equation}
			see Figure \ref{fig:lenses} for an illustration.
			\item[\rm (b)] 
			${\what S}$ satisfies the jump condition
			${\what S}_+(z) = {\what S}_-(z) J_{\what S}(z)$ for $z\in \Gamma_{\what S}$, where
			\begin{equation}\label{def:JS}
				J_{\what S}(z)
				=
				\begin{cases}
					\begin{pmatrix}
						0 & 1 \\
						-1 & 0
					\end{pmatrix},
					& z\in[a,b], \\
					
					\begin{pmatrix}
						1 & 0 \\
						e^{2n\varphi(z)} & 1
					\end{pmatrix},
					& z\in \partial \Omega_{\what R, \pm}, \\
					
					\begin{pmatrix}
						1 & e^{-2n\varphi(z)} \\
						0 & 1
					\end{pmatrix},
					& z\in(b,+\infty), \\
					
					\begin{pmatrix}
						1 & e^{-2n\tilde{\varphi}(z)} \\
						0 & 1
					\end{pmatrix},
					& z\in(-\infty,a).
				\end{cases}
			\end{equation}
			
			\item[\rm (c)] 
			As $z\to\infty$, ${\what S}(z)=I+\mathcal{O}\left(\frac{1}{z}\right).$
		\end{itemize}
	\end{rhp}

	\begin{figure}[t]
		\centering
		\begin{tikzpicture}[x=1.2pt,y=1.2pt,yscale=-1,xscale=1, thin]
			
			\draw    (158,118.85) -- (474.11,117.86) ;
			
			\draw [fill={rgb, 255:red, 0; green, 0; blue, 0 } ,fill opacity=1 ] (200.13,118.98) -- (194.9,121.7) -- (195.03,116.05) -- cycle ;
			
			\draw [fill={rgb, 255:red, 0; green, 0; blue, 0 } ,fill opacity=1 ] (421.2,118) -- (415.97,120.71) -- (416.1,115.06) -- cycle ;
			
			\draw   (259.24,117.86) .. controls (259.24,103.9) and (281.76,92.58) .. (309.55,92.58) .. controls (337.34,92.58) and (359.87,103.9) .. (359.87,117.86) .. controls (359.87,131.82) and (337.34,143.14) .. (309.55,143.14) .. controls (281.76,143.14) and (259.24,131.82) .. (259.24,117.86) -- cycle ;
			
			\draw [fill={rgb, 255:red, 0; green, 0; blue, 0 } ,fill opacity=1 ] (311.7,143.64) -- (306.47,146.35) -- (306.6,140.7) -- cycle ;
			\draw [fill={rgb, 255:red, 0; green, 0; blue, 0 } ,fill opacity=1 ] (311.7,92.35) -- (306.47,95.07) -- (306.6,89.42) -- cycle ;
			
			\draw (251.15,129.58) node [anchor=north west][inner sep=0.75pt]   [align=left] {$a$};
			\draw (357.56,128.59) node [anchor=north west][inner sep=0.75pt]   [align=left] {$b$};
			
			\draw (301,100) node [anchor=north west][inner sep=0.75pt]   [align=left] {$\Omega_{\widehat{R},+}$};
			\draw (300,125) node [anchor=north west][inner sep=0.75pt]   [align=left] {$\Omega_{\widehat{R},-}$};
			
			\draw (301,70) node [anchor=north west][inner sep=0.75pt]   [align=left] {$\partial \Omega_{\widehat{R},+}$};
			\draw (301,150) node [anchor=north west][inner sep=0.75pt]   [align=left] {$\partial \Omega_{\widehat{R},-}$};
			
		\end{tikzpicture}
		\caption{Regions $\Omega_{\widehat{R}, \pm}$ and the jump contours of the RH problem for $\what S$.}
		\label{fig:lenses}
	\end{figure}
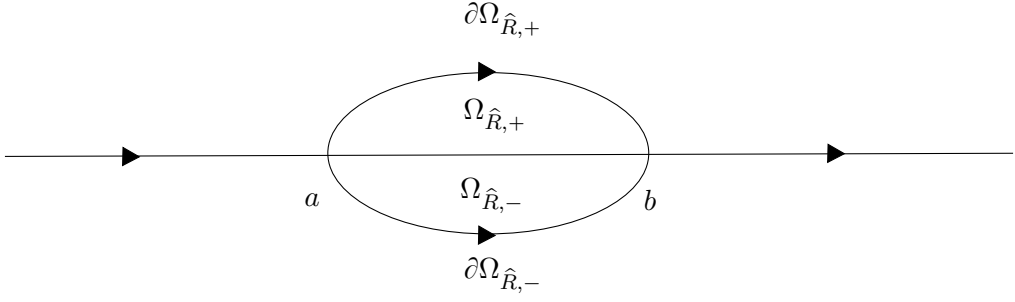

	\subsubsection{Global parametrix}
	For $z$ bounded away from $(a,b)$, it follows from Proposition \ref{prop:gV} that the jump matrices $J_{\what S}$ in \eqref{def:JS} converge exponentially fast to the identity matrix as $n \to \infty$. The global parametrix $\what N$ retains only the jump on $(a,b)$ and is characterized by the following RH problem.
	\begin{rhp}\label{rhN}
		\hfill
		\begin{itemize}
			
			\item[\rm (a)] 
			$\what N (z)$ is defined and analytic in $\mathbb{C} \setminus [a,b]$.
			
			\item[\rm (b)] 
			$\what N$ satisfies $
			\what N_{+}(z) = \what N_{-}(z)
			\begin{pmatrix}
				0 & 1 \\
				-1 & 0
			\end{pmatrix}$ for $z \in [a,b]$.
			
			\item[\rm (c)] 
			As $z \to \infty$, $\what N(z)= I+\mathcal{O} \left(\frac{1}{z}\right).$
			
		\end{itemize}
	\end{rhp}

	This RH problem can be solved explicitly; see \cite{Deift99}. Its solution is given by
	\begin{equation} \label{def:N}
		\what N(z) =  \begin{pmatrix}
			\dfrac{1}{2}\left(\what\gamma(z) + \what\gamma^{-1}(z)\right) & \dfrac{1}{2\mathrm{i}}\left(\what\gamma(z) - \what\gamma^{-1}(z)\right) \\
			- \dfrac{1}{2\mathrm{i}}\left(\what\gamma(z) - \what\gamma^{-1}(z)\right) & \dfrac{1}{2}\left(\what\gamma(z) + \what\gamma^{-1}(z)\right)
		\end{pmatrix},
	\end{equation}
	where $\what\gamma (z)=\left(\dfrac{z-b}{z-a} \right)^{\frac{1}{4}}$, $z\in \mathbb{C}\setminus (a,b)$.
	
	\subsubsection{Local parametrices}
	In a sufficiently small neighborhood of the endpoint $b$, we introduce a local parametrix for $\what S$, characterized by the following RH problem.
	\begin{rhp}\label{rhP-b}
		\hfill
		\begin{itemize}
			
			\item[\rm (a)] 
			$\what P^{(b)}(z)$ is defined and analytic in $D(b, \varepsilon) \setminus \Gamma_{\what S}$, where $D\left(b, \varepsilon\right)$ and $\Gamma_{\what S}$  are defined in \eqref{def:dz0r} and \eqref{jumpS},	respectively.

			\item[\rm (b)] 
			For $z \in D(b, \varepsilon) \cap \Gamma_{\what S}$, we have
			\begin{equation}\label{eq:P0-jump}
				\what P^{(b)}_+(z) = \what P^{(b)}_-(z) J_{\what S}(z),
			\end{equation}
			where $J_{\what S}(z)$ is defined in \eqref{def:JS}.

			\item[\rm (c)] 
			As $n \to \infty,$ $\what P^{(b)}(z)$ satisfies the matching condition
			\begin{equation}\label{eq:Pb-match}
				\what P^{(b)}(z)=\left(I+\mathcal{O} \left(n^{-1}\right) \right) \what N(z), \qquad z \in \partial D(b, \varepsilon),
			\end{equation}
			where $\what N(z)$ is given in \eqref{def:N}.
			
		\end{itemize}
	\end{rhp}
	
	This local RH problem is solved in terms of the Airy parametrix $\Phi^{(\operatorname{Ai})}(z)$ from Appendix \ref{airy}. To this end, we define
	\begin{equation}\label{def:f}
		\what f(z)=\left[\frac{3}{2}\varphi (z)\right]^{\frac{2}{3}}
		=\kappa_V\left(z-b\right)+d_b\left(z-b\right)^2+\Boh((z-b)^3), \quad z \to b,
	\end{equation}
	where $\varphi(z)$ is defined in \eqref{def:phi},  $\kappa_V$ and $d_b$ are defined in \eqref{def:kappaV-soft} and \eqref{def:db-soft}, respectively. The function $\what f(z)$ is a conformal map in $D(b, \varepsilon)$.

	Set
	\begin{equation}\label{def:Pb}
		\what P^{(b)}(z) = \what E_n^{(b)}(z) \Phi^{(\Ai)}\left(n^{\frac{2}{3}}\what f(z)\right) e^{n\varphi(z) \sigma_3},
	\end{equation} 
	where
	\begin{equation}\label{def:Enb}
		\what E_n^{(b)}(z)=\frac{1}{\sqrt{2} } \what N(z)
		\begin{pmatrix}
			1 & -\mathrm{i} \\
			-\mathrm{i} & 1
		\end{pmatrix} 
		\left(n^{\frac{2}{3}}\what f(z)\right)^{\frac{1}{4}\sigma _3 }.
	\end{equation}
	
	\begin{proposition}\label{pro:Pb}
		The function $\what P^{(b)}(z)$ defined in \eqref{def:Pb} solves the RH problem \ref{rhP-b}.
	\end{proposition}
	\begin{proof}
		From \eqref{def:N} and \eqref{def:Enb},
		\begin{equation}\label{def:E}
			\what E_n^{(b)}(z) = -\frac{1}{\sqrt{2}}\begin{pmatrix}
				1 & -\mathrm{i} \\
				-\mathrm{i} & 1
			\end{pmatrix} \left(n^{\frac{2}{3}}(z-a)\frac{\what f(z)}{z-b}\right)^{\frac{1}{4}\sigma _3},
		\end{equation}
		Since $\what f(z)$ has a simple zero at $z = b$, \eqref{def:E} implies that $\what E_n^{(b)}(z)$ is analytic in $D(b,\varepsilon)$.
		
		The jump relation \eqref{eq:P0-jump} follows from the analyticity of $\what E_n^{(b)}(z)$ together with the jump condition \eqref{jump:Airy} for the Airy parametrix. The matching condition \eqref{eq:Pb-match} follows from the large-$z$ asymptotics \eqref{infty:Ai} of the Airy parametrix.
	\end{proof}

	% \subsubsection{Local parametrix near $z=a$}
	% The local parametrix near $z=a$ satisfies the analogous RH problem.
	
	% \begin{rhp}\label{rhP-a}
	% 	\hfill
	% 	\begin{itemize}
			
	% 		\item[\rm (a)] 
	% 		$\what P^{(a)}(z)$ is defined and analytic in $D(a, \varepsilon) \setminus \Gamma_{\what S}$, where $D\left(a, \varepsilon\right)$ and $\Gamma_{\what S}$  are defined in \eqref{def:dz0r} and \eqref{jumpS},	respectively.

	% 		\item[\rm (b)] 
	% 		For $z \in D(a, \varepsilon) \cap \Gamma_{\what S}$, we have
	% 		\begin{equation*}
	% 			\what P^{(a)}_+(z) = \what P^{(a)}_-(z) J_{\what S}(z),
	% 		\end{equation*}
	% 		where $J_{\what S}(z)$ is defined in \eqref{def:JS}.
			
	% 		\item[\rm (c)] 
	% 		As $n \to \infty,$ $\what P^{(a)}(z)$ satisfies the matching condition
	% 		\begin{equation}\label{eq:Pa-match}
	% 			\what P^{(a)}(z)=\left(I+\mathcal{O} \left(n^{-1}\right) \right) \what N(z), z \in \partial D(a, \varepsilon),
	% 		\end{equation}
	% 		where $\what N(z)$ is given in \eqref{def:N}.
			
	% 	\end{itemize}
	% \end{rhp}

   In a small neighborhood  $D(a, \varepsilon)$ 
   of $z=a$, one has an RH problem for $\what P^{(a)}$ similar to RH problem \ref{rhP-b}. 
	It can be solved from $\what P^{(b)}(z)$ by interchanging $a$ and $b$ and conjugating by $\sigma_3$ defined in \eqref{def:Pauli}. We oimt the details here. 
	
	\subsubsection{Final transformation}
	Set
	\begin{equation}\label{def:hatR}
		\what R(z) = \begin{cases}
			\what S(z) \what P^{(a)}(z)^{-1}, & \quad z \in D(a, \varepsilon),\\
			\what S(z) \what P^{(b)}(z)^{-1}, & \quad z \in D(b, \varepsilon),\\
			\what S(z) \what N(z)^{-1}, & \quad \textrm{elsewhere}.
		\end{cases}
	\end{equation}
	From the RH problems for $\what S$, $\what N$, $\what P^{(a)}$, and $\what P^{(b)}$, both $a$ and $b$ are removable singularities of $\what R$. Thus $\what R$ satisfies the following RH problem.
	
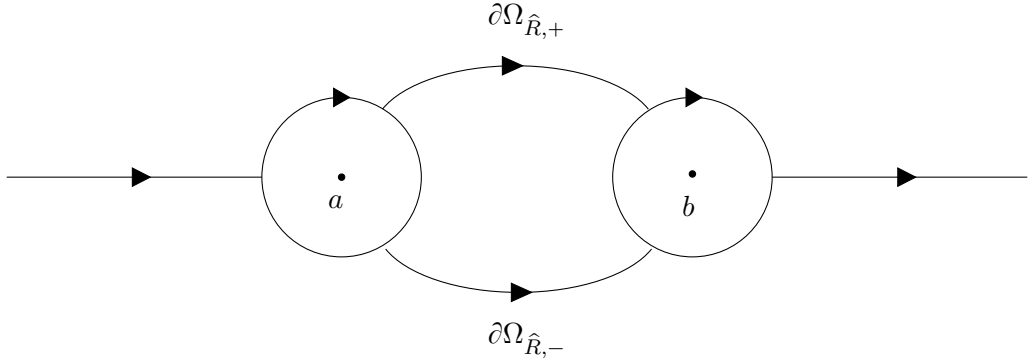
\begin{figure}[htbp]
	\centering
	\begin{tikzpicture}[x=1.2pt,y=1.2pt,yscale=-1,xscale=1, thin]
		
		\draw (190,145) .. controls (190,131.19) and (201.19,120) .. (215,120) .. controls (228.81,120) and (240,131.19) .. (240,145) .. controls (240,158.81) and (228.81,170) .. (215,170) .. controls (201.19,170) and (190,158.81) .. (190,145) -- cycle ;
		\draw [fill={rgb, 255:red, 0; green, 0; blue, 0 } ,fill opacity=1 ] (217.5,120) -- (212.5,122.87) -- (212.5,117.13) -- cycle ;
		\draw [fill={rgb, 255:red, 0; green, 0; blue, 0 } ,fill opacity=1 ] (214,145) .. controls (214,144.45) and (214.45,144) .. (215,144) .. controls (215.55,144) and (216,144.45) .. (216,145) .. controls (216,145.55) and (215.55,146) .. (215,146) .. controls (214.45,146) and (214,145.55) .. (214,145) -- cycle ;
		
		\draw (300,145) .. controls (300,131.19) and (311.19,120) .. (325,120) .. controls (338.81,120) and (350,131.19) .. (350,145) .. controls (350,158.81) and (338.81,170) .. (325,170) .. controls (311.19,170) and (300,158.81) .. (300,145) -- cycle ;
		\draw [fill={rgb, 255:red, 0; green, 0; blue, 0 } ,fill opacity=1 ] (328,120) -- (323,122.87) -- (323,117.13) -- cycle ;
		\draw [fill={rgb, 255:red, 0; green, 0; blue, 0 } ,fill opacity=1 ] (324,144) .. controls (324,143.45) and (324.45,143) .. (325,143) .. controls (325.55,143) and (326,143.45) .. (326,144) .. controls (326,144.55) and (325.55,145) .. (325,145) .. controls (324.45,145) and (324,144.55) .. (324,144) -- cycle ;
		
		\draw (110,145) -- (190,145) ;
		\draw [fill={rgb, 255:red, 0; green, 0; blue, 0 } ,fill opacity=1 ] (155,145) -- (149.39,147.87) -- (149.39,142.13) -- cycle ;
		
		\draw (350,145) -- (430,145) ;
		\draw [fill={rgb, 255:red, 0; green, 0; blue, 0 } ,fill opacity=1 ] (395,145) -- (389.39,147.87) -- (389.39,142.13) -- cycle ;
		
		\draw (227.81,123.69) .. controls (234.13,115.74) and (250.41,110.08) .. (269.5,110.08) .. controls (288.59,110.08) and (304.87,115.74) .. (311.19,123.69) ;
		\draw [fill={rgb, 255:red, 0; green, 0; blue, 0 } ,fill opacity=1 ] (271.77,110.08) -- (265.41,112.95) -- (265.41,107.21) -- cycle ;
		
		\draw (312.19,167.52) .. controls (305.87,175.48) and (289.59,181.13) .. (270.5,181.13) .. controls (251.41,181.13) and (235.13,175.48) .. (228.81,167.52) ;
		\draw [fill={rgb, 255:red, 0; green, 0; blue, 0 } ,fill opacity=1 ] (274.59,181.13) -- (268.23,184) -- (268.23,178.27) -- cycle ;
		
		\draw (210,150) node [anchor=north west][inner sep=0.75pt] [align=left] {$a$};
		\draw (321,150) node [anchor=north west][inner sep=0.75pt] [align=left] {$b$};
		
		\draw (260,90) node [anchor=north west][inner sep=0.75pt] [align=left] {$\partial \Omega_{\what R,+}$};
		\draw (260,190) node [anchor=north west][inner sep=0.75pt] [align=left] {$\partial \Omega_{\what R,-}$};
		
	\end{tikzpicture}
	\caption{The jump contours $\Gamma_{\what R}$ of the RH problem for $\what R$.}
	\label{fig:contour_R}
\end{figure}
	
	\begin{rhp}\label{rhp:R}
		\hfill
		\begin{itemize}
			\item [\rm{(a)}] $\what R(z)$ is defined and analytic in $\mathbb{C} \setminus \Gamma_{\what R}$; see Figure \ref{fig:contour_R} for an illustration of $\Gamma_{\what R}$.
			\item [\rm{(b)}] For $z \in \Gamma_{\what R}$, we have
			\begin{equation}\label{eq:Rjump}
				\what R_+(z) = \what R_-(z) J_{\what R}(z),
			\end{equation}
			where the jump matrix $J_{\what R}(z)$ is given by
			\begin{equation*}
				J_{\what R}(z) = \begin{cases}
					\what P^{(a)}(z)\what N^{-1}(z), & \quad z \in \partial D(a, \varepsilon),\\
					\what P^{(b)}(z)\what N^{-1}(z), & \quad z \in \partial D(b, \varepsilon),\\
					\what N(z) J_{\what S(z)} \what N^{-1}(z), & \quad z \in \Gamma_{\what R} \setminus \left(\partial D(a, \varepsilon) \cup \partial D(b, \varepsilon)\right).
				\end{cases}
			\end{equation*}
			\item [\rm{(c)}] As $z \to \infty$, $\what R(z) = I + \mathcal{O} \left(\frac{1}{z} \right).$
		\end{itemize}
	\end{rhp}
	
	We recall the asymptotic expansion of $\what P^{(b)}(z)\what N(z)^{-1}$ on $\partial D(b, \varepsilon)$ derived in \cite{KT09}:
	\begin{equation}\label{Pbasy}
		\what P^{(b)}(z) \what N(z)^{-1} \sim I + \sum_{k=1}^{\infty} \frac{1}{n^k} \Delta_k(z),
	\end{equation}
	where for even $k$,
	\begin{align}\label{evenk}
		\Delta_k(z) &= \frac{1}{\sqrt{\pi}}\left(\frac{\Gamma\left(3 k+\frac{1}{2}\right)}{9^k(2 k)!}-\frac{\Gamma\left(3 k-\frac{3}{2}\right)}{4 \cdot 9^{k-1}(2(k-1))!}\right) \frac{1}{\left(\frac{3}{2} \varphi(z)\right)^k} I \notag \\
		&\quad -\frac{1}{4 \sqrt{\pi}} \frac{\Gamma\left(3 k-\frac{3}{2}\right)}{9^{k-1}(2(k-1))!} \frac{1}{\left(\frac{3}{2} \varphi(z)\right)^k} \sigma_2,
	\end{align}
	and for odd $k$,
	\begin{align}\label{oddk}
		\Delta_k(z) &= -\frac{\what\gamma(z)^2}{\left(\frac{3}{2} \varphi(z)\right)^k} \frac{1}{2 \sqrt{\pi}}\left(\frac{\Gamma\left(3 k+\frac{1}{2}\right)}{9^k(2 k)!}-\frac{\Gamma\left(3 k-\frac{3}{2}\right)}{2 \cdot 9^{k-1}(2(k-1))!}\right)\left(\sigma_3+\mathrm{i} \sigma_1\right) \notag \\
		&\quad -\frac{\what\gamma(z)^{-2}}{\left(\frac{3}{2} \varphi(z)\right)^k} \frac{1}{2 \sqrt{\pi}} \frac{\Gamma\left(3 k+\frac{1}{2}\right)}{9^k(2 k)!}\left(\sigma_3-\mathrm{i} \sigma_1\right).
	\end{align}
	Here $\Gamma(\cdot)$ denotes the Gamma function, and $\sigma_k$, $k=1,2,3$, are the Pauli matrices defined in \eqref{def:Pauli}.
	
	A similar expansion
	\begin{equation*}
		\what P^{(a)}(z) \what N(z)^{-1} \sim I + \sum_{k=1}^{\infty} \frac{1}{n^k} \widetilde \Delta_k(z),
	\end{equation*}
	holds for $\partial D(a, \varepsilon)$.
	
	The jump matrices $J_{\what R}(z) = \what N(z)S(z)\what N(z)^{-1}$ tend to the identity matrix at an exponential rate as $n \to \infty$. The small-norm analysis of \cite{Deift99,DZ93} gives
	\begin{equation}\label{asy:R}
		\what R(z) \sim I+ \sum_{k=1}^{\infty} \frac{1}{n^k} \what R^{(k)}(z),
	\end{equation}
	where $\what R^{(k)}$ satisfies an additive RH problem. 
	\begin{rhp}\label{rhp:R_k}
		\hfill
		\begin{itemize}
			\item [\rm{(a)}] $\what R^{(k)}(z)$ is analytic on $\mathbb{C} \setminus \left(\partial D(a, \varepsilon) \cup \partial D(b, \varepsilon)\right)$.
			
			\item [\rm{(b)}] For $z \in \partial D(a, \varepsilon) \cup \partial D(b, \varepsilon)$, we have
			\begin{equation*}
				\what R^{(k)}_{+}(z) = \begin{cases}
					\what R^{(k)}_{-}(z) + \displaystyle\sum_{l=0}^{k-1} \what R^{(l)}_{+}(z) \widetilde \Delta_{k-l}(z), & \quad z \in \partial D(a, \varepsilon), \\
					\what R^{(k)}_{-}(z) + \displaystyle\sum_{l=0}^{k-1} \what R^{(l)}_{+}(z) \Delta_{k-l}(z), & \quad z \in \partial D(b, \varepsilon).
				\end{cases}
			\end{equation*}
			\item [\rm{(c)}] As $z \to \infty$,  $\what R^{(k)}(z) = \mathcal{O}\left(\frac{1}{z}\right).$
		\end{itemize}
	\end{rhp}
	Following \cite[Lemma~4.1]{KT09}, the structural properties of $\Delta_k(z)$ given in \eqref{evenk} and \eqref{oddk}, together with those of $\widetilde \Delta_k(z)$, determine the algebraic form of $\what R^{(k)}(z)$. Specifically, $\what R^{(k)}(z)$ is a linear combination of $\sigma_1$ and $\sigma_3$ when $k$ is odd, and a linear combination of the identity matrix $I$ and $\sigma_2$ when $k$ is even. Hence we may write $\what R^{(k)}(z)$ as
	\begin{align}\label{def:Rk}
		\what R^{(k)}(z) =
		\begin{cases}
			\what \alpha_{1,k}(z) (\sigma_3+\mathrm{i}\sigma_1) + \what \alpha_{2,k}(z) (\sigma_3-\mathrm{i}\sigma_1), & \text{if } k \text{ is odd}, \\[1ex]
			\what \alpha_{1,k}(z) (I+\sigma_2) + \what \alpha_{2,k}(z) (I-\sigma_2), & \text{if } k \text{ is even}. 
		\end{cases}
	\end{align}

	\subsection{Asymptotic analysis of RH problem \ref{rhp:Y} with Laguerre-type weight}\label{subsec:hard-rh-analysis}
	For the Laguerre-type ensembles \eqref{eq:Laguerre}, the associated weight function reads $w(x)=x^{\alpha}e^{-nQ(x)}$, $\alpha>-1$ on the positive real axis. Following 
    \cite{Vanlessen07,KMVV04},  the principal new feature here, as compared with Gaussian-type weight, is the singularity at the origin, which leads to a Bessel parametrix instead of the Airy one.

 %    with varying exponential weights
	% \begin{equation*}
	% 	w(x)=x^{\alpha}e^{-nQ(x)}, \qquad x\in (0,+\infty), \qquad \alpha>-1,
	% \end{equation*}
	% where the external potential $Q$ is one-cut regular at the hard edge in the sense of Assumption~\ref{def:Q}. We follow the Deift--Zhou steepest descent method in the form developed for Laguerre-type weights in \cite{Vanlessen07,ZCD14}. 
	
	%For the weight $w(x)=x^{\alpha}e^{-nQ(x)}$,
    
    To proceed, we also note that RH problem \ref{rhp:Y} in the present case is subject to the following local condition at the origin:
	\begin{equation}\label{hard-origin-Y}
		Y(z)=
		\begin{cases}
			\mathcal{O}\!\begin{pmatrix}
				1 & \lvert z\rvert^{\alpha}\\
				1 & \lvert z\rvert^{\alpha}
			\end{pmatrix}, & -1<\alpha<0,\\[2ex]
			\mathcal{O}\!\begin{pmatrix}
				1 & \log (1/\lvert z\rvert)\\
				1 & \log (1/\lvert z\rvert)
			\end{pmatrix}, & \alpha=0,\\[2ex]
			\mathcal{O}\!\begin{pmatrix}
				1 & 1\\
				1 & 1
			\end{pmatrix}, & \alpha>0,
		\end{cases}
		\qquad z\to 0, \quad z\in \mathbb{C}\setminus [0,+\infty).
	\end{equation}
	
	\subsubsection{First transformation: $Y \to \widetilde T$}
	% As in the soft-edge analysis, the first transformation is built from the equilibrium measure associated with the potential $Q$.
	
	Similar to \eqref{defgx}, we define the associated $g$-function
	\begin{equation}\label{def:gQ}
		g_Q(z):=\int_0^\beta \log(z-y)\psi_Q(y)\,\ud y, \qquad z\in \mathbb{C}\setminus (-\infty,\beta],
	\end{equation}
	where the logarithm is taken on its principal branch and $\psi_Q(x) \operatorname{d}\!x$ is the equilibrium measure of energy functional 
   \eqref{def:energy2}. Standard properties of the equilibrium measure imply the following relations.
	
	\begin{proposition}\label{prop:gQ}
		The function $g_Q$ defined in \eqref{def:gQ} is analytic in $\mathbb{C}\setminus (-\infty,\beta]$ and satisfies
		\begin{align}
			g_{Q,+}(x)+g_{Q,-}(x)-Q(x)-\ell_Q &=0, && x\in [0,\beta],\label{gQ-1}\\
			2g_Q(x)-Q(x)-\ell_Q &<0, && x\in (\beta,+\infty),\label{gQ-2}\\
			g_{Q,+}(x)-g_{Q,-}(x) &=2\pi \ii, && x\in (-\infty,0),\label{gQ-3}\\
			g_{Q,+}(x)-g_{Q,-}(x) &=2\pi \ii\int_x^\beta \psi_Q(y)\,\ud y, && x\in [0,\beta].\label{gQ-4}
		\end{align}
	\end{proposition}
	
	For the subsequent RH analysis, we also define
	\begin{equation}\label{def:xiQ}
		\widetilde\xi(z):=-\pi \ii\int_\beta^z \psi_Q(y)\,\ud y, \qquad z\in \mathbb{C}\setminus (-\infty,\beta].
	\end{equation}
	% where the path of integration does not cross $(-\infty,\beta]$. 
	Proposition~\ref{prop:gQ} gives
	\begin{align}
		\widetilde\xi_+(x)-\widetilde\xi_-(x)&=2\pi \ii, &&x\in (-\infty,0),\label{xiQ-1}\\
		2\widetilde\xi_+(x)=-2\widetilde\xi_-(x)&=g_{Q,+}(x)-g_{Q,-}(x), &&x\in (0,\beta),\label{xiQ-2}\\
		2\widetilde\xi(x)&=2g_Q(x)-Q(x)-\ell_Q, &&x\in [\beta,+\infty).\label{xiQ-3}
	\end{align}
	
	Set
	\begin{equation}\label{def:tildeT}
		\widetilde T(z)=e^{-\frac12 n\ell_Q\sigma_3}Y(z)e^{-ng_Q(z)\sigma_3}e^{\frac12 n\ell_Q\sigma_3},
	\end{equation}
	which satisfies the following RH problem.
	
	\begin{rhp}\label{rhpt1}
		\hfill
		\begin{itemize}
			\item[\rm(a)] $\widetilde T$ is analytic in $\mathbb{C}\setminus [0,+\infty)$.
			\item[\rm(b)] For $x\in (0,+\infty)$, we have $\widetilde T_+(x)=\widetilde T_-(x)J_{\widetilde T}(x)$, where
			\begin{equation}\label{def:JThat}
				J_{\widetilde T}(x)=
				\begin{cases}
					\begin{pmatrix}
						e^{-2n\widetilde\xi_+(x)} & x^{\alpha}\\
						0 & e^{-2n\widetilde\xi_-(x)}
					\end{pmatrix}, & x\in (0,\beta),\\[3ex]
					\begin{pmatrix}
						1 & x^{\alpha}e^{2n\widetilde\xi(x)}\\
						0 & 1
					\end{pmatrix}, & x\in (\beta,+\infty).
				\end{cases}
			\end{equation}
			\item[\rm(c)] As $z\to\infty$, $\widetilde T(z)=I+\mathcal{O}(z^{-1})$.
			% \item[\rm(d)] As $z\to 0$ with $z\in \mathbb{C}\setminus [0,+\infty)$, $\widetilde T$ satisfies the same local condition as in \eqref{hard-origin-Y}.
		\end{itemize}
	\end{rhp}
	
	\subsubsection{Second transformation: $\widetilde T \mapsto \widetilde S$}\label{rhp.S1}
    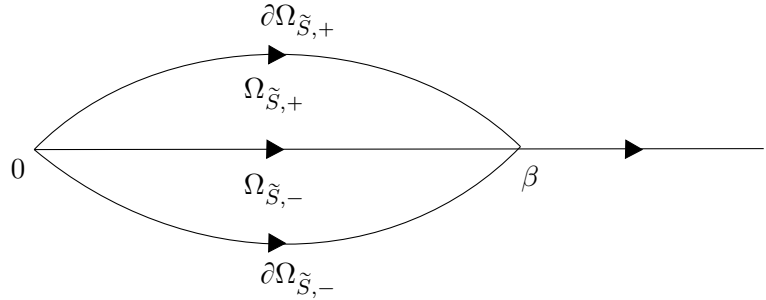
\begin{figure}[htbp]
	\centering
	\begin{tikzpicture}[x=1.2pt,y=1.2pt,yscale=-1,xscale=1, thin]
		\draw (130,81.25) -- (358.75,80.95);
		
		\draw (130,81.25) .. controls (168.75,41.575) and (242.5,41.575) .. (282.5,80.325);
		\draw (130,81.25) .. controls (177.5,120.95) and (243.125,120.95) .. (282.5,80.325);
		
		\draw [fill={rgb, 255:red, 0; green, 0; blue, 0 } ,fill opacity=1 ] (208.75,51.625) -- (203.52,54.34) -- (203.65,48.69) -- cycle ;
		\draw [fill={rgb, 255:red, 0; green, 0; blue, 0 } ,fill opacity=1 ] (208.75,110.625) -- (203.52,113.34) -- (203.65,107.69) -- cycle ;
		\draw [fill={rgb, 255:red, 0; green, 0; blue, 0 } ,fill opacity=1 ] (208.125,81.25) -- (202.895,83.965) -- (203.025,78.315) -- cycle ;
		\draw [fill={rgb, 255:red, 0; green, 0; blue, 0 } ,fill opacity=1 ] (320.625,81.25) -- (315.395,83.965) -- (315.525,78.315) -- cycle ;
		
		\draw (200,35) node [anchor=north west][inner sep=0.75pt] [align=left] {$\partial \Omega_{\widetilde S,+}$};
		\draw (200,116) node [anchor=north west][inner sep=0.75pt] [align=left] {$\partial \Omega_{\widetilde S,-}$};
		
		\draw (195,58) node [anchor=north west][inner sep=0.75pt] [align=left] {$\Omega_{\widetilde S,+}$};
		\draw (195,88) node [anchor=north west][inner sep=0.75pt] [align=left] {$\Omega_{\widetilde S,-}$};
		
		\draw (121.89,83.61) node [anchor=north west][inner sep=0.75pt] [align=left] {$0$};
		\draw (281.89,85.48) node [anchor=north west][inner sep=0.75pt] [align=left] {$\beta$};
	\end{tikzpicture}
	\caption{The jump contours of the RH problem for $\widetilde S$.}
	\label{figure-Bessel-lens}
\end{figure}
	On the interval $(0,\beta)$, the jump matrix admits the factorization
	\begin{align}\label{jump2}
		\begin{pmatrix}
			e^{-2n\widetilde\xi_{+}(x)} & x^\alpha \\
			0 & e^{-2n\widetilde\xi_{-}(x)}
		\end{pmatrix}
		=
		\begin{pmatrix}
			1 & 0 \\
			x^{-\alpha}e^{-2n\widetilde\xi_{-}(x)}  & 1
		\end{pmatrix}
		\begin{pmatrix}
			0 & x^\alpha \\
			-x^{-\alpha} & 0
		\end{pmatrix}
		\begin{pmatrix}
			1 & 0 \\
			x^{-\alpha}e^{-2n\widetilde\xi_{+}(x)} & 1
		\end{pmatrix}.
	\end{align}
	Open a lens around $(0,\beta)$. Let $\Omega_{\widetilde S,+}$ and $\Omega_{\widetilde S,-}$ denote the upper and lower lens-shaped regions bounded by $(0,\beta)$ and the contours $\Sigma_1$ and $\Sigma_3$, respectively, and set
	\begin{equation}\label{def:tildeS}
		\widetilde S(z)=\widetilde T(z)
		\begin{cases}
			\begin{pmatrix}
				1 & 0\\
				-z^{-\alpha}e^{-2n\widetilde\xi(z)} & 1
			\end{pmatrix}, & z\in \Omega_{\widetilde S,+},\\[3ex]
			\begin{pmatrix}
				1 & 0\\
				z^{-\alpha}e^{-2n\widetilde\xi(z)} & 1
			\end{pmatrix}, & z\in \Omega_{\widetilde S,-},\\[3ex]
			I, & \text{elsewhere}.
		\end{cases}
	\end{equation}
	The matrix $\widetilde S$ solves the following RH problem.
	
	\begin{rhp}\label{rhpt2}
		\hfill
		\begin{itemize}
			\item[\rm(a)] $\widetilde S$ is analytic in $\mathbb{C}\setminus \Gamma_{\widetilde S}$, where
			\begin{equation*}
				\Gamma_{\widetilde S}=[0,+\infty)\cup \Sigma_1\cup \Sigma_3;
			\end{equation*}
			see Figure \ref{figure-Bessel-lens} below.
			\item[\rm(b)] For $z\in \Gamma_{\widetilde S}$, we have $\widetilde S_+(z)=\widetilde S_-(z)J_{\widetilde S}(z)$ with
			\begin{equation}\label{def:JShat}
				J_{\widetilde S}(z)=
				\begin{cases}
					\begin{pmatrix}
						1 & 0\\
						z^{-\alpha}e^{-2n\widetilde\xi(z)} & 1
					\end{pmatrix}, & z\in \Sigma_1\cup \Sigma_3,\\[3ex]
					\begin{pmatrix}
						1 & z^{\alpha}e^{2n\widetilde\xi(z)}\\
						0 & 1
					\end{pmatrix}, & z\in (\beta,+\infty),\\[3ex]
					\begin{pmatrix}
						0 & z^{\alpha}\\
						-z^{-\alpha} & 0
					\end{pmatrix}, & z\in (0,\beta).
				\end{cases}
			\end{equation}
			\item[\rm(c)] As $z\to\infty$, $\widetilde S(z)=I+\mathcal{O}(z^{-1})$.
		\end{itemize}
	\end{rhp}

	\subsubsection{Global parametrix}
	By Proposition~\ref{prop:gQ}, all jump matrices of $\widetilde S$ outside $(0,\beta)$ converge exponentially fast to the identity matrix as $n\to\infty$. The global parametrix $\widetilde N$ retains only the jump on $(0,\beta)$ and is characterized by the following RH problem.
	
	\begin{rhp}\label{rhpforn}
		\hfill
		\begin{itemize}
			\item[\rm(a)] $\widetilde N$ is analytic in $\mathbb{C}\setminus [0,\beta]$.
			\item[\rm(b)] For $x\in (0,\beta)$,
			\begin{equation*}
				\widetilde N_+(x)=\widetilde N_-(x)
				\begin{pmatrix}
					0 & x^{\alpha}\\
					-x^{-\alpha} & 0
				\end{pmatrix}.
			\end{equation*}
			\item[\rm(c)] As $z\to\infty$, $\widetilde N(z)=I+\mathcal{O}(z^{-1})$.
		\end{itemize}
	\end{rhp}
	Its explicit solution is
	\begin{align}\label{defn}
		\widetilde N(z)
		&=\left(\frac{\beta}{4}\right)^{\frac{\alpha}{2}\sigma_3}
		\begin{pmatrix}
			\dfrac12\bigl(\wtil\gamma(z)+\wtil\gamma(z)^{-1}\bigr) & \dfrac{1}{2\ii}\bigl(\wtil\gamma(z)-\wtil\gamma(z)^{-1}\bigr)\\[1ex]
			-\dfrac{1}{2\ii}\bigl(\wtil\gamma(z)-\wtil\gamma(z)^{-1}\bigr) & \dfrac12\bigl(\wtil\gamma(z)+\wtil\gamma(z)^{-1}\bigr)
		\end{pmatrix}D(z)^{-\sigma_3},
	\end{align}
	where
	\begin{equation*}
		\wtil\gamma(z)=\left(\frac{z-\beta}{z}\right)^{1/4},
		\qquad
		D(z)=\left(\frac{z}{\varphi^*(z)}\right)^{\alpha/2},
		\qquad
		\varphi^*(z)=\frac{2}{\beta}\left(z-\frac{\beta}{2}+z^{1/2}(z-\beta)^{1/2}\right),
	\end{equation*}
	where each fractional power is taken on its principal branch.
	
	\subsubsection{Local parametrix near $z=0$}
	Let $U_\varepsilon(0)$ be a sufficiently small neighborhood of the origin. We seek a matrix $\widetilde P^{(0)}$ that satisfies the same jumps as $\widetilde S$ inside $U_\varepsilon(0)$ and matches $\widetilde N$ on the boundary.
	
	\begin{rhp}\label{rhpptilde}
		\hfill
		\begin{itemize}
			\item[\rm(a)] $\widetilde P^{(0)}$ is analytic in $U_\varepsilon(0)\setminus \Gamma_{\widetilde S}$.
			\item[\rm(b)] For $z\in U_\varepsilon(0)\cap \Gamma_{\widetilde S}$, we have
			\begin{equation*}
				\widetilde P^{(0)}_+(z)=\widetilde P^{(0)}_-(z)J_{\widetilde S}(z),
			\end{equation*}
			where $J_{\widetilde S}$ is defined in \eqref{def:JShat}.
			\item[\rm(c)] As $n\to\infty$,
			\begin{equation}\label{match-P0}
				\widetilde P^{(0)}(z)=\left(I+\mathcal{O}(n^{-1})\right)\widetilde N(z),
				\qquad z\in \partial U_\varepsilon(0).
			\end{equation}
		\end{itemize}
	\end{rhp}
	The construction uses the Bessel parametrix $\Phi_{\alpha}^{(\Bes)}$ from Appendix B. Define the conformal map $\widetilde f_n$ by
	\begin{equation}\label{deffn}
		e^{2\widetilde f_n(z)^{1/2}}=(-1)^n e^{n\widetilde\xi(z)},
	\end{equation}
	which yields the explicit expression
	\begin{equation}\label{deffn-explicit}
		\widetilde f_n(z)=-\frac14 n^2\left(\frac{\widetilde\xi(z)-\widetilde\xi_+(0)}{\ii}\right)^2.
	\end{equation}
	It is convenient to set
	\begin{equation}\label{deff-hard}
		\widetilde f(z):=n^{-2}\widetilde f_n(z),
	\end{equation}
	so that $\widetilde f$ is independent of $n$. The local parametrix at the hard edge is
	\begin{equation}\label{def:P0hard}
		\widetilde P^{(0)}(z)=\widetilde E_n(z)\Phi_{\alpha}^{(\Bes)}\bigl(\widetilde f_n(z)\bigr)e^{-n\widetilde\xi(z)\sigma_3}(-z)^{-\frac{\alpha}{2}\sigma_3},
	\end{equation}
	where the analytic prefactor is given by
	\begin{equation}\label{E2}
		\widetilde E_n(z)=(-1)^n\widetilde N(z)(-z)^{\frac{\alpha}{2}\sigma_3}\frac{1}{\sqrt2}
		\begin{pmatrix}
			1 & \ii\\
			\ii & 1
		\end{pmatrix}
		\widetilde f_n(z)^{\frac14\sigma_3}(2\pi)^{\frac12\sigma_3}.
	\end{equation}
	The hard-edge matching expansion is
	\begin{equation}\label{P0hard-asy}
		\widetilde P^{(0)}(z)\widetilde N(z)^{-1}\sim I+\sum_{k=1}^{\infty}\frac{1}{n^k}\widetilde\Delta_k^{(0)}(z),
		\qquad n\to\infty,
	\end{equation}
	uniformly for $z\in \partial U_\varepsilon(0)$, where each
	$\widetilde\Delta_k^{(0)}$ is meromorphic in $U_\varepsilon(0)$ with a pole only
	at $z=0$. 
    
    Let
	\begin{equation}\label{def:varphi-hard}
		\widetilde\varphi(z):=\frac12\sqrt{\frac{z-\beta}{z}}\,h_Q(z), \qquad z\in \mathbb{C}\setminus [0,\beta].
	\end{equation}
	Then $\widetilde\xi(z)=-\int_\beta^z \widetilde\varphi(s)\,\ud s$, so that
	\begin{equation*}
		\widetilde\xi(z)-\widetilde\xi_+(0)=-\int_0^z \widetilde\varphi(s)\,\ud s,
		\qquad
		\widetilde f_n(z)=\left(\frac n2\int_0^z \widetilde\varphi(s)\,\ud s\right)^2.
	\end{equation*}
	From \cite[Section~3]{Vanlessen07} and \cite[Section~3]{ZCD14}, we obtain
	\begin{equation}\label{Delta0-explicit}
		\widetilde\Delta_k^{(0)}(z)
		=
		\frac{(\alpha,k-1)}{2^k\bigl(\widetilde\xi(z)-\widetilde\xi_+(0)\bigr)^k}
		\widetilde N(z)(-z)^{\frac{\alpha}{2}\sigma_3}\widetilde C_k(-z)^{-\frac{\alpha}{2}\sigma_3}\widetilde N(z)^{-1},
	\end{equation}
	with
	\begin{equation*}
		(\alpha,m):=\frac{\prod_{j=1}^{m}\bigl(4\alpha^2-(2j-1)^2\bigr)}{2^{2m}m!},
		\qquad m\ge 1,
		\qquad
		(\alpha,0):=1,
	\end{equation*}
	where
	\begin{equation*}
		\widetilde C_k:=
		\begin{pmatrix}
			\dfrac{(-1)^k}{k}\left(\alpha^2+\dfrac{k}{2}-\dfrac14\right) & \left(k-\dfrac12\right)\ii\\[2ex]
			(-1)^{k+1}\left(k-\dfrac12\right)\ii & \dfrac{1}{k}\left(\alpha^2+\dfrac{k}{2}-\dfrac14\right)
		\end{pmatrix}.
	\end{equation*}
	In particular,
	\begin{equation}\label{Delta10-explicit}
		\widetilde\Delta_1^{(0)}(z)
		=
		\frac{1}{2\bigl(\widetilde\xi(z)-\widetilde\xi_+(0)\bigr)}
		\widetilde N(z)(-z)^{\frac{\alpha}{2}\sigma_3}
		\begin{pmatrix}
			-\alpha^2-\dfrac14 & \dfrac{\ii}{2}\\[2ex]
			\dfrac{\ii}{2} & \alpha^2+\dfrac14
		\end{pmatrix}
		(-z)^{-\frac{\alpha}{2}\sigma_3}\widetilde N(z)^{-1}.
	\end{equation}
	
	\subsubsection{Local parametrix near $z=\beta$}
	In a neighborhood $U_\varepsilon(\beta)$ of the right endpoint $\beta$, the local paramertix reads as follows. 
    
    % use the Airy parametrix from the soft-edge analysis and denote it by $\widetilde P^{(\beta)}$. It is characterized by the following RH problem.
	
	\begin{rhp}\label{pn}
		\hfill
		\begin{itemize}
			\item[\rm(a)] $\widetilde P^{(\beta)}$ is analytic in $U_\varepsilon(\beta)\setminus \Gamma_{\widetilde S}$.
			\item[\rm(b)] For $z\in U_\varepsilon(\beta)\cap \Gamma_{\widetilde S}$,
			\begin{equation*}
				\widetilde P^{(\beta)}_+(z)=\widetilde P^{(\beta)}_-(z)J_{\widetilde S}(z).
			\end{equation*}
			\item[\rm(c)] As $n\to\infty$,
			\begin{equation*}
				\widetilde P^{(\beta)}(z)=\left(I+\mathcal{O}(n^{-1})\right)\widetilde N(z),
				\qquad z\in \partial U_\varepsilon(\beta).
			\end{equation*}
		\end{itemize}
	\end{rhp}
	Similar to the construction of $P^{(b)}$ in \eqref{def:Pb}, one can solve the above RH problem with the aid of the Airy parametrix. We omit the details but note that as $n\to infty$, 
	\begin{equation}\label{PN-111}
		\widetilde P^{(\beta)}(z)\widetilde N(z)^{-1}\sim I+\sum_{k=1}^{\infty}\frac{1}{n^k}\widetilde\Delta_k^{(\beta)}(z),
		\qquad z\in \partial U_\varepsilon(\beta),
	\end{equation}
    for some functions $\Delta_k^{(\beta)}$.
	
	\subsubsection{Final transformation}\label{rhp.S2}
	Set
	\begin{equation}\label{def:tildeR}
		\widetilde R(z)=
		\begin{cases}
			\widetilde S(z)\widetilde P^{(\beta)}(z)^{-1}, & z\in U_\varepsilon(\beta),\\
			\widetilde S(z)\widetilde P^{(0)}(z)^{-1}, & z\in U_\varepsilon(0),\\
			\widetilde S(z)\widetilde N(z)^{-1}, & \text{elsewhere}.
		\end{cases}
	\end{equation}
	The matrix $\widetilde R$ satisfies the following RH problem.
\begin{figure}[htbp]
	\centering
	\begin{tikzpicture}[x=1.2pt,y=1.2pt,yscale=-1,xscale=1, thin]
		
		\draw (190,145) .. controls (190,131.19) and (201.19,120) .. (215,120) .. controls (228.81,120) and (240,131.19) .. (240,145) .. controls (240,158.81) and (228.81,170) .. (215,170) .. controls (201.19,170) and (190,158.81) .. (190,145) -- cycle ;
		\draw [fill={rgb, 255:red, 0; green, 0; blue, 0 } ,fill opacity=1 ] (217.5,120) -- (212.5,122.87) -- (212.5,117.13) -- cycle ;
		\draw [fill={rgb, 255:red, 0; green, 0; blue, 0 } ,fill opacity=1 ] (214,145) .. controls (214,144.45) and (214.45,144) .. (215,144) .. controls (215.55,144) and (216,144.45) .. (216,145) .. controls (216,145.55) and (215.55,146) .. (215,146) .. controls (214.45,146) and (214,145.55) .. (214,145) -- cycle ;
		
		\draw (300,145) .. controls (300,131.19) and (311.19,120) .. (325,120) .. controls (338.81,120) and (350,131.19) .. (350,145) .. controls (350,158.81) and (338.81,170) .. (325,170) .. controls (311.19,170) and (300,158.81) .. (300,145) -- cycle ;
		\draw [fill={rgb, 255:red, 0; green, 0; blue, 0 } ,fill opacity=1 ] (328,120) -- (323,122.87) -- (323,117.13) -- cycle ;
		\draw [fill={rgb, 255:red, 0; green, 0; blue, 0 } ,fill opacity=1 ] (324,144) .. controls (324,143.45) and (324.45,143) .. (325,143) .. controls (325.55,143) and (326,143.45) .. (326,144) .. controls (326,144.55) and (325.55,145) .. (325,145) .. controls (324.45,145) and (324,144.55) .. (324,144) -- cycle ;
		
		\draw (350,145) -- (430,145) ;
		\draw [fill={rgb, 255:red, 0; green, 0; blue, 0 } ,fill opacity=1 ] (395,145) -- (389.39,147.87) -- (389.39,142.13) -- cycle ;
		
		\draw (227.81,123.69) .. controls (234.13,115.74) and (250.41,110.08) .. (269.5,110.08) .. controls (288.59,110.08) and (304.87,115.74) .. (311.19,123.69) ;
		\draw [fill={rgb, 255:red, 0; green, 0; blue, 0 } ,fill opacity=1 ] (271.77,110.08) -- (265.41,112.95) -- (265.41,107.21) -- cycle ;
		
		\draw (312.19,167.52) .. controls (305.87,175.48) and (289.59,181.13) .. (270.5,181.13) .. controls (251.41,181.13) and (235.13,175.48) .. (228.81,167.52) ;
		\draw [fill={rgb, 255:red, 0; green, 0; blue, 0 } ,fill opacity=1 ] (274.59,181.13) -- (268.23,184) -- (268.23,178.27) -- cycle ;
		
		\draw (210,150) node [anchor=north west][inner sep=0.75pt] [align=left] {$0$};
		\draw (321,150) node [anchor=north west][inner sep=0.75pt] [align=left] {$\beta$};
		
		\draw (260,90) node [anchor=north west][inner sep=0.75pt] [align=left] {$\partial \Omega_{\widetilde S,+}$};
		\draw (260,190) node [anchor=north west][inner sep=0.75pt] [align=left] {$\partial \Omega_{\widetilde S,-}$};
		\draw (180,110) node [anchor=north west][inner sep=0.75pt] [align=left] {$\partial U_\varepsilon(0)$};
		\draw (330,110) node [anchor=north west][inner sep=0.75pt] [align=left] {$\partial U_\varepsilon(\beta)$};
	\end{tikzpicture}
	\caption{The jump contours $\Gamma_{\widetilde R}$ of the RH problem \ref{rhp57} for $\widetilde R$.}
	\label{fig:contour_R11}
\end{figure}
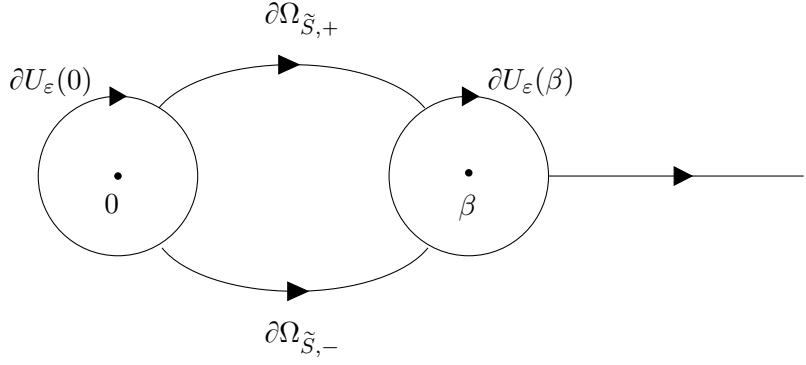
	\begin{rhp}\label{rhp57}
		\hfill
		\begin{itemize}
			\item[\rm(a)] $\widetilde R$ is analytic in $\mathbb{C}\setminus \Gamma_{\widetilde R}$, 
            see Figure \ref{fig:contour_R11} for an illustration of $\Gamma_{\widetilde R}$.
   %          where
			% \begin{equation*}
			% 	\Gamma_{\widetilde R}=
			% 	\left(\Gamma_{\widetilde S}\setminus \left(U_\varepsilon(0)\cup U_\varepsilon(\beta)\cup (0,+\infty)\right)\right)
			% 	\cup \partial U_\varepsilon(0)\cup \partial U_\varepsilon(\beta).
			% \end{equation*}
			\item[\rm(b)] For $z\in \Gamma_{	\widetilde R}$, we have $\widetilde R_+(z)=\widetilde R_-(z)J_{\widetilde R}(z)$, where
			\begin{equation}\label{def:JRhat}
				J_{\widetilde R}(z)=
				\begin{cases}
					\widetilde P^{(\beta)}(z)\widetilde N(z)^{-1}, & z\in \partial U_\varepsilon(\beta),\\
					\widetilde P^{(0)}(z)\widetilde N(z)^{-1}, & z\in \partial U_\varepsilon(0),\\
					\widetilde N(z)J_{\widetilde S}(z)\widetilde N(z)^{-1}, & \text{elsewhere}.
				\end{cases}
			\end{equation}
			\item[\rm(c)] As $z\to\infty$, $\widetilde R(z)=I+\mathcal{O}(z^{-1})$.
		\end{itemize}
	\end{rhp}
	On the circles $\partial U_\varepsilon(0)$ and $\partial U_\varepsilon(\beta)$, the jump matrix admits a full asymptotic expansion in inverse powers of $n$, whereas on the remaining contours it converges exponentially fast to the identity matrix. The small-norm RH analysis gives
	\begin{equation}\label{R1}
		\widetilde R(z)\sim I+\sum_{k=1}^{\infty}\frac{1}{n^k}\widetilde R_k(z),
		\qquad n\to\infty.
	\end{equation}
	The first coefficient is characterized by an additive RH problem with jumps $\widetilde\Delta_1^{(0)}$ and $\widetilde\Delta_1^{(\beta)}$, and similarly for all higher coefficients. In particular,
	\begin{equation}\label{Rk}
		\widetilde R_k(z)=\frac{1}{2\pi\ii}\oint_{\partial U_\varepsilon(\beta)}\frac{\sum_{\ell=1}^k \widetilde R_{k-\ell,-}(s)\widetilde\Delta_{\ell}^{(\beta)}(s)}{s-z}\,\ud s
		+\frac{1}{2\pi\ii}\oint_{\partial U_\varepsilon(0)}\frac{\sum_{\ell=1}^k \widetilde R_{k-\ell,-}(s)\widetilde\Delta_{\ell}^{(0)}(s)}{s-z}\,\ud s.
	\end{equation}
	\begin{corollary}\label{cor:hard-R-structure}
		For each $m\ge 1$, one has
		\begin{equation}\label{r2m1}
			\left(\frac{\beta}{4}\right)^{-\frac{\alpha}{2}\sigma_3}\widetilde R_m(z)\left(\frac{\beta}{4}\right)^{\frac{\alpha}{2}\sigma_3}
			=\widetilde\alpha_{0,m}(z)I+\widetilde\alpha_{1,m}(z)\sigma_1+\widetilde\alpha_{2,m}(z)\sigma_2+\widetilde\alpha_{3,m}(z)\sigma_3.
		\end{equation}
	\end{corollary}
	\begin{proof}
		This is the Pauli-basis decomposition of a $2\times2$ matrix.
	\end{proof}
	
	\section{Proofs of main results}\label{sec:proof}
	\subsection{Proof of Theorem \ref{softthm}}
		To analyze the rescaled kernel $\widehat K_n$ defined in \eqref{eq:softexpan} near the right endpoint $b$, set
		\begin{equation}\label{def:xn}
			x_n:= b+\frac{x}{\kappa_V n^{\frac{2}{3}}},
			\quad
			y_n:= b+\frac{y}{\kappa_V n^{\frac{2}{3}}}
		\end{equation}
		and, after shrinking $\varepsilon$ if necessary, assume that $D\!\left(b,\frac{3}{2}\varepsilon\right)$ is an admissible Airy neighborhood. Split the scaled half-line into
		\[
			I_n := [t_0,\varepsilon\kappa_V n^{2/3}), \qquad II_n := [\varepsilon\kappa_V n^{2/3},\infty),
		\]
		so that for all large $n$, $x\in I_n$ implies $x_n\in D(b,\varepsilon)$, while the bound $x<\frac32\varepsilon\kappa_V n^{2/3}$ implies $x_n\in D\!\left(b,\frac32\varepsilon\right)$. The local estimate is first obtained on
		\[
			(x,y)\in I_n\times I_n.
		\]
		For $(x,y)\in I_n\times I_n$ with $x,y<0$, tracing back the transformations $Y \to \what{T} \to \what S \to \what R$ in \eqref{def:hatT}, \eqref{def:hatS}, and \eqref{def:hatR} gives
		\begin{align}
			\frac{1}{\kappa_V n^{\frac{2}{3}}} K_n(x_n, y_n) 
			&= \frac{1}{2\pi \mathrm{i}(x-y)} \begin{pmatrix} 0 & e^{-n\varphi_{+}(y_n)} \end{pmatrix} \what T_+(y_n)^{-1} \what T_+(x_n) \begin{pmatrix} e^{-n\varphi_{+}(x_n)} \\ 0 \end{pmatrix} \nonumber\\
			&= \frac{1}{2\pi \mathrm{i}(x-y)} \begin{pmatrix} -e^{n\varphi_{+}(y_n)} & e^{-n\varphi_{+}(y_n)} \end{pmatrix} \what S_+(y_n)^{-1} \what S_+(x_n) \begin{pmatrix} e^{-n\varphi_{+}(x_n)} \\ e^{n\varphi_{+}(x_n)} \end{pmatrix} \nonumber\\
			&= \frac{1}{2\pi \mathrm{i}(x - y)}\begin{pmatrix} -1 & 1 \end{pmatrix} \Phi_+^{(\mathrm{Ai})}(n^{\frac{2}{3}}\what f(y_n))^{-1} \what E_n^{(b)}(y_n)^{-1} \what R(y_n)^{-1} \nonumber \\
			&\quad \times \what R(x_n) \what E_n^{(b)}(x_n) \Phi_+^{(\mathrm{Ai})}(n^{\frac{2}{3}}\what f(x_n)) \begin{pmatrix} 1 \\ 1 \end{pmatrix} \nonumber \\
			&= \frac{1}{\mathrm{i}(x - y)} \begin{pmatrix} \mathrm{i}\mathrm{Ai}'(n^{\frac{2}{3}}\what f(y_n)) &  \mathrm{Ai}(n^{\frac{2}{3}}\what f(y_n)) \end{pmatrix} \what E_n^{(b)}(y_n)^{-1} \what R(y_n)^{-1} \nonumber \\
			&\qquad \times \what R(x_n) \what E_n^{(b)}(x_n) \begin{pmatrix} \mathrm{Ai}(n^{\frac{2}{3}}\what f(x_n)) \\ -\mathrm{i}\mathrm{Ai}'(n^{\frac{2}{3}}\what f(x_n)) \end{pmatrix}. \label{eq:soft-local-kernel-representation}
		\end{align}
		The cases where $x > 0$ and/or $y > 0$ are handled in the same way; hence the final expression in \eqref{eq:soft-local-kernel-representation} holds throughout $I_n\times I_n$.

		Consider the factor $\what E_n^{(b)}(y_n)^{-1} \what R(y_n)^{-1} \what R(x_n) \what E_n^{(b)}(x_n)$. The representation \eqref{def:E} of $\what E_n^{(b)}$ gives
		\begin{align}\label{ERRE}
			&\what E_n^{(b)}(y_n)^{-1} \what R(y_n)^{-1} \what R(x_n) \what E_n^{(b)}(x_n) 
            \nonumber \\
			&= \what E_n^{(b)}(y_n)^{-1} \frac{1}{\sqrt 2} \begin{pmatrix} 1 & -\mathrm{i}\\ -\mathrm{i} & 1 \end{pmatrix} \frac{1}{\sqrt 2} \begin{pmatrix} 1 & \mathrm{i}\\ \mathrm{i} & 1 \end{pmatrix} \what R(y_n)^{-1} \frac{1}{\sqrt 2} \begin{pmatrix} 1 & -\mathrm{i}\\ -\mathrm{i} & 1 \end{pmatrix} \nonumber \\
			&\quad \times \frac{1}{\sqrt 2} \begin{pmatrix} 1 & \mathrm{i}\\ \mathrm{i} & 1 \end{pmatrix} \what R(x_n) \frac{1}{\sqrt 2} \begin{pmatrix} 1 & -\mathrm{i}\\ -\mathrm{i} & 1 \end{pmatrix} \frac{1}{\sqrt 2} \begin{pmatrix} 1 & \mathrm{i}\\ \mathrm{i} & 1 \end{pmatrix} \what E_n^{(b)}(x_n) \nonumber \\
			&= \left((y_n-a)\frac{\what f(y_n)}{y_n-b}\right)^{-\frac{\sigma_3}{4}} \left(n^{-\frac{\sigma_3}{6}}\frac{1}{\sqrt 2} \begin{pmatrix} 1 & \mathrm{i}\\ \mathrm{i} & 1 \end{pmatrix} \what R(y_n) \frac{1}{\sqrt 2} \begin{pmatrix} 1 & -\mathrm{i}\\ -\mathrm{i} & 1 \end{pmatrix}n^{\frac{\sigma_3}{6}}\right)^{-1} \nonumber\\
			&\quad  \times \left(n^{-\frac{\sigma_3}{6}}\frac{1}{\sqrt 2} \begin{pmatrix} 1 & \mathrm{i}\\ \mathrm{i} & 1 \end{pmatrix} \what R(x_n) \frac{1}{\sqrt 2} \begin{pmatrix} 1 & -\mathrm{i}\\ -\mathrm{i} & 1 \end{pmatrix}n^{\frac{\sigma_3}{6}}\right) \left((x_n-a)\frac{\what f(x_n)}{x_n-b}\right)^{\frac{\sigma_3}{4}}.
		\end{align}
Equations \eqref{asy:R} and \eqref{def:Rk} give the expansion of the conjugated term involving $\what R(x_n)$:
		\begin{multline}\label{asy:R_expanded}
			n^{-\frac{\sigma_3}{6}}\frac{1}{\sqrt 2} \begin{pmatrix} 1 & \mathrm{i}\\ \mathrm{i} & 1 \end{pmatrix} \what R(x_n) \frac{1}{\sqrt 2} \begin{pmatrix} 1 & -\mathrm{i}\\ -\mathrm{i} & 1 \end{pmatrix}n^{\frac{\sigma_3}{6}}
            \\
            \sim  I+2\sum_{k=1}^{\infty} \begin{pmatrix}
				\what \alpha_{2,2k}(x_n) n^{-2k}& -\mathrm{i} \what \alpha_{2,2k-1}(x_n) n^{-2k+\tfrac{2}{3}}\\
				\mathrm{i} \what \alpha_{1,2k-1}(x_n) n^{-2k+\tfrac{4}{3}} & \what \alpha_{1,2k}(x_n) n^{-2k}
			\end{pmatrix}.
		\end{multline}
		From the definitions of $x_n$ and $y_n$ in \eqref{def:xn}, if $x = y$, then 
		\begin{equation*}
			\what E_n^{(b)}(y_n)^{-1} \what R(y_n)^{-1} \what R(x_n) \what E_n^{(b)}(x_n)=I.
		\end{equation*}
		 Combining this with \eqref{ERRE} and \eqref{asy:R_expanded}, we conclude that, as $n\to \infty$,
		\begin{equation}\label{eq:ERRE_expansion}
			\what E_n^{(b)}(y_n)^{-1} \what R(y_n)^{-1} \what R(x_n) \what E_n^{(b)}(x_n)\sim I+(x-y)\sum_{j=1}^{\infty} \what e_j(x,y) n^{-\frac{2j}{3}},
		\end{equation}
		where $\what e_j(x, y)$ are matrices whose entries are polynomials in $x$ and $y$.

		For the Airy-function factors, the expansion of $\what f$ near $z = b$ in \eqref{def:f} gives
		\begin{equation} \label{eq111}
			n^{\frac{2}{3}}\what f(x_n)=x+\sum_{j = 1}^{\infty} \what p_{1,j}(x)n^{-\frac{2j}{3}}, \quad n\to \infty,
		\end{equation}
		where $\what p_{1,j}(x)$ is a polynomial of degree $j+1$.
		Expanding higher powers of $n^{\frac{2}{3}}\what f(x_n)-x$ yields
		\begin{equation} \label{eq333}
			\frac{1}{m!} \left(n^{\frac{2}{3}}\what f(x_n)-x\right) ^m=\sum_{j = m}^{\infty} \what p_{m,j}(x)n^{-\frac{2j}{3}} , \quad n \to \infty,
		\end{equation}
		where, for any integer $r$ with $1\le r\le m-1$,
		\begin{equation} \label{pmj}
			\what p_{m,j}(x)=\frac{(m-r)!r!}{m!}\sum_{k=m-r}^{j-r}\what p_{m-r,k}(x)\what p_{r,j-k}(x),\qquad m \leqslant j.   
		\end{equation}
		Inserting \eqref{eq111} and \eqref{eq333} into the Taylor series of the Airy function and its derivative, and using \eqref{Airy_function} to eliminate higher derivatives, we obtain
		\begin{align}\label{eq:Ai_expansion}
			\mathrm{Ai}\left(n^{\frac{2}{3}}\what f(x_n)\right) 
			&= \mathrm{Ai}(x)+\sum_{j = 1}^{\infty} \sum_{m = 1}^{j}\what p_{m,j}(x)\mathrm{Ai}^{(m)}(x)n^{-\frac{2j}{3}}\nonumber\\
			&= \left( 1 + \sum_{j=1}^{\infty} \sum_{m=1}^{j} \what p_{m,j}(x) \what P_{m}(x) n^{-\frac{2j}{3}} \right) \mathrm{Ai}(x)
			+ \left( \sum_{j=1}^{\infty} \sum_{m=1}^{j} \what p_{m,j}(x) \what Q_{m}(x) n^{-\frac{2j}{3}} \right) \mathrm{Ai}'(x),
		\end{align}
		and similarly,
		\begin{align}\label{eq:Ai_prime_expansion}
			\mathrm{Ai}'\left(n^{\frac{2}{3}}\what f(x_n)\right) 
			&= \mathrm{Ai}'(x)+\sum_{j = 1}^{\infty} \sum_{m = 1}^{j}\what p_{m,j}(x)\mathrm{Ai}^{(m+1)}(x)n^{-\frac{2j}{3}} \nonumber \\
			&= \left( \sum_{j=1}^{\infty} \sum_{m=1}^{j} \what p_{m,j}(x) \what P_{m+1}(x) n^{-\frac{2j}{3}} \right) \mathrm{Ai}(x)
			+ \left( 1 + \sum_{j=1}^{\infty} \sum_{m=1}^{j} \what p_{m,j}(x) \what Q_{m+1}(x) n^{-\frac{2j}{3}} \right) \mathrm{Ai}'(x).
		\end{align}
		Combining \eqref{eq:Ai_expansion} and \eqref{eq:Ai_prime_expansion}, we obtain
		\begin{align}\label{eq:soft-airy-wronskian-expansion}
			&\mathrm{Ai}'(n^{\frac{2}{3}}\what f(y_n))\mathrm{Ai}(n^{\frac{2}{3}}\what f(x_n))-\mathrm{Ai}(n^{\frac{2}{3}}\what f(y_n))\mathrm{Ai}'(n^{\frac{2}{3}}\what f(x_n)) \nonumber\\
			&= \left[ \left( 1 + \sum_{j=1}^{\infty} \sum_{m=1}^{j} \what p_{m,j}(x) \what P_{m}(x) n^{-\frac{2j}{3}} \right) \mathrm{Ai}(x) + \left( \sum_{j=1}^{\infty} \sum_{m=1}^{j} \what p_{m,j}(x) \what Q_{m}(x) n^{-\frac{2j}{3}} \right) \mathrm{Ai}'(x) \right] \nonumber\\
			&\qquad \times \left[ \left( \sum_{j=1}^{\infty} \sum_{m=1}^{j} \what p_{m,j}(y) \what P_{m+1}(y) n^{-\frac{2j}{3}} \right) \mathrm{Ai}(y) + \left( 1 + \sum_{j=1}^{\infty} \sum_{m=1}^{j} \what p_{m,j}(y) \what Q_{m+1}(y) n^{-\frac{2j}{3}} \right) \mathrm{Ai}'(y) \right] \nonumber\\
			&\quad -\left[ \left( \sum_{j=1}^{\infty} \sum_{m=1}^{j} \what p_{m,j}(x) \what P_{m+1}(x) n^{-\frac{2j}{3}} \right) \mathrm{Ai}(x) + \left( 1 + \sum_{j=1}^{\infty} \sum_{m=1}^{j} \what p_{m,j}(x) \what Q_{m+1}(x) n^{-\frac{2j}{3}} \right) \mathrm{Ai}'(x) \right] \nonumber\\
			&\qquad \times \left[ \left( 1 + \sum_{j=1}^{\infty} \sum_{m=1}^{j} \what p_{m,j}(y) \what P_{m}(y) n^{-\frac{2j}{3}} \right) \mathrm{Ai}(y) + \left( \sum_{j=1}^{\infty} \sum_{m=1}^{j} \what p_{m,j}(y) \what Q_{m}(y) n^{-\frac{2j}{3}} \right) \mathrm{Ai}'(y) \right] \nonumber\\
			&= \mathrm{Ai}(x)\mathrm{Ai}'(y)-\mathrm{Ai}'(x)\mathrm{Ai}(y)+\sum_{N = 1}^{\infty} L_N(x,y)n^{-\frac{2N}{3}},
		\end{align}
		where
		\begin{align*}
			L_N(x,y) &= a_{N,00}(x,y)\mathrm{Ai}(x)\mathrm{Ai}(y)+a_{N,01}(x,y)\mathrm{Ai}(x)\mathrm{Ai}'(y) \nonumber \\
			&\quad -a_{N,01}(y,x)\mathrm{Ai}'(x)\mathrm{Ai}(y)+a_{N,11}(x,y)\mathrm{Ai}'(x)\mathrm{Ai}'(y).
		\end{align*}
		Here, the coefficients $a_{N,ij}(x,y)$ for $i,j \in \{0,1\}$ are explicitly given  by
		\begin{align}
			a_{N,00}(x,y) 
			&= \sum_{\ell=1}^N \left(\what p_{\ell,N}(y) \what P_{\ell+1}(y)-\what p_{\ell,N}(x) \what P_{\ell+1}(x)\right) \nonumber \\ 
			&\quad +\sum_{\substack{j + k = N \\ j \geq 1,\, k \geq 1}} \sum_{m = 1}^j \sum_{\ell = 1}^k \left(\what p_{m,j}(x) \what P_{m}(x) \what p_{\ell,k}(y) \what P_{\ell+1}(y) - \what p_{m,j}(x) \what P_{m+1}(x) \what p_{\ell,k}(y) \what P_{\ell}(y) \right), \label{an00} \\
			a_{N,11}(x,y) 
			&= \sum_{\ell=1}^N \left(\what p_{\ell,N}(x) \what Q_{\ell+1}(x)-\what p_{\ell,N}(y) \what Q_{\ell+1}(y)\right) \nonumber \\ 
			&\quad +\sum_{\substack{j + k = N \\ j \geq 1,\, k \geq 1}} \sum_{m = 1}^j \sum_{\ell = 1}^k \left(\what p_{m,j}(x) \what Q_{m+1}(x) \what p_{\ell,k}(y) \what Q_{\ell}(y) - \what p_{m,j}(x) \what Q_{m}(x) \what p_{\ell,k}(y) \what Q_{\ell+1}(y) \right), \label{an11} \\
			a_{N,01}(x,y) 
			&= \sum_{\ell=2}^N\left(\what p_{\ell,N}(x) \what P_{\ell}(x)+\what p_{\ell,N}(y) \what Q_{\ell+1}(y)\right) \nonumber \\ 
			&\quad +\sum_{\substack{j + k = N \\ j \geq 1,\, k \geq 1}} \sum_{m = 1}^j \sum_{\ell = 1}^k \what p_{m,j}(x)\what p_{\ell,k}(y)\left(\what P_{m}(x)\what Q_{\ell+1}(y)- \what P_{m+1}(x)\what Q_{\ell}(y)\right), \label{an01} 
		\end{align}
		Equations \eqref{an00} and \eqref{an11} give $a_{N,00}(x,x)=a_{N,11}(x,x)=0$. For the mixed coefficient, regrouping the terms in \eqref{an01} according to $t=m+\ell$ and using \eqref{pmj} gives
		\begin{align*}
			a_{N,01}(x,x) &= \sum_{\ell=2}^N \what p_{\ell,N}(x)\left(\what P_{\ell}(x) + \what Q_{\ell+1}(x)\right) \nonumber\\
			&\quad + \sum_{t=2}^N \what p_{t, N}(x) \sum_{m=1}^{t-1} \frac{t!}{m!(t-m)!} \left( \what P_m(x) \what Q_{t-m+1}(x) - \what P_{m+1}(x) \what Q_{t-m}(x) \right) \nonumber\\
			&= \sum_{t=2}^{N} t!\, \what p_{t,N}(x) \left( \sum_{m=0}^{t} \left( \frac{1}{m! \, (t-m)!} \left( \what P_m(x) \what Q_{t+1-m}(x) - \what P_{t+1-m}(x)\what Q_m(x) \right) \right) \right).
		\end{align*}
		The inner sum is the Airy Wronskian identity \eqref{airy law} with $t$ in place of $N$, so $a_{N,01}(x,x)=0$. Consequently, each of the polynomials $a_{N,00}(x,y)$, $a_{N,11}(x,y)$, $a_{N,01}(x,y)$, and $a_{N,01}(y,x)$ vanishes along the diagonal $x=y$, and is therefore divisible by $x-y$. Hence there exist polynomials $\widehat b_{N,\kappa\lambda}(x,y)$ such that
		\begin{align}\label{eq:LN_divisible}
			\frac{L_N(x,y)}{x-y}
			&= \widehat b_{N,00}(x,y)\Ai(x)\Ai(y)+\widehat b_{N,01}(x,y)\Ai(x)\Ai'(y) \nonumber\\
			&\quad +\widehat b_{N,10}(x,y)\Ai'(x)\Ai(y)+\widehat b_{N,11}(x,y)\Ai'(x)\Ai'(y).
		\end{align}
		Substituting \eqref{eq:ERRE_expansion}, \eqref{eq:Ai_expansion}, \eqref{eq:Ai_prime_expansion}, and \eqref{eq:soft-airy-wronskian-expansion} into \eqref{eq:soft-local-kernel-representation} gives 
		\begin{align*}
			&\frac{1}{x-y}\left[\Ai'(n^{\frac{2}{3}}\what f(y_n))\Ai(n^{\frac{2}{3}}\what f(x_n))-\Ai(n^{\frac{2}{3}}\what f(y_n))\Ai'(n^{\frac{2}{3}}\what f(x_n))\right] \nonumber\\
			&= K_{\Ai}(x,y)+\sum_{j=1}^{\infty}\frac{L_j(x,y)}{x-y}n^{-\frac{2j}{3}}.
		\end{align*}
		The remaining contribution comes from the factor $(x-y)\sum_{j=1}^{\infty}\what e_j(x,y)n^{-2j/3}$ in \eqref{eq:ERRE_expansion}. Dividing by $x-y$ and using \eqref{eq:Ai_expansion} and \eqref{eq:Ai_prime_expansion}, one finds that each coefficient is again a finite linear combination of $\Ai(x)\Ai(y)$, $\Ai(x)\Ai'(y)$, $\Ai'(x)\Ai(y)$, and $\Ai'(x)\Ai'(y)$ with polynomial coefficients in $x$ and $y$. 

		It remains to prove the uniform remainder bound. Recall that
		\[
			I_n=[t_0,\varepsilon\kappa_V n^{2/3}),
			\qquad
			II_n=[\varepsilon\kappa_V n^{2/3},\infty).
		\]
		On $I_n\times I_n$, expanding \eqref{ERRE} and the Airy factors to order $m$ and using \eqref{eq:LN_divisible} gives the remainder
		\[
			\widehat K_n(x,y)-K_{\Ai}(x,y)-\sum_{j=1}^{m}\widehat L_j(x,y)n^{-2j/3},
		\]
		which is a finite sum of terms
		\[
			n^{-2(m+1)/3}P_{n,m}(x,y)\Ai^{(a)}(\xi_x)\Ai^{(b)}(\xi_y),
			\qquad a,b\in\{0,1\},
		\]
		Here, $P_{n,m}(x,y)$ is a polynomial in $x$ and $y$, and the RH factors are bounded in $D(b,\varepsilon)$. Since $\what f(z)=\kappa_V(z-b)+\mathcal O((z-b)^2)$, the points $\xi_x$ and $\xi_y$ have the same scale as $x$ and $y$. The Airy asymptotics \cite[\S9.7]{DLMF} imply, for $r=0,1$ and every fixed $\Gamma>0$,
		\[
			|\Ai^{(r)}(s)|\le C_{t_0,\Gamma}e^{-\Gamma s},\qquad s\ge t_0.
		\]
		Thus, 
		\begin{equation}\label{eq:soft-local-bound}
			\left|\widehat K_n(x,y)-K_{\Ai}(x,y)-\sum_{j=1}^{m}\widehat L_j(x,y)n^{-2j/3}\right|
			\le C_{t_0,m,\Gamma}n^{-2(m+1)/3}e^{-\Gamma(x+y)},
			\qquad (x,y)\in I_n\times I_n .
		\end{equation}
		The estimates on $II_n\times II_n$, $I_n\times II_n$, and $II_n\times I_n$ follow the region decomposition of \cite[\S3.2.3--\S3.2.4]{DeiftGioev05}.
		On $II_n\times II_n$, one has $x_n,y_n\ge b+\varepsilon$. The transformations \(Y\mapsto\what T\mapsto\what S\) give
		\[
			\widehat K_n(x,y)
			=
			\frac{e^{-n\varphi(x_n)}e^{-n\varphi(y_n)}}{2\pi\ii}
			\frac{
			\begin{pmatrix}0&1\end{pmatrix}
			\what S_+(y_n)^{-1}\what S_+(x_n)
			\begin{pmatrix}1\\0\end{pmatrix}}
			{x-y}.
		\]
		For \(z>b\), the strict Euler--Lagrange inequality in Assumption~\ref{def:V} gives \(\varphi(z)>0\). Moreover,
		\[
			2\what g(z)-V(z)-\ell_V=-2\varphi(z),
			\qquad
			\what g(z)=\log z+\mathcal O(z^{-1}),
			\qquad z\to+\infty.
		\]
		Since $V$ is a polynomial of even degree with positive leading coefficient, it follows that $\varphi(z)\sim V(z)/2$ as $z\to+\infty$. Hence, for any $\varepsilon>0$,
		\[
			\inf_{z\ge b+\varepsilon}\frac{\varphi(z)}{1+z-b}>0.
		\]
		Consequently, there are constants $c_0,\gamma>0$, depending only on $V$, such that, for
		\[
			z_u:=b+\frac{u}{\kappa_V n^{2/3}},
			\qquad u\ge \varepsilon\kappa_V n^{2/3},
		\]
		and all large $n$,
		\begin{equation}\label{eq:soft-exterior-phase}
			e^{-n\varphi(z_u)}
			\le C_m n^{-2(m+1)/3}
			e^{-c_0n}e^{-\gamma u}.
		\end{equation}
		The explicit formula for $\what N$ and the small-norm estimates for $\what R$ give
		\begin{equation*}
			\sup_{z\ge b+\varepsilon}\left(\|\what S_+(z)^{\pm1}\|+\|\what S_+'(z)\|\right)\le C.
		\end{equation*}
		Since the numerator vanishes at $x=y$, the mean-value formula gives
		\begin{align}\label{eq:soft-exterior-quotient}
			\left|\frac{\begin{pmatrix}0&1\end{pmatrix}\what S_+(y_n)^{-1}\what S_+(x_n)\begin{pmatrix}1\\0\end{pmatrix}}{x-y}\right|
			&=\frac{1}{\kappa_Vn^{2/3}}\left|\int_0^1 \begin{pmatrix}0&1\end{pmatrix}\what S_+(y_n)^{-1}
			\what S_+'\!\left(y_n+\theta(x_n-y_n)\right)\begin{pmatrix}1\\0\end{pmatrix}\,\mathrm d\theta\right| \nonumber\\
			&\le C n^{-2/3}.
		\end{align}
		The Airy asymptotics give the same bound for $K_{\Ai}$ and each $\widehat L_j$. Combining these estimates gives
		\[
			|\widehat K_n(x,y)|
			+\left|K_{\Ai}(x,y)+\sum_{j=1}^{m}\widehat L_j(x,y)n^{-2j/3}\right|
			\le C_{t_0,m}n^{-2(m+1)/3}e^{-\gamma(x+y)}
		\]
		on $II_n\times II_n$.

		In the mixed regions, we first prove the estimate for $x\in I_n$ and $y\in II_n$; the case $x\in II_n$ and $y\in I_n$ is obtained by the same argument with the two variables interchanged. If
		\[
			y<\frac32\varepsilon\kappa_V n^{2/3},
		\]
		then $x_n,y_n\in D(b,3\varepsilon/2)$. Repeating the proof of \eqref{eq:soft-local-bound} on this larger disk with $\Gamma=\gamma$ and using
		\[
			e^{-c y^{3/2}}\le C_m n^{-2(m+1)/3}e^{-\gamma y},
		\]
		gives the desired bound.
		It remains to consider
		\[
			y\ge \frac32\varepsilon\kappa_V n^{2/3}.
		\]
		Since $x\in I_n$,
		\begin{equation}\label{eq:soft-mixed-denominator}
			\frac1{|x-y|}
			\le \frac{2}{\varepsilon\kappa_V}n^{-2/3}.
		\end{equation}
		For this remaining mixed subcase,
		\[
			y_n\ge b+3\varepsilon/2,
		\]
		and hence
		\[
			\what T_+(y_n)=\what S_+(y_n)=\what R_+(y_n)\what N(y_n).
		\]
		At the $x$-side, $x_n$ remains in the Airy disk. Set $\zeta_x:=n^{2/3}\what f(x_n)$. The relation between $\what T_+(x_n)$ and $\what S_+(x_n)$ is
		\[
			\what T_+(x_n)
			=
			\begin{cases}
				\what S_+(x_n)
				\begin{pmatrix}1&0\\ e^{2n\varphi_+(x_n)}&1\end{pmatrix},
				& x<0,\\[6pt]
				\what S_+(x_n),
				& x>0.
			\end{cases}
		\]
		Substituting this relation into \eqref{def:K} and \eqref{def:hatT}, and then using $\what S=\what R\what P^{(b)}$ at $x_n$, gives
		\begin{equation*}
			\widehat K_n(x,y)
			=\frac{e^{-n\varphi(y_n)}}{2\pi\ii(x-y)}
			\begin{pmatrix}0&1\end{pmatrix}\what S_+(y_n)^{-1}\what R(x_n)\what E_n^{(b)}(x_n)
			\sqrt{2\pi}
			\begin{pmatrix}\Ai(\zeta_x)\\-\ii\Ai'(\zeta_x)\end{pmatrix}.
		\end{equation*}
		Replacing $\what S_+(y_n)$ by $\what R_+(y_n)\what N(y_n)$ gives
		\begin{equation*}
			\widehat K_n(x,y)
			=\frac{e^{-n\varphi(y_n)}}{\sqrt{2\pi}\,\ii(x-y)}
			\begin{pmatrix}0&1\end{pmatrix}\what N(y_n)^{-1}\what R_+(y_n)^{-1}
			\what R(x_n)\what E_n^{(b)}(x_n)
			\begin{pmatrix}\Ai(\zeta_x)\\-\ii\Ai'(\zeta_x)\end{pmatrix}.
		\end{equation*}
		The factors $\what N^{\pm1}$ and $\what R^{\pm1}$ are uniformly bounded on the exterior ray and in the Airy disk, and \eqref{def:E} gives $\|\what E_n^{(b)}(x_n)\|\le Cn^{1/6}$. The Airy estimates and the local behavior of $\what f$ imply, after decreasing $\gamma$ if necessary,
		\[
			|\Ai(\zeta_x)|+|\Ai'(\zeta_x)|
			\le C_{t_0}e^{-\gamma x},\qquad x\in I_n.
		\]
		Combining these bounds with \eqref{eq:soft-mixed-denominator} and \eqref{eq:soft-exterior-phase}, we obtain
		\[
			|\widehat K_n(x,y)|
			\le \frac{Cn^{1/6}}{|x-y|}e^{-n\varphi(y_n)}e^{-\gamma x}
			\le C_m n^{-2(m+1)/3}e^{-\gamma(x+y)}.
		\]
		The same estimate holds for $K_{\Ai}$ and the $\widehat L_j$, hence the desired bound holds on $I_n\times II_n$ and $II_n\times I_n$.

		The preceding coefficient calculation and the estimates on $I_n\times I_n$, $II_n\times II_n$, $I_n\times II_n$, and $II_n\times I_n$ give
		\begin{equation*}
			\widehat K_n(x,y)
			= K_{\Ai}(x,y)+ \sum_{j = 1}^{m} \widehat L_j(x,y)n^{-\frac{2j}{3}} + n^{-\frac{2(m+1)}{3}}\Boh(e^{-\gamma(x+y)}),
		\end{equation*}
		uniformly for $x,y\in [t_0,\infty)$, where each $\widehat L_j(x,y)$ is of the form
		\begin{align*}
			\widehat L_j(x,y) = & \what a_{j,00}(x,y)\Ai(x)\Ai(y) + \what a_{j,01}(x,y)\Ai(x)\Ai'(y) \nonumber\\
			& +\what a_{j,10}(x,y)\Ai'(x)\Ai(y) +\what a_{j,11}(x,y)\Ai'(x)\Ai'(y),
		\end{align*}
		with $\what a_{j,\kappa \lambda}(x,y) \in \mathbb{R}[x,y]$. This proves the expansion and the asserted polynomial form in Theorem~\ref{softthm}.

	% \begin{proof}[Proof of the explicit first correction in Theorem~\ref{softthm}]
		To show \eqref{eq:L1-explicit}, we see from \(\what f(z)=\kappa_V(z-b)+d_b(z-b)^2+\Boh((z-b)^3)\) gives
		\begin{align*}
			(z-a)\frac{\what f(z)}{z-b}
			&=(b-a)\kappa_V\left[
			1+\left(\frac1{b-a}+\frac{d_b}{\kappa_V}\right)(z-b)
			+\Boh((z-b)^2)\right].
		\end{align*}
		Since \(x_n-y_n=(x-y)/(\kappa_V n^{2/3})\), this yields
		\begin{align*}
			\left((y_n-a)\frac{\what f(y_n)}{y_n-b}\right)^{-\sigma_3/4}
			\left((x_n-a)\frac{\what f(x_n)}{x_n-b}\right)^{\sigma_3/4}
			&=I+
			\frac14\left(\frac1{b-a}+\frac{d_b}{\kappa_V}\right)
			(x_n-y_n)\sigma_3
			+\Boh\!\left((x-y)n^{-4/3}\right)\\
			&=I+(x-y)\frac{\kappa_V+(b-a)d_b}{4(b-a)\kappa_V^2}
			n^{-2/3}\sigma_3
			+\Boh\!\left((x-y)n^{-4/3}\right).
		\end{align*}
		By \eqref{asy:R_expanded}, the possible \(n^{-2/3}\) contribution of the conjugated \(\what R\)-factor is
		\[
			2\ii\bigl(\what\alpha_{2,1}(x_n)-\what\alpha_{2,1}(y_n)\bigr)n^{-2/3}
			\begin{pmatrix}0&0\\1&0\end{pmatrix}.
		\]
		Since \(\what\alpha_{2,1}\) is analytic at \(b\), this term is \(\Boh((x-y)n^{-4/3})\). Hence the first RH-prefactor coefficient in \eqref{eq:ERRE_expansion} is
		\begin{equation}\label{eq:e1_explicit}
			c_b:=\frac{\kappa_V+(b-a)d_b}{4(b-a)\kappa_V^2},
			\qquad
			\what e_1(x,y)=c_b\sigma_3
			=
			c_b
			\begin{pmatrix}1&0\\0&-1\end{pmatrix}.
		\end{equation}
		For the Airy factors, \eqref{def:f} and \eqref{def:xn} give
		\begin{equation}\label{eq:first-order-airy-argument}
			n^{\frac23}\what f(x_n)=x+\frac{d_b}{\kappa_V^2}x^2 n^{-\frac23}+\Boh(n^{-\frac43}),
			\qquad
			n^{\frac23}\what f(y_n)=y+\frac{d_b}{\kappa_V^2}y^2 n^{-\frac23}+\Boh(n^{-\frac43}).
		\end{equation}
		Uniformly for $x,y$ in compact sets, Taylor expansion together with $\Ai''(z)=z\Ai(z)$ gives
		\begin{align}\label{eq:first-order-airy-expansion}
			\Ai(n^{\frac23}\what f(x_n))
			&=\Ai(x)+\frac{d_b}{\kappa_V^2}x^2\Ai'(x)n^{-\frac23}+\Boh(n^{-\frac43}), \nonumber\\
			\Ai'(n^{\frac23}\what f(x_n))
			&=\Ai'(x)+\frac{d_b}{\kappa_V^2}x^3\Ai(x)n^{-\frac23}+\Boh(n^{-\frac43}),
		\end{align}
		and the same formulas with $x$ replaced by $y$.
The contribution of the identity term in \eqref{eq:ERRE_expansion} is
		\begin{align}\label{eq:first-order-airy-quotient}
			&\frac{\Ai'(n^{\frac23}\what f(y_n))\Ai(n^{\frac23}\what f(x_n))
			-\Ai(n^{\frac23}\what f(y_n))\Ai'(n^{\frac23}\what f(x_n))}{x-y}\nonumber\\
			&\quad =K_{\Ai}(x,y)
			+\left[-\frac{d_b}{\kappa_V^2}(x^2+xy+y^2)\Ai(x)\Ai(y)
			+\frac{d_b}{\kappa_V^2}(x+y)\Ai'(x)\Ai'(y)\right]n^{-\frac23}
			+\Boh(n^{-\frac43}).
		\end{align}
		The \(n^{-2/3}\) term of the RH prefactor is \((x-y)c_b\sigma_3\),with the leading Airy factors, it gives
		\begin{multline}\label{eq:first-order-rh-contribution}
			\frac{1}{\mathrm i(x-y)}
			\begin{pmatrix}\mathrm i\Ai'(y)&\Ai(y)\end{pmatrix}
			(x-y)c_b\sigma_3
			\begin{pmatrix}\Ai(x)\\-\mathrm i\Ai'(x)\end{pmatrix}
		=c_b\left[\Ai(x)\Ai'(y)+\Ai'(x)\Ai(y)\right].
		\end{multline}
		 Equations \eqref{eq:first-order-airy-quotient} and \eqref{eq:first-order-rh-contribution} therefore give \eqref{eq:L1-explicit}.

        This completes the proof of Theorem \ref{softthm}. \qed

\subsection{Proof of Corollary~\ref{soft edge lifting}}
		By the definition of $\widehat F(n;t)$ in \eqref{eq:soft-distribution-expansion}, its Fredholm determinant representation is
		\begin{equation*}
			\widehat F(n;t)=\det(\mathbf I-\widehat{\mathbf K}_n),
		\end{equation*}
        where $\widehat{\mathbf K}_n$ is the integral operator acting on $L^2(t,\infty)$ associated with the rescaled kernel $\widehat{K}_n$.
		The remainder bound in Theorem~\ref{softthm} allows us to apply the Fredholm determinant argument from the proof of \cite[Theorem~2.1]{Bornemann24b}. Writing $\widehat{\mathbf L}_j$ for the integral operator on $L^2(t,\infty)$ with kernel $\widehat L_j$, set
		\begin{equation*}
			\mathbf E_{m,n}:=(\mathbf I-\mathbf K_{\Ai})^{-1}\sum_{j=1}^{m}n^{-2j/3}\widehat{\mathbf L}_j,
			\qquad
			D_{m,n}(t):=\det(\mathbf I-\mathbf E_{m,n}).
		\end{equation*}
		Then, uniformly for $t\in[t_0,\infty)$,
		\begin{equation*}
			\widehat F(n;t)
			=
			F_{\Ai}(t)D_{m,n}(t)
			+n^{-2(m+1)/3}\Boh(e^{-2\gamma t}).
		\end{equation*}
		Applying Proposition~\ref{prop:finite-rank-det-template} with $\eta=n^{-2/3}$
         to $D_{m,n}(t)$ yields
		\begin{equation*}
			\widehat F(n;t)
			=
			F_{\Ai}(t)\left(1+\sum_{j=1}^{m} d_j(t)n^{-2j/3}\right)
			+n^{-2(m+1)/3}\Boh(e^{-2\gamma t}),
		\end{equation*}
		Thus \eqref{eq:soft-distribution-expansion} holds with $F_{\Ai,j}(t)=F_{\Ai}(t)d_j(t)$. 
        
		To compute the first coefficient, use $c_b$ from \eqref{eq:e1_explicit} and set
		\[
			u_{jk}(t):=\left\langle(\mathbf I-\mathbf K_{\Ai})^{-1}\Ai^{(j)},\Ai^{(k)}\right\rangle_{L^2(t,\infty)}.
		\]
		In view of the equalities 
		\[
			x\Ai(x)=\Ai''(x),\qquad x\Ai'(x)=\Ai'''(x)-\Ai(x),\qquad x^2\Ai(x)=\Ai^{(4)}(x)-2\Ai'(x),
		\]
		we rewrite \eqref{eq:L1-explicit} as the operator identity
		\begin{align*}
			\widehat{\mathbf L}_1
			&=-\frac{d_b}{\kappa_V^2}\{(\Ai^{(4)}-2\Ai')\otimes\Ai+\Ai''\otimes\Ai''
			+\Ai\otimes(\Ai^{(4)}-2\Ai')\} \\
			&\quad +c_b(\Ai\otimes\Ai'+\Ai'\otimes\Ai)
			+\frac{d_b}{\kappa_V^2}\{(\Ai'''-\Ai)\otimes\Ai'
			+\Ai'\otimes(\Ai'''-\Ai)\}.
		\end{align*}
		Applying \eqref{eq:generic-d1} to this finite-rank operator and using the self-adjointness of $\mathbf K_{\Ai}$, we obtain
		\[
			d_1(t)
			=
			\frac{2d_b}{\kappa_V^2}u_{04}(t)+\frac{d_b}{\kappa_V^2}u_{22}(t)-\frac{2d_b}{\kappa_V^2}u_{13}(t)-2\left(\frac{d_b}{\kappa_V^2}+c_b\right)u_{01}(t).
		\]
		The Shinault--Tracy table for Airy resolvent quantities \cite[p.~68]{ShinaultTracy2011}, together with the linear $F$-form calculus described in \cite[Appendix~B]{Bornemann24}, gives the following formulas for $F=F_{\Ai}$:
		\begin{equation}\label{eq:soft-airy-resolvent-identities}
		\begin{aligned}
			F_{\Ai}(t)u_{01}(t)
			&=F_{\Ai}(t)u_{10}(t)
			=\frac12 F_{\Ai}''(t),\\
			F_{\Ai}(t)u_{22}(t)
			&=\frac{t^2}{5}F_{\Ai}'(t)-\frac{3}{10}F_{\Ai}''(t)+\frac{1}{20}F_{\Ai}^{(5)}(t),\\
			F_{\Ai}(t)u_{13}(t)
			&=F_{\Ai}(t)u_{31}(t)
			=-\frac{t^2}{5}F_{\Ai}'(t)+\frac{2}{15}F_{\Ai}''(t)
			+\frac{t}{6}F_{\Ai}^{(3)}(t)+\frac{1}{30}F_{\Ai}^{(5)}(t),\\
			F_{\Ai}(t)u_{04}(t)
			&=F_{\Ai}(t)u_{40}(t)
			=\frac{t^2}{5}F_{\Ai}'(t)+\frac{47}{60}F_{\Ai}''(t)
			+\frac{t}{6}F_{\Ai}^{(3)}(t)+\frac{1}{120}F_{\Ai}^{(5)}(t),
		\end{aligned}
		\end{equation}
        where we have also made use of the symmetry relation $u_{jk}=u_{kj}$.
		%Therefore the $F_{\Ai}^{(3)}$ and $F_{\Ai}^{(5)}$ terms cancel and
        Thus, 
		\begin{equation}\label{eq:soft-airy-resolvent-combination}
			F_{\Ai}(t)\bigl(2u_{04}(t)+u_{22}(t)-2u_{13}(t)\bigr)
			=t^2F_{\Ai}'(t)+F_{\Ai}''(t).
		\end{equation}
		Since $F_{\Ai,1}(t)=F_{\Ai}(t)d_1(t)$, equations \eqref{eq:soft-airy-resolvent-identities} and \eqref{eq:soft-airy-resolvent-combination} yield \eqref{eq:soft-F1-explicit}.

    This completes the proof of Corollary~\ref{soft edge lifting}. \qed

	\subsection{Proof of Theorem \ref{hardthm}}
		To analyze the rescaled kernel $\widetilde K_n$ defined in \eqref{expansion-hard} near the hard edge at the origin, introduce the scaled variables
		\begin{equation}\label{def:uvn}
			u_n:=\frac{u}{\kappa_Q n^2}, \qquad v_n:=\frac{v}{\kappa_Q n^2},
		\end{equation}
		and consider $u,v$ in a fixed compact interval $[0,s]$. Tracing back the transformations $Y\to \widetilde T\to \widetilde S\to \widetilde R$ and using the local parametrix \eqref{def:P0hard} at the origin, we obtain
		\begin{align}\label{hard-kernel-representation}
			\widetilde K_n(u,v)
			&=\frac{1}{2\pi \ii (u-v)}
			\begin{pmatrix}
				\pi \ii \,\widehat v_n J_\alpha'(\widehat v_n) & J_\alpha(\widehat v_n)
			\end{pmatrix}
			\widetilde E_n(v_n)^{-1}\widetilde R(v_n)^{-1}\widetilde R(u_n)\widetilde E_n(u_n)
			\begin{pmatrix}
				J_\alpha(\widehat u_n)\\[1ex]
				-\pi \ii \,\widehat u_n J_\alpha'(\widehat u_n)
			\end{pmatrix},
		\end{align}
		where
		\begin{equation}\label{def:uhatvhat}
			\widehat u_n:=2\bigl(-\widetilde f_n(u_n)\bigr)^{1/2}, \qquad \widehat v_n:=2\bigl(-\widetilde f_n(v_n)\bigr)^{1/2}.
		\end{equation}

		Consider the factor
		\begin{equation*}
			\widetilde E_n(v_n)^{-1}\widetilde R(v_n)^{-1}\widetilde R(u_n)\widetilde E_n(u_n).
		\end{equation*}
		The same conjugation as in the soft-edge case gives us
		\begin{align}\label{erre11}
			&\widetilde E_n(v_n)^{-1}\widetilde R(v_n)^{-1}\widetilde R(u_n)\widetilde E_n(u_n) \notag\\
			&=\left(n^{-\frac{\sigma_3}{2}}\frac{1}{\sqrt2}
			\begin{pmatrix}
				1 & \ii\\
				\ii & 1
			\end{pmatrix}
			\left(\frac{\beta}{4}\right)^{-\frac{\alpha}{2}\sigma_3}\widetilde E_n(v_n)\right)^{-1} \notag\\
			&\quad \times
			\left(n^{-\frac{\sigma_3}{2}}\frac{1}{\sqrt2}
			\begin{pmatrix}
				1 & \ii\\
				\ii & 1
			\end{pmatrix}
			\left(\frac{\beta}{4}\right)^{-\frac{\alpha}{2}\sigma_3}
			\widetilde R(v_n)
			\left(\frac{\beta}{4}\right)^{\frac{\alpha}{2}\sigma_3}
			\frac{1}{\sqrt2}
			\begin{pmatrix}
				1 & -\ii\\
				-\ii & 1
			\end{pmatrix}n^{\frac{\sigma_3}{2}}\right)^{-1} \notag\\
			&\quad \times
			\left(n^{-\frac{\sigma_3}{2}}\frac{1}{\sqrt2}
			\begin{pmatrix}
				1 & \ii\\
				\ii & 1
			\end{pmatrix}
			\left(\frac{\beta}{4}\right)^{-\frac{\alpha}{2}\sigma_3}
			\widetilde R(u_n)
			\left(\frac{\beta}{4}\right)^{\frac{\alpha}{2}\sigma_3}
			\frac{1}{\sqrt2}
			\begin{pmatrix}
				1 & -\ii\\
				-\ii & 1
			\end{pmatrix}n^{\frac{\sigma_3}{2}}\right) \notag\\
			&\quad \times
			\left(n^{-\frac{\sigma_3}{2}}\frac{1}{\sqrt2}
			\begin{pmatrix}
				1 & \ii\\
				\ii & 1
			\end{pmatrix}
			\left(\frac{\beta}{4}\right)^{-\frac{\alpha}{2}\sigma_3}\widetilde E_n(u_n)\right).
		\end{align}
		Using \eqref{defn}, \eqref{deff-hard}, and \eqref{E2}, we obtain 
		\begin{align}\label{etilde}
			&n^{-\frac{\sigma_3}{2}}\frac{1}{\sqrt2}
			\begin{pmatrix}
				1&\ii\\
				\ii&1
			\end{pmatrix}
			\left(\frac{\beta}{4}\right)^{-\frac{\alpha}{2}\sigma_3}\widetilde E_n(z) \nonumber\\
			&=(-1)^n
			\begin{pmatrix}
				0 & \ii\\
				\ii & 0
			\end{pmatrix}
			\left(\frac{z-\beta}{z}\right)^{\frac14\sigma_3}
			n^{\frac{\sigma_3}{2}}\frac{1}{\sqrt2}
			\begin{pmatrix}
				1&-\ii\\
				-\ii&1
			\end{pmatrix}
			(-\varphi^*(z))^{\frac{\alpha}{2}\sigma_3} \nonumber\\
			&\quad\times\frac{1}{\sqrt2}
			\begin{pmatrix}
				1&\ii\\
				\ii&1
			\end{pmatrix}
			n^{\frac{\sigma_3}{2}}
			\widetilde f(z)^{\frac14\sigma_3}(2\pi)^{\frac12\sigma_3}.
		\end{align}
		In \eqref{erre11}, the scalar sign $(-1)^n$ cancels between the inverse $v$-factor and the $u$-factor. At a hard-scaled point $z=x/(\kappa_Qn^2)$, the two factors $n^{\sigma_3/2}$ combine with the boundary expansions of $\wtil\gamma$, $\varphi^*$, and $\widetilde f$ retain explicit $n$-dependence.

		Turning to the $\widetilde R$-factor, Corollary~\ref{cor:hard-R-structure} gives
		\begin{align}\label{cor:R-conjugation1}
			&n^{-\frac{\sigma_3}{2}}\frac{1}{\sqrt2}
			\begin{pmatrix}
				1 & \ii\\
				\ii & 1
			\end{pmatrix}
			\left(\frac{\beta}{4}\right)^{-\frac{\alpha}{2}\sigma_3}
			\widetilde R(z)
			\left(\frac{\beta}{4}\right)^{\frac{\alpha}{2}\sigma_3}
			\frac{1}{\sqrt2}
			\begin{pmatrix}
				1 & -\ii\\
				-\ii & 1
			\end{pmatrix}n^{\frac{\sigma_3}{2}} \nonumber\\
			&= I+
			\begin{pmatrix}
				\displaystyle\sum_{m=1}^{\infty}
				\bigl(
				\widetilde\alpha_{0,m}(z)-\widetilde\alpha_{2,m}(z)
				\bigr)n^{-m} & \displaystyle\sum_{m=1}^{\infty}
				\bigl(
				\widetilde\alpha_{1,m}(z)-\ii\widetilde\alpha_{3,m}(z)
				\bigr)n^{-m+1}\\[3ex]
				\displaystyle\sum_{m=1}^{\infty}
				\bigl(
				\widetilde\alpha_{1,m}(z)+\ii\widetilde\alpha_{3,m}(z)
				\bigr)n^{-m-1} &\displaystyle\sum_{m=1}^{\infty}
				\bigl(
				\widetilde\alpha_{0,m}(z)+\widetilde\alpha_{2,m}(z)
				\bigr)n^{-m}
			\end{pmatrix}.
		\end{align}
		As in the soft-edge case, substituting $u_n=u/(\kappa_Qn^2)$ and $v_n=v/(\kappa_Qn^2)$ into \eqref{erre11}, \eqref{etilde}, and \eqref{cor:R-conjugation1} gives coefficient functions that are polynomial in $u$ and $v$ at each fixed order in $n^{-1}$. Moreover, the left-hand side of \eqref{erre11} equals to the identity matrix when $u=v$. Thus
		\begin{equation}\label{tildee}
			\widetilde E_n(v_n)^{-1}\widetilde R(v_n)^{-1}\widetilde R(u_n)\widetilde E_n(u_n)
			\sim I+(u-v)\sum_{j=1}^{\infty}\widetilde e_j(u,v)n^{-j},
		\end{equation}
		where each $\widetilde e_j(u,v)$ is a matrix whose entries are polynomials in $u$ and $v$.
		
		For the Bessel-function factors, we use the basis in Theorem~\ref{hardthm} and set
		\begin{equation*}
			\phi_0(x):=J_\alpha(\sqrt x),
			\qquad
			\phi_1(x):=\sqrt x\,J_\alpha'(\sqrt x).
		\end{equation*}
		By the local behavior of $\widetilde\xi$ at the origin and \eqref{deffn-explicit}, one has
		\begin{equation*}
			-4\widetilde f(z)=\kappa_Q z\bigl(1+\mathcal{O}(z)\bigr),
			\qquad z\to 0.
		\end{equation*}
		Therefore, for  $x\in [0,s]$,
		\begin{equation}\label{eq:hard-lambda-expansion}
			\lambda_n(x):=\widehat x_n^{\,2}=4\bigl(-\widetilde f_n(x_n)\bigr)
			=x+\sum_{j=1}^{\infty}c_{1,j}(x)n^{-j},
		\end{equation}
		where each $c_{1,j}(x)$ belongs to $x\mathbb{R}[x]$. More generally,
		\begin{equation*}
			\frac{1}{m!}\bigl(\lambda_n(x)-x\bigr)^m=\sum_{j=m}^{\infty}c_{m,j}(x)n^{-j},
			\qquad c_{m,j}(x)\in x^m\mathbb{R}[x].
		\end{equation*}
		
		Now let $y(x):=J_\alpha(\sqrt x)$. The classical Bessel differential equation gives
		\begin{equation*}
			y''(x)=-\frac{1}{x}y'(x)+\frac{\alpha^2-x}{4x^2}y(x).
		\end{equation*}
		Before refering to Proposition~\ref{lemma}, we verify that $y$ and $y'$ are linearly independent over $\mathbb C(x)$. It suffices to show that the associated Riccati equation
		\begin{equation}\label{eq:bessel-riccati}
			w'(x)+w(x)^2+\frac{1}{x}w(x)-\frac{\alpha^2-x}{4x^2}=0
		\end{equation}
		has no rational solution. Suppose to the contrary that $w\in\mathbb C(x)$ solves \eqref{eq:bessel-riccati}, and write its Laurent expansion at infinity as
		\begin{equation*}
			w(x)=\sum_{k=\ell}^{\infty}c_kx^{-k},
			\qquad c_\ell\neq 0.
		\end{equation*}
		If $\ell<0$, then $w^2$ contributes the unique highest positive power of $x$ in \eqref{eq:bessel-riccati}, which is impossible. Hence $\ell\ge 0$. If $\ell=0$, then the constant term in \eqref{eq:bessel-riccati} gives $c_0^2=0$, so in fact $\ell\ge 1$. But then
		\begin{equation*}
			w'(x)+w(x)^2+\frac{1}{x}w(x)=\mathcal O(x^{-2}),
			\qquad x\to\infty,
		\end{equation*}
		whereas
		\begin{equation*}
			-\frac{\alpha^2-x}{4x^2}=\frac{1}{4x}+\mathcal O(x^{-2}),
			\qquad x\to\infty.
		\end{equation*}
		This contradiction shows that \eqref{eq:bessel-riccati} admits no rational solution. Therefore $y$ and $y'$ are linearly independent over $\mathbb C(x)$.
		%Arguing as in Proposition~\ref{lemma}, 
        We now introduce a formal power series in $w$,
		\begin{equation*}
			y(x+w)=y'(x)P(x;w)+y(x)Q(x;w),
		\end{equation*}
		with $P(x;0)=0$, $Q(x;0)=1$, $\partial_wP(x;0)=1$, and $\partial_wQ(x;0)=0$. Since the coefficient of $y'$ in the differential equation is $-1/x$, Proposition~\ref{lemma} yields the conservation law
		\begin{equation*}
			(x+w)\left(P(x;w)\partial_wQ(x;w)-Q(x;w)\partial_wP(x;w)\right)=-x.
		\end{equation*}
		Define
		\begin{equation*}
			A(x;w):=
			\begin{pmatrix}
				Q(x;w) & \dfrac{1}{2x}P(x;w)\\[2ex]
				2(x+w)\partial_wQ(x;w) & \dfrac{x+w}{x}\partial_wP(x;w)
			\end{pmatrix}.
		\end{equation*}
		Then, using $y'(x)=\phi_1(x)/(2x)$ and $\phi_1(x+w)=2(x+w)\partial_wy(x+w)$, we obtain
		\begin{equation*}
			\begin{pmatrix}
				\phi_0(x+w)\\[1ex]
				\phi_1(x+w)
			\end{pmatrix}
			=
			A(x;w)
			\begin{pmatrix}
				\phi_0(x)\\[1ex]
				\phi_1(x)
			\end{pmatrix}.
		\end{equation*}
		Moreover, the conservation law gives
		\begin{equation*}
			\det A(x;w)=\frac{x+w}{x}\Bigl(Q(x;w)\partial_wP(x;w)-P(x;w)\partial_wQ(x;w)\Bigr)=1.
		\end{equation*}
		As a formal power series in $w$, one has
		\begin{align*}
			P(x;w)&=w+\mathcal{O}(w^2), \qquad
			\partial_wP(x;w)=1+\mathcal{O}(w), \nonumber\\
			Q(x;w)&=1+\mathcal{O}(w), \qquad
			\partial_wQ(x;w)=\mathcal{O}(w).
		\end{align*}
		Let $w=\lambda_n(x)-x$ and using $(\lambda_n(x)-x)^m\in x^m\mathbb{R}[x]$, we see term by term that
		\[
			\frac{P(x;\lambda_n(x)-x)}{x}, \qquad
			\frac{\lambda_n(x)}{x}\partial_wP(x;\lambda_n(x)-x), \qquad \lambda_n(x)\partial_wQ(x;\lambda_n(x)-x)
		\]
		are all regular at $x=0$. Hence
		\begin{equation*}
			A_n(x):=A\bigl(x;\lambda_n(x)-x\bigr)
			=I+\sum_{j=1}^{\infty}A^{(j)}(x)n^{-j},
		\end{equation*}
		where each $A^{(j)}$ is a $2\times 2$ matrix with entries in $\mathbb{R}[x]$, and
		\begin{equation*}
			\begin{pmatrix}
				J_\alpha(\widehat x_n)\\[1ex]
				\widehat x_nJ_\alpha'(\widehat x_n)
			\end{pmatrix}
			=
			A_n(x)
			\begin{pmatrix}
				\phi_0(x)\\[1ex]
				\phi_1(x)
			\end{pmatrix},
			\qquad
			\det A_n(x)=1.
		\end{equation*}
		
	Let
		\begin{equation*}
			\Omega:=\begin{pmatrix}
				0 & -1\\
				1 & 0
			\end{pmatrix},
			\qquad
			\Phi(x):=
			\begin{pmatrix}
				\phi_0(x)\\[1ex]
				\phi_1(x)
			\end{pmatrix}.
		\end{equation*}
		Since $\det A_n(x)=1$, one has $A_n(x)^T\Omega A_n(x)=\Omega$. Therefore, for each $j\ge 1$, the coefficient of $n^{-j}$ in $A_n(v)^T\Omega A_n(u)-\Omega$ is a polynomial in $u$ and $v$ divisible by $u-v$. 
		
		% The identity-matrix contribution in \eqref{tildee} is
		% \begin{equation}
		% 	\frac{1}{2(u-v)}\Phi(v)^TA_n(v)^T\Omega A_n(u)\Phi(u).
		% \end{equation}
		% Its leading term is $K_{\Bes}(u,v)$, while the divisibility above cancels the denominator in every higher-order coefficient. The second term in \eqref{tildee} is regular at $u=v$ because it already contains the factor $u-v$.
		% The leading term is
		% \begin{equation}
		% 	\frac{1}{2(u-v)}\Phi(v)^T\Omega\Phi(u)=K_{\Bes}(u,v).
		% \end{equation}
		
		% It remains to consider the second term in \eqref{tildee}. The factor $u-v$ cancels the denominator in \eqref{hard-kernel-representation}. After this cancellation, each coefficient is obtained from products of the matrices $\widetilde e_j(u,v)$, $A^{(\ell)}(u)$, and $A^{(r)}(v)$. These matrices have polynomial entries after evaluation at the corresponding variables. Hence, together with the identity contribution above, the coefficient of each power $n^{-j}$ in $\widetilde K_n(u,v)$ lies in the $\mathbb R[u,v]$-span of
		% \begin{equation}
		% 	J_\alpha(\sqrt u)J_\alpha(\sqrt v), \quad J_\alpha(\sqrt u)\sqrt v\,J_\alpha'(\sqrt v), \quad \sqrt u\,J_\alpha'(\sqrt u)J_\alpha(\sqrt v), \quad \sqrt u\,J_\alpha'(\sqrt u)\sqrt v\,J_\alpha'(\sqrt v).
		% \end{equation}

		Combining \eqref{hard-kernel-representation} and \eqref{tildee} with the preceding divisibility yields
		\begin{equation*}
			\widetilde K_n(u,v)=K_{\Bes}(u,v)+\sum_{j=1}^{m}\widetilde L_j(u,v)n^{-j}+\mathcal{O}(n^{-m-1}),
		\end{equation*}
		for each fixed $u,v\in (0,s]$. The weighted basis functions
		\[
			u^{-\alpha/2}J_\alpha(\sqrt u),\qquad
			u^{-\alpha/2}\sqrt u\,J_\alpha'(\sqrt u)
		\]
		are bounded on $[0,s]$. These bounds also apply to the finite Taylor remainders on $0,s^2$, so multiplication by $(uv)^{-\alpha/2}$ yields \eqref{expansion-hard-weighted}.

	% \begin{proof}[Proof of the explicit first correction in Theorem~\ref{hardthm}]
    To prove \eqref{eq:hard-L1-explicit}, first note that $x_n=x/(\kappa_Qn^2)$ and $\widetilde f_n=n^2\widetilde f$. Together with \eqref{eq:hard-lambda-expansion}, the expansion of $\widetilde f$ at the origin gives
		\begin{equation}\label{eq:hard-lambda-no-n1}
			\lambda_n(x)=x+\mathcal O(n^{-2}),
		\end{equation}
		uniformly for $x$ in compact subsets of $[0,\infty)$. Thus $A_n(x)=I+\mathcal O(n^{-2})$.

		For $x>0$, take the upper or lower boundary values consistently. From \eqref{defn}, \eqref{def:varphi-hard} and \eqref{deffn-explicit},
		\begin{align}\label{eq:hard-prefactor-scalars}
			\wtil\gamma_\pm(x_n)
			&=e^{\pm\pi\ii/4}\sqrt{\beta h_Q(0)}\,n^{1/2}x^{-1/4}
			\bigl(1+\mathcal O(n^{-2})\bigr),\nonumber\\
			\bigl(-\varphi^*_\pm(x_n)\bigr)^{\alpha/2}
			&=1\mp\frac{\alpha\ii\sqrt{x}}{\beta h_Q(0)}\,n^{-1}
			+\mathcal O(n^{-2}),\nonumber\\
			\widetilde f_{n,\pm}(x_n)^{1/4}
			&=e^{\mp\pi\ii/4}\frac{x^{1/4}}{\sqrt2}
			\bigl(1+\mathcal O(n^{-2})\bigr).
		\end{align}
		Substituting the above formulas into \eqref{etilde}, we have 
		\begin{equation}\label{eq:hard-prefactor-one-point}
			\widetilde E_n(0)^{-1}\widetilde E_n(x_n)
			=I+\frac{\alpha\pi\ii}{\beta h_Q(0)}x
			\begin{pmatrix}
				0&0\\
				1&0
			\end{pmatrix}n^{-1}
			+\mathcal O(xn^{-2}).
		\end{equation}
		The same expression is obtained from either boundary value in \eqref{eq:hard-prefactor-scalars}. Consequently,
		\begin{equation}\label{eq:hard-prefactor-two-point}
			\widetilde E_n(v_n)^{-1}\widetilde E_n(u_n)
			=I+\frac{\alpha\pi\ii}{\beta h_Q(0)}(u-v)
			\begin{pmatrix}
				0&0\\
				1&0
			\end{pmatrix}n^{-1}
			+\mathcal O\bigl((u-v)n^{-2}\bigr).
		\end{equation}

		The RH problem \ref{rhp57} and the small-norm estimate imply that $\widetilde R$ is analytic at the origin and satisfies $\widetilde R(z)=I+\mathcal O(n^{-1})$ uniformly in a fixed disk there. Cauchy's estimate therefore gives
		\begin{equation*}
			\widetilde R(v_n)^{-1}\widetilde R(u_n)-I
			=\mathcal O\bigl((u-v)n^{-3}\bigr).
		\end{equation*}
		Moreover, \eqref{E2} and \eqref{eq:hard-prefactor-scalars} give us $\widetilde E_n(x_n)^{\pm1}=\mathcal O(n^{1/2})$. Combining these estimates with \eqref{eq:hard-prefactor-two-point}, we obtain
		\begin{equation}\label{eq:hard-total-first-product}
			\widetilde E_n(v_n)^{-1}\widetilde R(v_n)^{-1}\widetilde R(u_n)\widetilde E_n(u_n)
			=I+\frac{\alpha\pi\ii}{\beta h_Q(0)}(u-v)
			\begin{pmatrix}
				0&0\\
				1&0
			\end{pmatrix}
			n^{-1}
			+\mathcal O\bigl((u-v)n^{-2}\bigr).
		\end{equation}
		% The estimate is uniform for $u,v$ in compact subsets of $[0,\infty)$.

		Finally, substituting \eqref{eq:hard-total-first-product} into \eqref{hard-kernel-representation}, the first correction reads
		\[
			\frac{1}{2\pi\ii(u-v)}
			\begin{pmatrix}\pi\ii\phi_1(v)&\phi_0(v)\end{pmatrix}
			\frac{\alpha\pi\ii}{\beta h_Q(0)}(u-v)
			\begin{pmatrix}0&0\\1&0\end{pmatrix}
			\begin{pmatrix}\phi_0(u)\\-\pi\ii\phi_1(u)\end{pmatrix}
			=
			\frac{\alpha}{2\beta h_Q(0)}\phi_0(u)\phi_0(v).
		\]
		This proves \eqref{eq:hard-L1-explicit} and completes the proof of Theorem~\ref{hardthm}. \qed
	%\end{proof}

	% \begin{remark}
	% 	For the classical Laguerre case $Q(x)=x$, one has $\beta=4$, $h_Q(0)=1$, and therefore the coefficient in \eqref{eq:hard-L1-explicit} reduces to $\alpha/8$. For the quadratic hard-edge potential $Q(x)=x^2$, the equilibrium density is obtained from
	% 	\[
	% 		W(z)=z-\frac12(2z+\beta)\sqrt{\frac{z-\beta}{z}}
	% 		=\frac1z+\mathcal O(z^{-2}),\qquad z\to\infty,
	% 	\]
	% 	which gives $\beta=2\sqrt6/3$ and $h_Q(x)=2x+\beta$. Hence $h_Q(0)=\beta$ and the first coefficient is $\alpha/(2\beta^2)=3\alpha/16$.
	% \end{remark}

\subsection{Proof of Corollary~\ref{hard edge lifting}}
By the definition of $\widetilde F(n;t)$ in \eqref{eq:hard-distribution-expansion}, its Fredholm determinant representation is
		\begin{equation*}
			\widetilde F(n;t)=\det(\mathbf I-\widetilde{\mathbf K}_n),
		\end{equation*}
        where $\widetilde{\mathbf K}_n$ is the integral operator acting on $L^2(0,t)$ associated with the rescaled kernel $\widetilde{K}_n$.
		% For the determinantal point process with rescaled kernel $\widetilde K_n$, the hard-edge gap probability is
		% \begin{equation*}
		% 	\widetilde F(n;t)=\det(\mathbf I-\widetilde{\mathbf K}_n).
		% \end{equation*}
		Fix $s>0$ and let $\nu_\alpha(\ud x):=x^\alpha \ud x$. For $0<t\le s$, define the unitary map
		\begin{equation*}
			U:L^2(0,t)\to L^2((0,t),\nu_\alpha),
			\qquad
			(Uf)(x)=x^{-\alpha/2}f(x).
		\end{equation*}
		With respect to $\nu_\alpha$, the operator $U\widetilde{\mathbf K}_nU^{-1}$ has kernel
		\begin{equation*}
			\mathcal K_n^{(\alpha)}(u,v):=(uv)^{-\alpha/2}\widetilde K_n(u,v).
		\end{equation*}
		Consequently,
		\begin{equation*}
			\det(\mathbf I-\widetilde{\mathbf K}_n)
			=
			\det\left.(\mathbf I-\bm{\mathcal K}_n^{(\alpha)})\right|_{L^2((0,t),\nu_\alpha)}.
		\end{equation*}
        
		Define the truncated conjugated kernel by
		\begin{equation*}
			\mathcal K_{n,m}^{(\alpha)}(u,v):=(uv)^{-\alpha/2}\left(K_{\Bes}(u,v)+\sum_{j=1}^{m}\widetilde L_j(u,v)n^{-j}\right).
		\end{equation*}
		The kernels $\mathcal K_n^{(\alpha)}$ and $\mathcal K_{n,m}^{(\alpha)}$ extend continuously to $[0,s]^2$ and are uniformly bounded there. Since $\nu_\alpha((0,t))\leq\nu_\alpha((0,s))<\infty$, the Lipschitz constant in the finite-measure Fredholm determinant estimate is independent of $t$; see \cite[Section~2, paragraph following (4)]{Bornemann16} and \cite[Lemma~3.4.5]{AGZ10}. Hence \eqref{expansion-hard-weighted} gives, uniformly for $t\in[0,s]$,
		\begin{equation*}
			\det(\mathbf I-\widetilde{\mathbf K}_n)
			=
			\det\left.(\mathbf I-\bm{\mathcal K}_{n,m}^{(\alpha)})\right|_{L^2((0,t),\nu_\alpha)}
			+\mathcal{O}(n^{-m-1}).
		\end{equation*}
		% The resolvent of the limiting Bessel operator depends continuously on $t$ and, since $F_{\Bes}(t)>0$, is uniformly bounded for $t\in[0,s]$. The unitary conjugation preserves the finite-rank correction structure and the associated coefficient functions. 
		Applying Proposition~\ref{prop:finite-rank-det-template} with $\eta=n^{-1}$ to the weighted operators gives
		\begin{equation*}
			\widetilde F(n;t)
			=
			F_{\Bes}(t)\left(1+\sum_{j=1}^{m} d_j(t)n^{-j}\right)+\mathcal{O}(n^{-m-1}).
		\end{equation*}
		This expansion is uniform for $t\in[0,s]$. 
		% The functions $d_j$ are smooth for $t>0$; at $t=0$, the correction operators vanish, so they extend continuously with $d_j(0)=0$. Setting $F_{\Bes,j}(t):=F_{\Bes}(t)d_j(t)$ gives the expansion stated in the corollary.

		To compute the first coefficient, we obtain by applying \eqref{eq:generic-d1} to the rank-one correction \eqref{eq:hard-L1-explicit} that
		\begin{equation}\label{eq:hard-d1-v00-proof}
			d_1(t)
			=
			-\frac{\alpha}{2\beta h_Q(0)}v_{00}(t),
			\qquad
			v_{00}(t):=\left\langle(\mathbf I-\mathbf K_{\Bes})^{-1}\phi_0,\phi_0\right\rangle_{L^2(0,t)}.
		\end{equation}
		The Bessel endpoint identity \cite[(6)]{Bornemann16} (see also \cite[(1.5), (2.5), (2.21)]{TracyWidom94Bessel}) gives
		\begin{equation}\label{eq:v00-logderivative}
			\frac{\operatorname{d}\!}{\operatorname{d}\!t}\log F_{\Bes}(t)
			=
			-\frac{1}{4t}v_{00}(t).
		\end{equation}
		Since $F_{\Bes,1}(t)=F_{\Bes}(t)d_1(t)$, combining \eqref{eq:hard-d1-v00-proof} and \eqref{eq:v00-logderivative} gives \eqref{eq:hard-d1-explicit}. 
        
        This completes the proof of Corollary~\ref{hard edge lifting}.
        \qed

	\section*{Acknowledgements}
	We thank Folkmar Bornemann, Sung-Soo Byun, Peter Forrester and Ben Liu for helpful suggestions and discussions.  This work is partially supported by NSFC under grant numbers 12271105 and 11822104,
	
	\appendix
	\section{The Airy parametrix}\label{airy}
	The Airy parametrix $\Phi^{({\Ai})}$ is the unique solution of the following RH problem.
	\begin{rhp}
		\hfill
		\begin{itemize}
			\item[\rm(a)] $\Phi^{(\mathrm{Ai})}(z)$ is analytic in $\mathbb{C} \setminus \left(\bigcup_{j=1}^4 \Sigma_j \cup \{0\}\right)$, where the contours $\Sigma_j$, $j=1, 2, 3, 4$, are indicated in Figure \ref{fig:Airy}.
			\item[\rm(b)] For $z \in \cup_{j=1}^{4} \Sigma_j$, we have
			\begin{align}\label{jump:Airy}
				\Phi^{({\Ai})}_+(z)=\Phi^{({\Ai})}_-(z)\begin{cases}
					\begin{pmatrix}
						1 & 1\\
						0 & 1
					\end{pmatrix}, & \qquad z \in \Sigma_1,\\
					\begin{pmatrix}
						1 & 0\\
						1& 1
					\end{pmatrix}, & \qquad z \in \Sigma_2 \cup \Sigma_4,\\
					\begin{pmatrix}
						0 & 1\\
						-1 & 0
					\end{pmatrix}, & \qquad z \in \Sigma_3.
				\end{cases}
			\end{align}
			\item[\rm(c)] As $z \to \infty$, we have
			\begin{align}\label{infty:Ai}
				\Phi^{({\Ai})}(z) = \frac{1}{\sqrt{2}} \begin{pmatrix}
					z^{-\frac 14} & 0\\
					0 & z^{\frac 14}
				\end{pmatrix} \begin{pmatrix}
					1 & \ii\\
					\ii & 1
				\end{pmatrix}\left(I + \Boh(z^{-\frac 32})\right)e^{-\frac 23 z^{3/2} \sigma_3}.
			\end{align}
			\item[\rm(d)] $\Phi^{({\Ai})}(z)$ is bounded near the origin.
		\end{itemize}
	\end{rhp}
	
	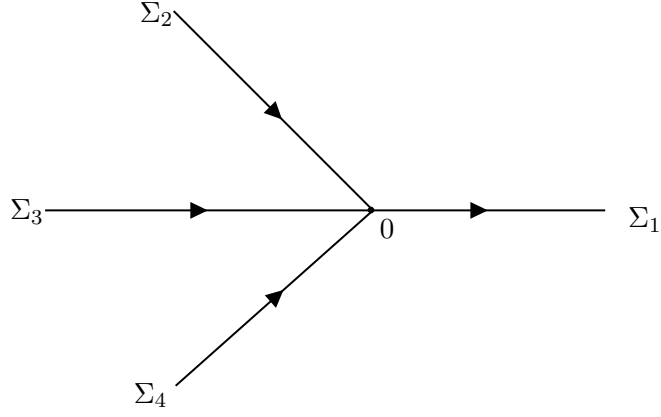
\begin{figure}[t]
		\begin{center}
			
			\tikzset{every picture/.style={line width=0.75pt}}
			
			\begin{tikzpicture}[x=0.75pt,y=0.75pt,yscale=-1,xscale=1]
				\draw    (52,120) -- (203,120) -- (333,120) ;
				
				\draw [shift={(132.5,120)}, rotate = 180] [fill={rgb, 255:red, 0; green, 0; blue, 0 }] (7.14,-3.43) -- (0,0) -- (7.14,3.43) -- cycle ;  
				
				\draw [shift={(273,120)}, rotate = 180] [fill={rgb, 255:red, 0; green, 0; blue, 0 }  ](7.14,-3.43) -- (0,0) -- (7.14,3.43) -- cycle    ;
				
				\draw    (116.5,20) -- (216.5,120) ;
				
				\draw [shift={(170.04,73.54)}, rotate = 225] [fill={rgb, 255:red, 0; green, 0; blue, 0 }  ](7.14,-3.43) -- (0,0) -- (7.14,3.43) -- cycle    ;
				
				\draw    (117.5,208) -- (216.5,120) ;
				
				\draw [shift={(170.74,160.68)}, rotate = 138.37] [fill={rgb, 255:red, 0; green, 0; blue, 0 }  ](7.14,-3.43) -- (0,0) -- (7.14,3.43) -- cycle    ;
				
				\draw  [fill={rgb, 255:red, 0; green, 0; blue, 0 }  ,fill opacity=1 ] (216.67,119.83) .. controls (216.67,119.19) and (216.15,118.66) .. (215.5,118.66) .. controls (214.85,118.66) and (214.33,119.19) .. (214.33,119.83) .. controls (214.33,120.48) and (214.85,121) .. (215.5,121) .. controls (216.15,121) and (216.67,120.48) .. (216.67,119.83) -- cycle ;
				
				\draw (218.5,123) node [anchor=north west][inner sep=0.75pt]   [align=left] {$0$};
				\draw (343,117) node [anchor=north west][inner sep=0.75pt]   [align=left] {$\Sigma_1$};
				\draw (98,14) node [anchor=north west][inner sep=0.75pt]   [align=left] {$\Sigma_2$};
				\draw (33,113) node [anchor=north west][inner sep=0.75pt]   [align=left] {$\Sigma_3$};
				\draw (95,205) node [anchor=north west][inner sep=0.75pt]   [align=left] {$\Sigma_4$};

			\end{tikzpicture}
			\caption{The jump contours of the RH problem for $\Phi^{({\Ai})}$.}
			\label{fig:Airy}
		\end{center}
	\end{figure}
	
	Let $\omega:= e^{2 \pi \ii/3}$. The unique solution is given by (cf. \cite{Deift99})
\begin{align}\label{eq:airy-parametrix-sector-form}
		\Phi^{({\Ai})}(z) = \sqrt{2 \pi} \begin{cases}
			\begin{pmatrix}
				\Ai (z) & - \omega^2 \Ai (\omega^2 z)\\
				-\ii \Ai'(z) & \ii \omega \Ai' (\omega^2 z)
			\end{pmatrix}, & \arg z \in \left(0, \frac{3 \pi}{4} \right),\\
			\begin{pmatrix}
				-\omega \Ai (\omega z) & - \omega^2 \Ai (\omega^2 z)\\
				\ii \omega^2 \Ai'(\omega z) & \ii \omega \Ai' (\omega^2 z)
			\end{pmatrix}, & \arg z \in \left(\frac{3 \pi}{4}, \pi \right),\\
			\begin{pmatrix}
				-\omega^2 \Ai (\omega^2 z) & \omega \Ai (\omega z)\\
				\ii \omega \Ai'(\omega^2 z) & -\ii \omega^2 \Ai' (\omega z)
			\end{pmatrix}, & \arg z \in \left(-\pi, -\frac{3 \pi}{4} \right),\\
			\begin{pmatrix}
				\Ai (z) &  \omega \Ai (\omega z)\\
				-\ii \Ai'(z) & -\ii \omega^2 \Ai' (\omega z)
			\end{pmatrix}, & \arg z \in \left(-\frac{3 \pi}{4}, 0\right).
		\end{cases}
	\end{align}
	
	\section{The Bessel parametrix}
	\label{bessel}
	The Bessel parametrix is the unique solution of the following RH problem.
	
	\begin{rhp}\label{rhp:Bessel}
		\hfill
		\begin{itemize}
			\item [\rm(a)] $\Phi_{\alpha}^{(\Bes)}(z)$ is defined and analytic in $\mathbb{C} \setminus \left(\bigcup_{j=1}^3 \Gamma_j \cup \{0\}\right)$, where the contours $\Gamma_j$, $j=1,2,3$, are shown in Figure \ref{figure-Bessel}.
			\item [\rm(b)] For $z \in \bigcup_{j=1}^3 \Gamma_j$, we have
			\begin{equation}\label{eq:jump:Bessel}
				\Phi_{\alpha,+}^{(\Bes)}(z) = \Phi_{\alpha, -}^{(\Bes)}(z) \begin{cases}
					\begin{pmatrix}
						1 & 0\\
						e^{\alpha \pi \ii} & 1
					\end{pmatrix}, \quad & z \in \Gamma_1,\\
					\begin{pmatrix}
						0 &1\\
						-1 & 0
					\end{pmatrix}, \quad & z \in \Gamma_2,\\
					\begin{pmatrix}
						1 & 0\\
						e^{-\alpha \pi \ii} & 1
					\end{pmatrix}, \quad & z \in \Gamma_3,
				\end{cases}
			\end{equation}
			where $\alpha>-1$. 
			\item [\rm(c)] As $z \to \infty$, we have
			\begin{multline}\label{eq:infty:Bessel}
				\Phi_{\alpha}^{(\Bes)}(z)
				= \frac{(4\pi^2 z)^{-\sigma_3/4}}{\sqrt{2}}
				\begin{pmatrix}
					1 & -\ii\\
					-\ii & 1
				\end{pmatrix} \\
				\times
				\left(I + \frac{1}{16z^{1/2}}
				\begin{pmatrix}
					-1-4\alpha^2 & 2 \ii\\
					2 \ii & 1+4\alpha^2
				\end{pmatrix} + \Boh\left(\frac{1}{z}\right) \right)e^{2z^{1/2}\sigma_3}.
			\end{multline}
		\end{itemize}
	\end{rhp}
		By \cite[Eq.~(3.81)]{Vanlessen07}, we have
\begin{equation}\label{eq:bessel-parametrix-sector-form}
	 \Phi_\alpha^{(\Bes)}(z)= \begin{cases}
    \begin{pmatrix} I_\alpha(2z^{\frac{1}{2}})&  -\frac{\ii}{\pi} K_\alpha(2z^{\frac{1}{2}})\\-2\pi \ii z^{\frac{1}{2}}I_\alpha'(2z^{\frac{1}{2}})&  -2z^{\frac{1}{2}}K_\alpha'(2z^{\frac{1}{2}})\end{pmatrix}, & z \in \mathrm{I}, \\
    \begin{pmatrix} \frac{1}{2}H_\alpha^{(1)}(2(-z)^{\frac{1}{2}})&  -\frac{1}{2} H_\alpha^{(2)}(2(-z)^{\frac{1}{2}})\\-\pi z^{\frac{1}{2}}(H_\alpha^{(1)})'(2(-z)^{\frac{1}{2}})&  \pi z^{\frac{1}{2}}(H_\alpha^{(2)})'(2(-z)^{\frac{1}{2}})\end{pmatrix} e^{\frac{\alpha \pi \ii}{2} \sigma_3}, & z \in \mathrm{II}, \\
    \begin{pmatrix} \frac{1}{2}H_\alpha^{(2)}(2(-z)^{\frac{1}{2}})&  \frac{1}{2} H_\alpha^{(1)}(2(-z)^{\frac{1}{2}})\\ \pi z^{\frac{1}{2}}(H_\alpha^{(2)})'(2(-z)^{\frac{1}{2}})&  \pi z^{\frac{1}{2}}(H_\alpha^{(1)})'(2(-z)^{\frac{1}{2}})\end{pmatrix}e^{-\frac{\alpha \pi \ii}{2} \sigma_3}, & z \in \mathrm{III},
  \end{cases}
\end{equation}
% where the domains $\mathrm{I}$--$\mathrm{III}$ are illustrated in Figure~\ref{figure-Bessel}.
where the domains I--III are illustrated in Figure \ref{figure-Bessel},
$I_\alpha$ and $K_\alpha$ are the modified Bessel functions, and $H_\alpha^{(1)}$ and
$H_\alpha^{(2)}$ are the Hankel functions of the first and second kind, respectively (cf. \cite{DLMF} for the definitions and properties of these functions) and the  principal branch is taken for $z^{1/2}$.
	\begin{figure}[t]
    \centering
    
    \tikzset{every picture/.style={line width=0.75pt}}
    
    \begin{tikzpicture}[x=0.75pt,y=0.75pt,yscale=-1,xscale=1]
        \draw (100,130) -- (269.5,131) ;
        \draw [shift={(188.55,130.52)}, rotate = 180.34] [fill={rgb, 255:red, 0; green, 0; blue, 0 }] (-7.14,-3.43) -- (0,0) -- (-7.14,3.43) -- cycle ; 
        
        \draw (169.12,29.61) -- (269.5,131) ;
        \draw [shift={(221.98,83.01)}, rotate = 225.29] [fill={rgb, 255:red, 0; green, 0; blue, 0 }] (-7.14,-3.43) -- (0,0) -- (-7.14,3.43) -- cycle ; 
        
        \draw (269.5,131) -- (171.12,231.61) ;
        \draw [shift={(224.01,177.52)}, rotate = 134.36] [fill={rgb, 255:red, 0; green, 0; blue, 0 }] (-7.14,-3.43) -- (0,0) -- (-7.14,3.43) -- cycle ; 
        
        \draw [fill={rgb, 255:red, 0; green, 0; blue, 0 } ,fill opacity=1 ][line width=0.75]  (267.5,131.5) .. controls (267.5,130.67) and (268.17,130) .. (269,130) .. controls (269.83,130) and (270.5,130.67) .. (270.5,131.5) .. controls (270.5,132.33) and (269.83,133) .. (269,133) .. controls (268.17,133) and (267.5,132.33) .. (267.5,131.5) -- cycle ;
        
        \draw (281,80) node [anchor=north west][inner sep=0.75pt]   [align=left] {I};
        \draw (271.5,133) node [anchor=north west][inner sep=0.75pt]   [align=left] {0};
        \draw (145,78) node [anchor=north west][inner sep=0.75pt]   [align=left] {II};
        \draw (144,156) node [anchor=north west][inner sep=0.75pt]   [align=left] {III};
        \draw (175,13) node [anchor=north west][inner sep=0.75pt]   [align=left] {$\Gamma_1$};
        \draw (80,122) node [anchor=north west][inner sep=0.75pt]   [align=left] {$\Gamma_2$};
        \draw (181,232) node [anchor=north west][inner sep=0.75pt]   [align=left] {$\Gamma_3$};
        
    \end{tikzpicture}
    \caption{The jump contours $\Gamma_j$, $j=1,2,3$, and the domains I--III in the RH problem for $\Phi_{\alpha}^{(\text{Bes})}$.}
    \label{figure-Bessel}
\end{figure}
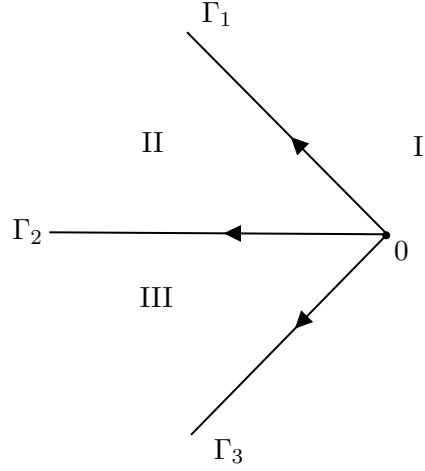

\end{document}